\documentclass[preprint,11pt,3p]{elsarticle}

\usepackage{times}
\usepackage{color}
\usepackage[colorlinks=true]{hyperref}

\usepackage{amsmath,amssymb,amsthm,amscd}
\numberwithin{equation}{section}

\newtheorem{theorem}{Theorem}[section]

\newtheorem{proposition}{Proposition}[section]
\newtheorem{lemma}{Lemma}[section]
\newtheorem{definition}{Definition}[section]

\newtheorem{remark}{Remark}[section]
\newtheorem{example}{Example}[section]
\biboptions{sort&compress}

\journal{Elsevier}

\begin{document}
\begin{frontmatter}

\title{Multiple sign-changing  solutions for semilinear subelliptic Dirichlet problem}
\author[label1]{Hua Chen\corref{cor1}}
\ead{chenhua@whu.edu.cn}
\author[label2]{Hong-Ge Chen}
\ead{hongge\_chen@whu.edu.cn}
\author[label1]{Jin-Ning Li}
\ead{lijinning@whu.edu.cn}
\author[label1]{Xin Liao}
\ead{xin_liao@whu.edu.cn}

\address[label1]{School of Mathematics and Statistics, Wuhan University, Wuhan 430072, China}
\address[label2]{Wuhan Institute of Physics and Mathematics, Innovation Academy for Precision Measurement Science \\
 and Technology, Chinese Academy of Sciences, Wuhan 430071, China}

\cortext[cor1]{corresponding author}
\begin{abstract}
We study the following perturbation from symmetry problem for the semilinear subelliptic equation
\[ \left\{
      \begin{array}{cc}
      -\triangle_{X} u=f(x,u)+g(x,u) & \mbox{in}~\Omega, \\[2mm]
      u\in H_{X,0}^{1}(\Omega),\hfill
      \end{array}
 \right.  \]
where $\triangle_{X}=-\sum_{i=1}^{m}X_{i}^{*}X_{i}$ is the self-adjoint sub-elliptic operator associated with H\"{o}rmander vector fields $X=(X_{1},X_{2},\ldots,X_{m})$, $\Omega$ is an open bounded subset in $\mathbb{R}^n$, and $H_{X,0}^{1}(\Omega)$ denotes the weighted Sobolev space. We establish multiplicity results for sign-changing solutions using a perturbation method alongside refined techniques for invariant sets.  The pivotal aspect lies in the estimation of the lower bounds of min-max values associated with sign-changing critical points.  In this paper, we construct two distinct lower bounds of these min-max values. The first one is derived from the lower bound of Dirichlet eigenvalues of $-\triangle_{X}$, while the second one is based on the Morse-type estimates and  Cwikel-Lieb-Rozenblum type inequality in degenerate cases. These lower bounds provide different sufficient conditions for multiplicity results, each with unique advantages and are not mutually inclusive, particularly in the general non-equiregular case. This novel observation suggests that in some sense, the situation for sub-elliptic equations would have essential difference from the classical elliptic framework.

\end{abstract}
\end{frontmatter}
\section{Introduction and main results}
Let $X=(X_{1},X_{2},\ldots,X_{m})$ be a system of real smooth vector fields defined on an open domain $W\subset \mathbb{R}^{n}~(n\geq 2)$, satisfying the following condition:
\[ {\rm dim~ Lie}\{X_{1},X_{2},\ldots,X_{m}\}(x)=n, \qquad \forall x\in W. \leqno(H) \]
Condition (H) is known as the H\"{o}rmander's condition (cf. \cite{hormander1967}), and the vector fields $X$ satisfying condition (H) are usually referred to as H\"{o}rmander vector fields. We then denote by $ \triangle_{X}:=-\sum_{i=1}^{m}X_{i}^{*}X_{i}$ the H\"ormander operator associated with $X$, where $X_{i}^{*}=-X_{i}-{\rm div}(X)$ is the formal adjoint of $X_{i}$.

In this paper, we study the existence of sign-changing solutions to the semilinear subelliptic Dirichlet problem
\begin{equation}\label{problem1-1}
\left\{
      \begin{array}{cc}
      -\triangle_{X} u=f(x,u)+g(x,u) & \mbox{in}~\Omega, \\[2mm]
      u\in H_{X,0}^{1}(\Omega). \hfill           &
      \end{array}
 \right.
\end{equation}
Here, $\Omega\subset\subset W$ is a bounded open subset, $H_{X,0}^{1}(\Omega)$ denotes the Sobolev space associated with vector fields $X$ (precise definition provided in Section \ref{Section2} below),
   and $f\in C(\overline{\Omega}\times\mathbb{R})$ is a function satisfying the following assumptions:
 \begin{enumerate}[($f$1)]
 \item  $\lim_{u\to 0}{\frac{f(x,u)}{u}}= 0$ uniformly in $x\in\overline{\Omega}$.
            \item There exists $2<p<2_{\tilde{\nu}}^{*}:=\frac{2\tilde{\nu}}{\tilde{\nu}-2}$ and  $C>0$ such that
         \[ |f(x,u)|\leq C(1+|u|^{p-1})\qquad\forall (x,u)\in \overline{\Omega}\times \mathbb{R}, \]
       where  $\tilde{\nu}\geq 3 $ is the generalized M\'{e}tivier index (also called the non-isotropic dimension of $\Omega$ related to the vector fields $X$, see Definition \ref{def2-1} below).
      \item $f(x,-u)=-f(x,u)$ for all $(x,u)\in\overline{\Omega}\times \mathbb{R}$.
 \item There exist $\mu>2$ and $R_0>0$ such that
\[0<\mu F(x,u)\leq f(x,u)u\]
for all  $x\in\overline{\Omega}$ and all $u\in \mathbb{R}$ with $|u|\geq R_0$, where $F(x,u)=\int_{0}^{u}f(x,v)dv$.
\item $f(x,u)u\geq 0$ for all $(x,u)\in\overline{\Omega}\times \mathbb{R}$.
\end{enumerate}
Moreover, we assume that $g$ is a Carath\'{e}odory function defined on $\Omega\times \mathbb{R}$ such that
 \begin{enumerate}[($g$1)]
             \item  $\lim_{u\to 0}{\frac{g(x,u)}{u}}= 0$ uniformly in $x\in\overline{\Omega}$.
            \item There exist $0<\sigma<\mu-1$ and  $C>0$ such that
         \[ |g(x,u)|\leq C(1+|u|^{\sigma})\qquad\forall~ (x,u)\in \overline{\Omega}\times \mathbb{R}. \]
      \item $g(x,u)u\geq0$ for all $x\in \overline{\Omega},u\in \mathbb{R}$.
  \end{enumerate}

When $X=(\partial_{x_{1}},\ldots,\partial_{x_{n}})$, $\triangle_{X}$ reduces to the classical Laplacian $\triangle$, and the generalized Métivier index $\tilde{\nu}$ coincides with the topological dimension $n$. It is noteworthy that the perturbation term $g(x,u)$ is not assumed to be odd symmetric in $u$, thus breaking the symmetry of the corresponding functional. Such semilinear elliptic problems are usually referred to as ``perturbation from symmetry problems." Since the 1980s, significant progress has been made in this area by researchers such as Bahri \cite{Bahri1981}, Bahri-Brestycki \cite{Bahri-Berestycki1981}, Bahri-Lions \cite{Bahri1988,Bahri1992}, Struwe \cite{Struwe1980,Struwe2000}, Rabinowitz \cite{Rabinowitz1982,Rabinowitz1986}, and Tanaka \cite{Tanaka1989}. These papers have demonstrated the existence of infinitely many weak solutions to problem \eqref{problem1-1} under appropriate assumptions on $f$ and $g$.

Compared to positive and negative solutions, sign-changing solutions exhibit more intricate qualitative properties, such as  the number and shapes of nodal domains and the measure of nodal sets. Consequently, sign-changing solutions pose interesting mathematical challenges. Initially, the existence of sign-changing solutions for problem \eqref{problem1-1}, equipped with the classical Laplacian $\triangle$, was primarily explored in the symmetric case. Notable works in this area include those by  Bartsch-Liu-Weth \cite{Bartsch2005}, Li-Wang \cite{Li-Wang-2002-tams}, among others. The sign-changing solutions in the non-symmetric case
was first investigated by Schechter-Zou \cite{Schechter-Zou2005}. Specifically, by utilizing the concepts of invariant sets and critical point theory,   they established the following sufficient condition:
\begin{equation}\label{1-1-2}
  \frac{2p}{n(p-2)}-1>\frac{\mu}{\mu-\sigma-1}.
\end{equation}
for the existence of infinitely many sign-changing solutions. Subsequently, Ramos-Tavares-Zou \cite{Ramos2009} refined their result, providing another sufficient condition:
\begin{equation}\label{1-2}
  \frac{2p}{n(p-2)}>\frac{\mu}{\mu-\sigma-1}.
\end{equation}
We refer readers to \cite{ Wang2014, Wang2016,Schechter-Zou2006}  for more results concerning the sign-changing solutions of classical elliptic equations.

It is worth pointing out that condition \eqref{1-2} is broader than condition \eqref{1-1-2},  as the triple parameters $(p,\mu,\sigma)$ satisfying \eqref{1-1-2} must also satisfy \eqref{1-2}. By analyzing the arguments of Schechter-Zou \cite{Schechter-Zou2005} and Ramos-Tavares-Zou \cite{Ramos2009}, it turn outs that \eqref{1-1-2} is determined by estimates on min-max values derived from the lower bound of the $k$-th Dirichlet eigenvalue $\lambda_{k}$ of the Laplacian $\triangle$ on $\Omega$, specifically $\lambda_k\geq Ck^{\frac{2}{n}}$,  while \eqref{1-2} was achieved through estimates on the min-max values via the Morse index estimate for associated sign-changing critical points. Comparing these arguments with the previous multiplicity results obtained by Rabinowitz \cite{Rabinowitz1982,Rabinowitz1986} and Bahri-Lions \cite{Bahri1988,Bahri1992} respectively, a philosophy insight arises: \emph{The Morse index approach usually provides a broader sufficient condition for estimating the min-max values associated with sign-changing critical points, enabling problem \eqref{problem1-1} to possess a sequence of sign-changing solutions.}

Recently, the study of nonlinear degenerate elliptic equations on sub-Riemannian manifolds has garnered considerable attention. These equations, particularly those involving H\"{o}rmander vector fields, have been extensively studied on equiregular sub-Riemannian manifolds (i.e., under M\'{e}tivier's condition; see Remark \ref{remark2-1} below). Early investigations in this area include works by Jerison-Lee \cite{Jerison-lee1987}, Garofalo-Lanconelli \cite{Garofalo1992}, Xu-Zuily \cite{Xu1997}, Citti \cite{Citti1995}, and Loiudice \cite{Loiudice2007}.

In non-equiregular sub-Riemannian manifolds, the existence and multiplicity of weak solutions to problem \eqref{problem1-1} were initially explored by Luyen-Tri \cite{Luyen2019} for specific Grushin-type operators, and subsequently by Chen-Chen-Yuan \cite{Chen-Chen-Yuan2022} for general H\"{o}rmander operators. These studies extended the classical results initially established by Rabinowitz \cite{Rabinowitz1982,Rabinowitz1986}.  Furthermore, our recent work \cite{Chen-Chen-Li-Liao2022} has improved the results presented in \cite{Luyen2019,Chen-Chen-Yuan2022}. Building upon the multiplicity results presented in \cite{Chen-Chen-Yuan2022,Chen-Chen-Li-Liao2022}, it naturally arises to inquire:  \emph{For problem \eqref{problem1-1} associated with general H\"{o}rmander operators $\triangle_{X}$, are these weak solutions sign-changing?} In this paper, we provide a satisfactory answer to this question by presenting two different types of sufficient conditions for problem \eqref{problem1-1} to have a sequence of sign-changing solutions.

To explore sign-changing solutions for problem \eqref{problem1-1}, one effective strategy is to utilize the perturbation method developed by Rabinowitz in \cite{Rabinowitz1982,Rabinowitz1986}, along with  the principle of flow invariance within restricted cones, as introduced by Ramos-Tavares-Zou \cite{Ramos2009}. However,  in comparison to the classical elliptic case, several technical issues arise due to the degeneracy of H\"{o}rmander operators. Firstly, to investigate the sign-changing critical points of the perturbation functional, it is necessary to select suitable min-max values that are compatible with sign-changing settings. The min-max values proposed by Ramos-Tavares-Zou \cite{Ramos2009}, which are based on topological considerations regarding odd continuous extensions as detailed by Castro-Clapp \cite{Castro2006}, are inapplicable to H\"{o}rmander vector fields. Secondly, the understanding of the Weyl-type asymptotic formula for Dirichlet eigenvalues of H\"{o}rmander operators in non-equiregular sub-Riemannian manifolds remains incomplete. This gap in knowledge poses additional challenges in estimating the min-max values associated with sign-changing critical points, which are derived from the lower bounds of Dirichlet eigenvalues.

Drawing inspiration from Liu-Liu-Wang \cite{Wang2014}, we construct a new intersection lemma in sign-changing settings (see Lemma \ref{lemma3-10} below). This intersection lemma enables us to define the min-max value associated with sign-changing critical points for degenerate elliptic equations. Additionally, instead of seeking explicit lower bounds of Dirichlet eigenvalues for H\"{o}rmander operators, we introduce the following lower bound assumption:
\[ \lambda_k\geq Ck^{\frac{2}{\vartheta}}(\ln k)^{-\kappa}~~~\mbox{for sufficiently large}~~k,\leqno(L)\]
Here, $\lambda_k$ denotes the $k$-th Dirichlet eigenvalue of $-\triangle_{X}$ on $\Omega$, $n\leq\vartheta\leq \tilde{\nu}$, $\kappa\geq 0$ and $C>0$ are constants dependent on $X$ and $\Omega$. We would like to mention that the class of H\"{o}rmander operators under assumption $(L)$ is quite large. For instance, if $0\in \Omega \subset\subset W=\mathbb{R}^n$, and the H\"{o}rmander vector fields $X$ further satisfy the following homogeneity condition:
\begin{enumerate}
  \item [(H.1)]
   There exists a family of non-isotropic dilations $\{\delta_{t}\}_{t>0}$ of the form
  \[ \delta_{t}:\mathbb{R}^n\to \mathbb{R}^n,\qquad \delta_{t}(x)=(t^{\sigma_{1}}x_{1},t^{\sigma_{2}}x_{2},\ldots,t^{\sigma_{n}}x_{n}), \]
 where $1=\sigma_{1}\leq \sigma_{2}\leq\cdots\leq \sigma_{n}$ are positive integers, such that for all $t>0$ and $f\in  C^{\infty}(\mathbb{R}^n)$,
 \[ X_{j}(f\circ \delta_{t})=t(X_{j}f)\circ \delta_{t}\qquad \forall j=1,\ldots,m, \]
 \end{enumerate}
then $\triangle_{X}=-\sum_{j=1}^{m}X_{j}^{*}X_{j}$ is the so-called homogeneous H\"{o}rmander operator. As shown in \cite{Chen-Chen-Li2022}, the corresponding Dirichlet eigenvalue $\lambda_{k}$ exhibits the asymptotic behaviour  $\lambda_k\approx k^{\frac{2}{\vartheta}}(\ln k)^{-\frac{2d}{\vartheta}}$ as $k\to+\infty$, where $n\leq \vartheta\leq \tilde{\nu}$ is a positive rational number and $0\leq d\leq n-1$ is an integer.

Combining our intersection lemma and the lower bound assumption $(L)$ on Dirichlet eigenvalue of $-\triangle_{X}$, we  can now establish the following multiplicity results for sign-changing solutions to problem \eqref{problem1-1}.

\begin{theorem}\label{thm1}
Let $X=(X_1,X_{2},\ldots,X_m)$ satisfy condition $(H)$, and let $\Omega\subset\subset U$ be a bounded open subset. Assume that for sufficiently large $k\geq 1$, the Dirichlet eigenvalue $\lambda_k$ of $-\triangle_{X}$ on $\Omega$ satisfy the lower bound assumption $(L)$. Suppose further that the functions $f$ and $g$ satisfy the assumptions $(f1)$-$(f5)$ and $(g1)$-$(g3)$, respectively.  If
\[ \frac{2p}{\vartheta(p-2)}-\frac{\tilde{\nu}}{\vartheta}>\frac{\mu}{\mu-\sigma-1},\leqno(A1)\]
then the problem \eqref{problem1-1} possesses an unbounded sequence of sign-changing weak solutions in $H_{X,0}^{1}(\Omega)$.
\end{theorem}

We further investigate the multiplicity of sign-changing solutions to problem \eqref{problem1-1} in the absence of assumption $(L)$. Utilizing the work of Bartsch-Liu-Weth \cite{Bartsch2005}, we introduce new min-max values that characterize sign-changing critical points through the relative genus. We then utilize the Marino-Prodi perturbation technique, as modified by Liu-Liu-Wang \cite{Wang2016}, to select perturbation functionals that exhibit a finite number of non-degenerate sign-changing critical points. Employing the Morse-Palais lemma and the transformation strategy proposed by Lazer-Solimini \cite{Lazer1988}, we establish a lower bound for the augmented Morse index at these sign-changing critical points. By combining Morse-index type estimates with the degenerate Cwikel-Lieb-Rozenblum inequality, we arrive at the following result.

\begin{theorem}\label{thm2}
Let $X=(X_{1},X_{2},\ldots,X_{m})$ and $\Omega$ satisfy the assumptions of Theorem \ref{thm1}. Suppose that the functions $f$ and $g$ satisfy the assumptions $(f1)$-$(f5)$ and $(g1)$-$(g3)$, respectively.  If
\[ \frac{2p}{\tilde{\nu}(p-2)}>\frac{\mu}{\mu-\sigma-1},\leqno(A2)\]
then the problem \eqref{problem1-1} possesses an unbounded sequence of sign-changing weak solutions.
\end{theorem}
\begin{remark}
Theorem \ref{thm1} and Theorem \ref{thm2} extend the classical results on sign-changing solutions previously established by Schechter-Zou \cite{Schechter-Zou2005} and  Ramos-Tavares-Zou \cite{Ramos2009}, respectively. In particular, if $X=(\partial_{x_{1}},\partial_{x_{2}},\ldots,\partial_{x_{n}})$, then $\triangle_{X}=\triangle$ and $\tilde{\nu}=n$. In this case, condition (A1) coincides with \eqref{1-1-2}, and condition (A2) simplifies to \eqref{1-2}.
\end{remark}
\begin{remark}
We would like to make more descriptions of the conditions $(A1)$ and $(A2)$. If the level set $H=\{x\in \Omega|\nu(x)=\tilde{\nu}\}$ has positive measure, \cite{Chen2021} establishes that $\lambda_k\sim C\cdot k^{\frac{2}{\tilde{\nu}}}$, where $\nu(x)$ denotes the pointwise homogeneous dimension defined in \eqref{2-1} below, and $C>0$ is a positive constant dependent on  $X$ and $\Omega$. This implies that $\tilde{\nu}=\vartheta$ and $\kappa=0$ in assumption $(L)$. In this case, condition $(A2)$, which is derived from estimating min-max values through the Morse index estimate on sign-changing critical points, is broader than condition $(A1)$. Note that the latter condition stems from estimates on min-max values that depend on the lower bound assumption  $(L)$.

However, it is important to mention that, unlike in the classical elliptic case, both approaches have their advantages in the degenerate case. In fact, when considering the general sub-elliptic equations with non-equiregular cases, there are examples associated with the triple parameters $(p,\mu,\sigma)$ that satisfy condition $(A1)$ but not condition $(A2)$ (see Section \ref{Section6} below). Therefore, Theorem \ref{thm2} cannot generally encompass Theorem \ref{thm1}. This presents a new phenomenon, where the multiplicity of sign-changing results for problem \eqref{problem1-1} in the degenerate case can significantly differ from those in the classical elliptic case.
\end{remark}

The rest of the paper will be organized as follows. In Section \ref{Section2}, we present several preliminary concepts, including the degenerate Rellich-Kondrachov compact embedding theorem, the degenerate Cwikel-Lieb-Rozenblum inequality, and appropriate estimates in critical point theory. In Section \ref{Section3}, we first construct a perturbation functional in sign-changing settings and discuss the corresponding intersection lemma and invariant properties. Then, we present the min-max values associated with the sign-changing critical points for the perturbation functional, along with a first lower bound of these min-max values derived from the lower bound of Dirichlet eigenvalue of $-\triangle_{X}$. In Section \ref{Section4}, we introduce the auxiliary functional and estimates augmented Morse index of its sign-changing critical points by using Marino-Prodi perturbation arguments and Morse-type estimates. In Section \ref{Section5}, we present the proofs of Theorem \ref{thm1} and Theorem \ref{thm2}. Finally, in Section \ref{Section6}, we illustrate the relationship between conditions $(A1)$ and $(A2)$ through a detailed example.

\textbf{\emph{Notations:}} For the sake of simplicity, different positive constants are usually denoted by $C$ sometimes without indices.

\section{Preliminaries}
\label{Section2}

\subsection{The framework of H\"{o}rmander vector fields}
Let  $X=(X_{1},X_{2},\ldots, X_{m})$ satisfy the H\"{o}rmander's condition $(H)$ on an open domain $W\subset \mathbb{R}^n$. For any
 bounded open subset $\Omega\subset\subset W$, there exists a smallest positive integer $r_{0}\geq 1$  such that $X_{1},X_{2},\ldots,X_{m}$,  together with their commutators of length up to $r_0$, span the tangent space $T_{x}(W)$ at each point  $x\in \overline{\Omega}$. The integer $r_{0}$ is referred to as the  H\"{o}rmander index of $\overline{\Omega}$ with respect to $X$. For H\"{o}rmander vector fields $X$, we can introduce the generalized M\'{e}tivier index, which is also called the non-isotropic dimension of $\Omega$ related to $X$ (cf. \cite{Chen2021,Yung2015}).

 \begin{definition}[Generalized M\'{e}tivier index]
    \label{def2-1}
    For each $x\in \overline{\Omega}$ and $1\leq j\leq r_{0}$, let $V_{j}(x)$
be the subspace of the tangent space at $x$  spanned by all
commutators of $X_{1},\ldots,X_{m}$ with length at most $j$. We denote by $\nu_{j}(x)$ the dimension of vector space $V_{j}(x)$ at  $x\in \overline{\Omega}$. The pointwise homogeneous dimension at $x$ is given
by
\begin{equation}\label{2-1}
  \nu(x):=\sum_{j=1}^{r_{0}}j(\nu_{j}(x)-\nu_{j-1}(x)),\qquad \nu_{0}(x):=0.
\end{equation}
Then we let
\begin{equation}\label{2-2}
  \tilde{\nu}:=\max_{x\in\overline{\Omega}} \nu(x)
\end{equation}
be the generalized M\'{e}tivier index of $\Omega$ associated with the vector fields $X$. According to \eqref{2-1} and \eqref{2-2},  $n+r_{0}-1\leq \tilde{\nu}< nr_{0}$ for $r_{0}>1$,
\end{definition}

\begin{remark}
\label{remark2-1}
 If for every $1\leq j\leq r_{0}$, $\nu_{j}(x)$ is a constant $\nu_{j}$ in some neighborhood of each $x\in \overline{\Omega}$, then we say  $X$ satisfy M\'{e}tivier's condition on $\Omega$ (cf.  \cite{Metivier1976}). The M\'{e}tivier index is given by
\begin{equation}\label{2-3}
  \nu=\sum_{j=1}^{r_{0}}j(\nu_{j}-\nu_{j-1}),\qquad \nu_{0}:=0,
\end{equation}
which coincides with the Hausdorff dimension of $\Omega$ related to the subelliptic metric induced by the vector fields $X$. Note that $\nu=\tilde{\nu}$ if the M\'{e}tivier's condition is satisfied.\par
 The M\'{e}tivier's condition, also known as the equiregular assumption in sub-Riemannian geometry (cf. \cite{Andrei2019}), imposes a strong restriction on the H\"ormander vector fields $X$. However, many H\"ormander vector fields fail to fulfill M\'{e}tivier's condition (e.g. the Grushin vector fields $X_{1}=\partial_{x_{1}}, X_{2}=x_{1}\partial_{x_{2}}$ in $\mathbb{R}^{2}$), and the generalized M\'{e}tivier index $\tilde{\nu}$ plays an important role in the geometry and  functional settings.
\end{remark}

Analogous to classical Sobolev spaces, the weighted Sobolev spaces associated with the vector fields
$X$ serve as the natural function spaces for addressing degenerate elliptic equations related to H\"{o}rmander operators $\triangle_{X}$. For vector fields $X=(X_{1},X_{2},\ldots,X_{m})$ defined on $W$, we set
\[ H_{X}^{1}(W)=\{u\in L^{2}(W):X_{j}u\in L^{2}(W),~ j=1,\ldots,m\}, \]
It is known that $H_{X}^{1}(W)$  is a Hilbert space endowed with the norm
\begin{equation}\label{2-4}
 \|u\|^2_{H^{1}_{X}(W)}=\|u\|_{L^2(W)}^2+\|Xu\|_{L^2(W)}^2=\|u\|_{L^2(W)}^2+\sum_{j=1}^{m}\|X_{j}u\|_{L^2(W)}^2.
\end{equation}
In addition, we denote by
   $H^{1}_{X,0}(\Omega)$ the closure of $C_{0}^{\infty}(\Omega)$ in $H_{X}^{1}(W)$, which is also a Hilbert space.

We next revist the Sobolev embedding theorems on $H^{1}_{X,0}(\Omega)$. In 1993, an initial study by Capogna-Danielli-Garofalo \cite{Capogna1993} established the embedding:
\begin{equation}\label{CDG}
  H_{X,0}^{1}(\Omega)\hookrightarrow L^{\frac{2Q}{Q-2}}(\Omega),
\end{equation}
where $Q$ denotes the local homogeneous dimension (refer to \cite[(3.4), p. 1166]{Capogna1996} and \cite[(3.1), p. 105]{Garofalo2018} for a precise definition) relative to the bounded set $\overline{\Omega}$. However, the Sobolev exponent $\frac{2Q}{Q-2}$ in \eqref{CDG} is not optimal for the space $H^{1}_{X,0}(\Omega)$ associated with general H\"{o}rmander vector fields. In 2015, Yung \cite[Corollary 1]{Yung2015} obtained that if $\Omega\subset\subset W$ is a bounded open set with smooth boundary $\partial\Omega$, then
\begin{equation}\label{Yung}
  \|f\|_{L^{\frac{2\tilde{\nu}}{\tilde{\nu}-2}}(\Omega)}\leq C\left(\|Xf\|_{L^2(\Omega)}+\|f\|_{L^2(\Omega)}\right)\qquad \forall f\in C^{\infty}(\overline{\Omega}),
\end{equation}
where $\tilde{\nu}$ is the generalized M\'{e}tivier index defined in \eqref{2-2}. From \cite[Proposition 2.2]{Chen2019}, \cite[(3.4), p. 1166]{Capogna1996} and \cite[(3.1), p. 105]{Garofalo2018}, it follows readily that $\tilde{\nu}\leq Q$. Additionally, as evidenced in  \cite[Example 2.4]{Chen-Chen-Li-Liao2022}, $\tilde{\nu}$ is strictly less than $Q$ in certain degenerate cases. Consequently, \eqref{Yung} suggests that the Sobolev exponent $\frac{2Q}{Q-2}$ in \eqref{CDG} can be refined to the sharp exponent $\frac{2\tilde{\nu}}{\tilde{\nu}-2}$ when $\partial\Omega$ is smooth.

Recalling the classical elliptic case, where the space $H_{0}^{1}(\Omega)$ consists of functions that ``\emph{vanishing at the boundary}",  the Sobolev embedding $H_{0}^{1}(\Omega)\hookrightarrow L^{\frac{2n}{n-2}}(\Omega)$ does not require any boundary smoothness. Building on this insight, the result in our recent paper  \cite{Chen-Chen-Li2024} extends \eqref{CDG} to arbitrary bounded open sets, adopting the sharp Sobolev exponent $\frac{2\tilde{\nu}}{\tilde{\nu}-2}$. Specifically, we have
 \begin{proposition}[Sharp Sobolev inequality, see {\cite[Theorem 1.1]{Chen-Chen-Li2024}}]
\label{prop2-1}
Let $X=(X_1,X_{2},\ldots,X_m)$ satisfy condition (H). Then, for any bounded open subset $\Omega\subset\subset W$, there exists a positive constant $C>0$ such that for $1\leq q\leq \frac{2\tilde{\nu}}{\tilde{\nu}-2}:=2_{\tilde{\nu}}^{*}$,
  \begin{equation}\label{sobolev-ine}
\|u\|_{L^{q}(\Omega)}\leq C\|Xu\|_{L^{2}(\Omega)}\qquad \forall u\in H^{1}_{X,0}(\Omega).
\end{equation}
Here, $\tilde{\nu}\geq 3$ is the generalized M\'{e}tivier index defined in \eqref{2-2}.
\end{proposition}

By setting $q=2$ in \eqref{sobolev-ine}, we derive the Friedrichs-Poincar\'{e} type inequality:
\begin{equation}\label{2-5}
 \|u\|_{L^{2}(\Omega)}\leq C\|Xu\|_{L^{2}(\Omega)}\qquad \forall u\in H^{1}_{X,0}(\Omega),
\end{equation}
which permits the adoption of the equivalent inner product $(\cdot,\cdot)_{H_{X,0}^{1}(\Omega)}$ on $H_{X,0}^{1}(\Omega)$, defined as:
\begin{equation}\label{2-6}
   (u,v)_{H_{X,0}^{1}(\Omega)}:=\int_{\Omega}Xu\cdot Xv dx,~~~\forall u,v\in H_{X,0}^{1}(\Omega).
\end{equation}
The associated norm is given by:
\begin{equation}\label{2-7}
   \|u\|_{H_{X,0}^{1}(\Omega)}^{2}:=\int_{\Omega}|Xu|^{2} dx,~~~\forall u\in H_{X,0}^{1}(\Omega).
\end{equation}
Throughout this paper, instead of \eqref{2-4}, we adopt \eqref{2-7} as the norm in $H_{X,0}^{1}(\Omega)$.

From Proposition \ref{prop2-1}, we can deduce  the following degenerate compact embedding theorem.
\begin{proposition}[Rellich-Kondrachov compact embedding theorem, see {\cite[Theorem 1.6]{Chen-Chen-Li2024}}]
\label{prop2-3}
Let $X$ and $\Omega$ satisfy the assumptions of Proposition  \ref{prop2-1}. Then  the embedding
\[ H_{X,0}^1(\Omega)\hookrightarrow L^s(\Omega) \]
is compact for every $1\leq s<2_{\tilde{\nu}}^*$, where $2_{\tilde{\nu}}^*=\frac{2\tilde{\nu}}{\tilde{\nu}-2}$ and $\tilde{\nu}\geq 3$ is the generalized M\'{e}tivier index defined in \eqref{2-2}.
\end{proposition}

Then, we consider the Dirichlet eigenvalue problem for the subelliptic Schr\"{o}dinger operator $-\triangle_{X}+V$ on $\Omega$,
  \begin{equation}\label{2-10}
    \left\{
      \begin{array}{ll}
    -\triangle_{X}u+V(x)u=\mu u   , & \hbox{$x\in \Omega$;} \\[2mm]
        u\in H_{X,0}^{1}(\Omega)
      \end{array}
    \right.
  \end{equation}
 where $V\in L^{\frac{p_{1}}{2}}(\Omega)$ with $p_{1}>\tilde{\nu}$.   Proposition \ref{prop2-1} and Proposition \ref{prop2-3} imply the validity of eigenvalue problem \eqref{2-10}.
 \begin{proposition}[{\cite[Proposition 2.12]{Chen-Chen-Li-Liao2022}}]
\label{prop2-5}
 Suppose that $X$ and $\Omega$ satisfy the assumptions of Proposition \ref{prop2-1}. If the potential term $V\in L^{\frac{p_{1}}{2}}(\Omega)$ with $p_{1}>\tilde{\nu}$, then the Dirichlet eigenvalue problem \eqref{2-10} of subelliptic Schr\"{o}dinger operator $-\triangle_{X}+V$ is well-defined, i.e. the self-adjoint operator $-\triangle_{X}+V$ admits a sequence of discrete Dirichlet eigenvalues $\mu_{1}\leq \mu_2\leq\cdots\leq\mu_k\leq\cdots$, and $\mu_{k}\to +\infty $ as $k\to +\infty$. Moreover, the corresponding eigenfunctions $\{\varphi_{k}\}_{k=1}^{\infty}$ constitute an orthonormal basis of $L^2(\Omega)$ and also an orthogonal basis of $H_{X,0}^{1}(\Omega)$.
\end{proposition}

Using Proposition \ref{prop2-1} and the results in \cite{Frank2010,Levin1997}, we can deduce the following degenerate Cwikel-Lieb-Rozenblum inequality for the  subelliptic Schr\"{o}dinger operator $-\triangle_{X}+V$.

\begin{proposition}[Cwikel-Lieb-Rozenblum inequality, see {\cite[Proposition 2.13]{Chen-Chen-Li-Liao2022}}]
\label{CLR}
Let $p_{1}>\tilde{\nu}$ be a given positive constant and $V\in L^{\frac{p_{1}}{2}}(\Omega)$ be a function satisfying $V\leq 0$, then there exists a positive constant $C>0$ such that
\begin{equation}\label{clr-ine}
  N(0,-\triangle_{X}+V)\leq C\int_{\Omega}|V(x)|^{\frac{\tilde{\nu}}{2}}dx,
\end{equation}
where $N(0,-\triangle_{X}+V):=\#\{k:\mu_{k}<0 \}$ denotes the numbers of negative Dirichlet eigenvalue of $-\triangle_{X}+V$, and $\tilde{\nu}$ is the generalized M\'{e}tivier index of $X$.
\end{proposition}

\subsection{Basic variational settings}
Let $H$ be a Hilbert space. We recall the following definitions.
 \begin{definition}
 \label{def2-2}
For $E\in C^{1}(H, \mathbb{R})$, we say a sequence $\{u_m\}_{m=1}^{\infty}$ in $H$ is a Palais-Smale sequence for functional $E$, if ${|E(u_m)|}\leq C$ uniformly in $m$, and ${\|E'(u_m)\|}_{H}\rightarrow0$ as $m\rightarrow\infty$, where $E':H\to H$ denotes
the Fr\'{e}chet derivative of $E$.
 \end{definition}

  \begin{definition}
  \label{def2-3}
 A functional $E\in C^{1}(H, \mathbb{R})$ satisfies the Palais-Smale condition  if any Palais-Smale sequence has a subsequence that is convergent in $H$.
  \end{definition}

We next introduce the weak solution to problem \eqref{problem1-1}.

\begin{definition}
\label{def2-4}
We say $u\in H_{X,0}^{1}(\Omega)$ is a weak solution to problem  \eqref{problem1-1} if
\begin{equation}
\label{2-12}
  \int_{\Omega}Xu\cdot Xvdx-\int_{\Omega}f(x,u)vdx-\int_{\Omega}g(x,u)vdx=0,~~~\forall v\in H_{X,0}^{1}(\Omega).
\end{equation}
\end{definition}

Consider the  energy functional $E: H_{X,0}^{1}(\Omega)\to \mathbb{R}$ of problem \eqref{problem1-1}, defined by
\begin{equation}\label{2-13}
  E(u):=\frac{1}{2}\int_{\Omega}{{|{Xu}|}^2} dx-\int_{\Omega}F(x,u) dx-\int_{\Omega}G(x,u)dx.
\end{equation}
Then, we have
\begin{proposition}[{\cite[Proposition 3.2]{Chen-Chen-Li-Liao2022}}]
\label{prop2-9}
If the functions $f$ and $g$ respectively satisfy the assumptions $(f2)$ and $(g2)$, then
  \[ E(u)=\frac{1}{2}\int_{\Omega}{{|{Xu}|}^2} dx-\int_{\Omega}F(x,u) dx-\int_{\Omega}G(x,u)dx \]
belongs to $C^{1}(H_{X,0}^{1}(\Omega), \mathbb{R})$. Thus the semilinear equation \eqref{problem1-1} is the Euler-Lagrange equation of the variational problem for the energy functional \eqref{2-13}. Furthermore, the Fr\'{e}chet derivative of $E$ at $u$ is given by
\begin{equation}\label{2-14}
 (E'(u),v)_{H_{X,0}^{1}(\Omega)}=\int_{\Omega}Xu\cdot Xvdx-\int_{\Omega}f(x,u)vdx-\int_{\Omega}g(x,u)vdx.
\end{equation}
Therefore, the critical point of $E$ in $H_{X,0}^{1}(\Omega)$ is the weak solution to \eqref{problem1-1}.
\end{proposition}

\section{Perturbation from symmetry in sign-changing settings}
\label{Section3}
In this section, we combine the perturbation method with sign-changing settings.

\subsection{Perturbation theory}

We first invoke the perturbation arguments by Rabinowitz \cite{Rabinowitz1986},  incorporating modifications inspired by of Ramos-Tavares-Zou \cite{Ramos2009}. Roughly speaking, we construct a perturbation functional $J$, which is a modification of $E$ such that $J$ and $E$ share the same critical points and critical values at high energy levels.
Then, seeking a sequence of sign-changing solutions to problem \eqref{problem1-1} amounts to showing that $J$ has a sequence of sign-changing critical points. To ensure the exposition is reasonably self-contained,  we repeat the details in the construction of $J$.

Before constructing the functional $J$, we present some necessary estimates. According to assumptions ($f$1)-($f$2) and ($g$1)-($g$2), we have
\begin{equation}\label{3-1}
\begin{aligned}
 |f(x,u)| &\leq \varepsilon |u|+C(\varepsilon)|u|^{p-1},~~~\forall (x,u)\in \overline{\Omega}\times\mathbb{R},\\
|g(x,u)| &\leq \varepsilon |u|+C(\varepsilon)|u|^{\sigma},~~~\forall (x,u)\in \overline{\Omega}\times\mathbb{R},
\end{aligned}
\end{equation}
which means that for all $(x,u)\in\overline{\Omega}\times\mathbb{R}$ and any $\varepsilon>0$,
\begin{equation}\label{3-2}
|F(x,u)|\leq \varepsilon|u|^2+C(\varepsilon)|u|^{p},\quad\mbox{and}\quad |G(x,u)| \leq \varepsilon|u|^2+C(\varepsilon)|u|^{\sigma+1}.
\end{equation}
By assumption ($f$4), for $x\in\overline{\Omega}$ and $|u|\geq R_0$ we have
\[ u|u|^{\mu}\frac{\partial}{\partial u}\left(|u|^{-\mu}F(x,u)\right)=f(x,u)u-\mu F(x,u)\geq 0, \]
 which implies that $F(x,u)|u|^{-\mu}\geq \gamma_0(x):=R_0^{-\mu}\min\{F(x,-R_0),F(x,R_0)\}$.
 Since the positive functions $F(x, R_0)$ and $F(x,-R_0)$ belong to $C(\overline{\Omega})$, there exists a constant $c_{1}>0$ such that $\gamma_0(x)\geq c_{1}>0$ for any $x\in\overline{\Omega}$. Thus
 \begin{equation}\label{3-3}
 F(x,u)\geq c_{1}|u|^\mu,~~~\forall x\in\overline{\Omega},~|u|\geq R_0.
 \end{equation}
 Combining \eqref{3-3} and assumption ($f$4), for some $c_{2}, c_{3}>0 $, we have
 \begin{equation}\label{3-4}
\frac{1}{\mu}\left(uf(x,u)+c_3\right)\geq F(x,u)+c_2\geq c_1|u|^\mu~~~\forall (x,u)\in \overline{\Omega}\times\mathbb{R}.
\end{equation}
Observing that $\mu>2$ and $1+\sigma<\mu$, \eqref{3-2} and \eqref{3-4} imply that $1+\sigma<\mu\leq p$. Thus, assumption ($g$2) gives
\begin{equation}\label{3-5}
|g(x,u)|\leq C(1+|u|^{p-1})~~~\forall (x,u)\in \overline{\Omega}\times\mathbb{R}.
\end{equation}
Moreover,
\begin{equation}\label{3-6}
|g(x,u)| \leq \varepsilon|u|+C(\varepsilon)|u|^{p-1},~~~|G(x,u)| \leq \varepsilon|u|^2+C(\varepsilon)|u|^{p},~~~\forall (x,u)\in \overline{\Omega}\times\mathbb{R}.
\end{equation}

Then, we present the following Lemmas \ref{lemma3-1}-\ref{lemma3-7} under the assumptions  $(f1)$-$(f5)$ and $(g1)$-$(g3)$.

\begin{lemma}\label{lemma3-1}
There exists a positive constant $A>0$ such that if $u\in H_{X,0}^{1}(\Omega)$ is a critical point of $E$, then
 \begin{equation}\label{3-7}
\int_{\Omega}F(x,u)dx\leq A(I(u)^2+1)^{\frac{1}{2}},
\end{equation}
where
\begin{equation}\label{3-8}
  I(u):=\frac{1}{2}\int_{\Omega}|Xu|^2dx-\int_{\Omega}F(x,u)dx.
\end{equation}
\end{lemma}
\begin{proof}
  Let $u\in H_{X,0}^{1}(\Omega)$ be a critical point of $E$, we have $(E'(u),u)_{H_{X,0}^1(\Omega)}=0$. Hence, by \eqref{3-4} and assumption ($g$3) we obtain
  \begin{equation}
\begin{aligned}\label{3-9}
I(u)&=I(u)-\frac{1}{2}(E'(u),u)_{H_{X,0}^1(\Omega)}=\int_{\Omega}\left(\frac{1}{2}uf(x,u)-F(x,u)+\frac{1}{2}g(x,u)u\right)dx\\
&\geq \left(\frac{\mu}{2}-1\right)\int_{\Omega}F(x,u)dx-C\geq a_1 \int_{\Omega}F(x,u)dx-a_2,
\end{aligned}
\end{equation}
where $a_{1},a_{2}$ are some positive constants. Consequently,  \eqref{3-7} is derived from \eqref{3-9}.
\end{proof}

Let $\chi \in C^\infty(\mathbb{R})$ be a real smooth  function with $\chi(\xi)\equiv 1$ for $\xi\leq 1$, $\chi(\xi)\equiv 0$ for $\xi\geq 2$, and $\chi'(\xi)\in (-2,0)$ for $\xi \in (1,2)$. Denote by
\[Q(u):=2A(I(u)^2+1)^{\frac{1}{2}}\]
and
\[\theta(u):=\chi\left(Q(u)^{-1}\int_{\Omega}F(x,u)dx \right).\]
 Lemma \ref{lemma3-1} indicates that, if $u$ is a critical point of $E$,  $Q(u)^{-1}\int_{\Omega}F(x,u)dx$ lies in $\left[0,\frac{1}{2}\right]$ and then $\theta(u)=1$. Now, we present the perturbation functional
\begin{equation}\label{3-10}
\begin{aligned}
J(u)&:=E(u)+(1-\theta(u))\int_{\Omega} G(x,u)dx\\
&=\frac{1}{2}\int_{\Omega}{{|{Xu}|}^2} dx-\int_{\Omega}F(x,u) dx-\theta(u)\int_{\Omega}G(x,u)dx~~~\forall u\in H_{X,0}^{1}(\Omega).
\end{aligned}
\end{equation}
Clearly, $J(u)=E(u)$ if $u$ is a critical point of $E$. Since $\chi$ is smooth, then $\theta\in C^1(H_{X,0}^1(\Omega), \mathbb{R})$ and therefore $J\in C^1(H_{X,0}^1(\Omega), \mathbb{R})$. Furthermore, $J(u)$ satisfies the following estimate.
\begin{lemma}\label{lemma3-2}
There is a positive constant $C>0$ such that
\begin{equation}\label{3-11}
|J(u)-J(-u)|\leq C(|J(u)|^{\frac{\sigma+1}{\mu}}+1)~~~\forall u\in H_{X,0}^{1}(\Omega).
\end{equation}

\end{lemma}

\begin{proof}
If $u\notin {\rm supp}~\theta$, then $\theta(u)=0$, and \eqref{3-11} is trivially holds since $F(x,u)=F(x,-u)$.
Hence, we only need to consider the case where  $u\in {\rm supp}~\theta$. For each $u\in {\rm supp}~\theta$, the definition of $\theta$ implies
\begin{equation}\label{3-12}
\int_{\Omega}F(x,u)dx\leq 4A(I(u)^2+1)^{\frac{1}{2}}\leq C|I(u)|+C.
\end{equation}
  Then by assumptions $(g2)$ and $(g3)$, \eqref{3-4} and \eqref{3-12} we have
   \begin{equation}\label{3-13}
    \begin{aligned}
   0\leq\int_{\Omega}G(x,u)dx&\leq C\int_{\Omega}\left(|u|+|u|^{\sigma+1}\right)dx\leq\int_{\Omega}\left(C+C|u|^{\sigma+1}\right)dx\\
   &\leq C+C\left(\int_{\Omega}|u|^\mu dx\right)^{\frac{\sigma+1}{\mu}}\leq C+C\left(\int_{\Omega}(F(x,u)+C) dx\right)^{\frac{\sigma+1}{\mu}}\\
   &\leq C+C\left(|I(u)|+C \right)^{\frac{\sigma+1}{\mu}}\leq C+C|I(u)|^{\frac{\sigma+1}{\mu}}.
    \end{aligned}
  \end{equation}
Besides, \eqref{3-8} and \eqref{3-13} implies
  \begin{equation}\label{3-14}
    \begin{aligned}
    |I(u)|=\left|E(u)+\int_{\Omega}G(x,u)dx\right|\leq |E(u)|+C|I(u)|^{\frac{\sigma+1}{\mu}}+C.
    \end{aligned}
  \end{equation}
Since $\frac{\sigma+1}{\mu}<1$, the term $C|I(u)|^{\frac{\sigma+1}{\mu}}$ in \eqref{3-14} can be absorbed into the left hand side by Young's inequality and we have
  \begin{equation}\label{3-15}
  |I(u)|\leq C|E(u)|+C.
  \end{equation}
Observing that $I(u)$ and $\theta(u)$ are both even functionals, we can deduce from \eqref{3-10}, \eqref{3-13} and \eqref{3-15} that
  \begin{equation}\label{3-16}
    \begin{aligned}
    |J(u)-J(-u)|&\leq \theta(u)\left|\int_{\Omega}G(x,u)dx-\int_{\Omega}G(x,-u)dx\right|\\
    &\leq C+C|I(u)|^{\frac{\sigma+1}{\mu}}\leq C+C|E(u)|^{\frac{\sigma+1}{\mu}}.
    \end{aligned}
  \end{equation}
Additionally, \eqref{3-10} and \eqref{3-13} give that
  \begin{equation}\label{3-17}
    |E(u)|\leq |J(u)|+\left|\int_{\Omega}G(x,u)dx\right|\leq |J(u)|+C|E(u)|^{\frac{\sigma+1}{\mu}}+C,
  \end{equation}
which means $|E(u)|\leq C|J(u)|+C$. Therefore,
 $|J(u)-J(-u)|\leq C(|J(u)|^{\frac{\sigma+1}{\mu}}+1)$ for all $u\in H_{X,0}^{1}(\Omega)$.
\end{proof}

Let us analyze the Fr\'{e}chet derivative of functional $J$. For any $u,v\in H_{X,0}^1(\Omega)$, we have
\begin{equation}\label{3-18}
(J'(u), v)_{H_{X,0}^{1}(\Omega)}=(u,v)_{H_{X,0}^{1}(\Omega)}-(K_J(u),v)_{H_{X,0}^{1}(\Omega)},
\end{equation}
where
\begin{equation}\label{3-19}
 (K_J(u),v)_{H_{X,0}^{1}(\Omega)}=(1+T_2(u))\int_{\Omega}f(x,u)v dx+\theta(u)\int_{\Omega}g(x,u)vdx-T_1(u)\int_{\Omega}Xu\cdot Xvdx,
\end{equation}
with
\begin{equation}
\begin{aligned}\label{3-20}
T_1(u)&=\chi'(\varphi(u))\varphi(u)(2A)^2Q(u)^{-2}I(u)\int_{\Omega}G(x,u)dx,\\
T_2(u)&=\chi'(\varphi(u))Q(u)^{-1}\int_{\Omega}G(x,u)dx+T_1(u),
\end{aligned}
\end{equation}
and
\begin{equation}\label{3-21}
 \varphi(u)=Q(u)^{-1}\int_{\Omega}F(x,u)dx.
\end{equation}

Then, it follows that
\begin{lemma}
\label{lemma3-3}
The operators $T_{1}$ and $T_{2}$ satisfy
\begin{equation}\label{3-22}
\begin{split}
|T_1(u)|&\leq C(|I(u)|^{\frac{\sigma+1}{\mu}}+1)|I(u)|^{-1},\\
|T_2(u)|&\leq C(|I(u)|^{\frac{\sigma+1}{\mu}}+1)|I(u)|^{-1}.
\end{split}
\end{equation}
Moreover, if $J(u)\to +\infty$, then  $I(u)\to+\infty$, $T_{1}(u)\to 0$ and $T_2(u)\to 0$.
\end{lemma}
\begin{proof}
If $u\in \text{supp}~\theta$ with  $\varphi(u)<1$, then $\chi'(\varphi(u))=0$, which gives $T_{1}(u)=T_{2}(u)=0$.
For $u\in \text{supp}~\theta$ with $1\leq \varphi(u)\leq 2$, \eqref{3-22} is derived from \eqref{3-13} and \eqref{3-20}.

On the other hand, observing that
\[I(u)+\int_{\Omega}G(x,u) dx\geq J(u),\]
 by \eqref{3-13} we have
\begin{equation}\label{3-23}
I(u)+C|I(u)|^{\frac{\sigma+1}{\mu}}\geq J(u)-C.
\end{equation}
Consequently,  $I(u)\to+\infty$, $T_{1}(u)\to 0$ and $T_2(u)\to 0$ as $J(u)\to +\infty$.
\end{proof}

\begin{lemma}
\label{lemma3-4}
There is a constant $M_0>0$ such that if $J(u)\geq M_0$ and $J'(u)=0$, then $J(u)=E(u)$ and $E'(u)=0$.
\end{lemma}
\begin{proof}
 It is sufficient to show that if $M_0$ is large and $u$ is a critical point of $J$ with $J(u)\geq M_0$, then \begin{equation}\label{3-24}
Q(u)^{-1}\int_{\Omega}F(x,u)dx<1.
\end{equation}
Indeed, \eqref{3-24} gives $\varphi<1$ and $\theta\equiv1$ in a neighborhood of $u$. This means $\chi'(\varphi(u))=0$ and $T_{1}(u)=T_{2}(u)=0$. Hence,  $J(u)=E(u)$ and $J'(u)=E'(u)=0$, which yields  Lemma \ref{lemma3-4}.

Let $u\in H_{X,0}^{1}(\Omega)$ be a critical point of $J$. If $u\in \text{supp}~\theta$ with $\varphi(u)<1$, then $T_1(u)=T_2(u)=0$ and $\theta(v)\equiv1 $ in a neighborhood of $u$. Thus, $u$ is also a critical point of $E$ and Lemma \ref{lemma3-1} gives \eqref{3-24}. For $u\notin \text{supp}~\theta$, one has $\varphi(u)>2$, and $T_1(u)=T_2(u)=0$ due to $\chi'(\varphi(u))=0$. Observe that
\begin{equation}\label{3-25}
\begin{aligned}
J(u)&=J(u)-\frac{1}{2(1+T_1(u))}(J'(u), u)_{H_{X,0}^{1}(\Omega)}\\
&=\frac{1+T_2(u)}{2(1+T_1(u))}\int_{\Omega}f(x,u)udx-\int_{\Omega}F(x,u)dx\\
&-\theta(u)\int_{\Omega}G(x,u)dx+\frac{\theta(u)}{2(1+T_{1}(u))}\int_{\Omega}g(x,u)udx\\
&=\frac{1}{2}\int_{\Omega}f(x,u)udx-\int_{\Omega}F(x,u)dx,
\end{aligned}
\end{equation}
and $I(u)=J(u)$ for the functional $I(u)$ defined in \eqref{3-8}.
By \eqref{3-25} and the similar estimates in \eqref{3-9}, we can obtain \eqref{3-24}. When $u\in \text{supp}~\theta$ with $1\leq \varphi(u)\leq 2$,  Lemma \ref{lemma3-3} implies that there exists a positive constant $M_{0}$ such that
$|T_{1}(u)|, |T_{2}(u)|\leq \frac{1}{2}$ and $ \frac{1+T_{2}(u)}{1+T_{1}(u)}>\frac{1}{\mu}+\frac{1}{2}$ for $J(u)\geq M_0$.
Owing to \eqref{3-8} and \eqref{3-10},
\begin{equation}\label{3-26}
J(u)\leq |J(u)|\leq |I(u)|+|\theta(u)|\left|\int_{\Omega}G(x,u)dx\right|\leq |I(u)|+\left|\int_{\Omega}G(x,u)dx\right|.
\end{equation}
On the other hand, by \eqref{3-4} and \eqref{3-25}, we have
\begin{equation}\label{3-27}
\begin{aligned}
J(u)&\geq \left(\frac{1+T_2(u)}{2(1+T_1(u))}-\frac{1}{\mu}\right)\int_{\Omega}(uf(x,u)+c_{3}) dx\\
&-\widehat{C}(u)\int_{\Omega}|g(x,u)u|dx-\theta(u)\left|\int_{\Omega}G(x,u)dx\right|-c_8,
\end{aligned}
\end{equation}
where $$\widehat{C}(u)=\left|\frac{\theta(u)}{2(1+T_{1}(u))}\right|\leq 1.$$
According to assumption ($g$2), \eqref{3-2}, \eqref{3-26} and \eqref{3-27}, we obtain
\begin{equation}\label{3-28}
\begin{aligned}
|I(u)|&\geq \left(\frac{1}{2}\left(\frac{1}{\mu}+\frac{1}{2}\right)-\frac{1}{\mu}\right)\int_{\Omega}(uf(x,u)+c_{3}) dx-\int_{\Omega}|g(x,u)u|dx-2\int_{\Omega}|G(x,u)|dx-c_8\\
&\geq \frac{1}{2}\left(\frac{\mu}{2}-1\right)\int_{\Omega}F(x,u)dx-C\int_{\Omega}|u|^{2}dx-C\int_{\Omega}|u|^{\sigma+1}dx-C.
\end{aligned}
\end{equation}
The Young's inequality gives that for any $u\in\mathbb{R}$ and any $\varepsilon>0$,
 if $0<s<\mu$,
\begin{equation}\label{3-29}
  |u|^s\leq \varepsilon|u|^\mu+\varepsilon^{-\frac{s}{\mu-s}}.
\end{equation}
Using \eqref{3-4}, \eqref{3-28} and \eqref{3-29}, we can also deduce \eqref{3-24} by the similar estimates in \eqref{3-9} with $A$ replaced by a larger constant which is smaller than $2A$.
\end{proof}

\begin{lemma}
\label{lemma3-5}
There is a constant $M_1\geq M_0$ such that $J$ satisfies Palais-Smale condition on $\widehat{A}_{M_{1}}:=\{u\in H_{X,0}^{1}(\Omega):J(u)\geq M_{1}\}$, where $M_0$ is the positive constant appeared in Lemma \ref{lemma3-4}.
\end{lemma}
\begin{proof}
Lemma \ref{lemma3-5} amounts to showing that there exists $M_1> M_0$ such that if the sequence $\{u_m\}_{m=1}^{\infty}\subset H_{X,0}^1(\Omega)$ satisfies $M_1\leq J(u_m)\leq M^*$ for some $M^*>0$ and $J'(u_m)\to 0$ as $m\to \infty$, then $\{u_m\}_{m=1}^{\infty}$ is bounded in $H_{X,0}^1(\Omega)$ and admits a convergent subsequence.

For sufficient large $m$ (such that $\|J'(u_m)\|_{H_{X,0}^{1}(\Omega)}<1$) and any $\rho>0$, it follows from \eqref{3-10}, \eqref{3-18} and \eqref{3-19} that
\begin{equation}\label{3-30}
\begin{aligned}
M^*+\rho\|u_m\|_{H_{X,0}^1(\Omega)}&\geq J(u_m)-\rho(J'(u_m), u_m)_{H_{X,0}^1(\Omega)}\\
&\geq \left(\frac{1}{2}-\rho(1+T_1(u_m))\right)\|u_m\|^2_{H_{X,0}^1(\Omega)}\\
&+\rho(1+T_2(u_m))\int_{\Omega}f(x,u_m)u_m dx-\int_{\Omega}F(x,u_m)dx\\
&+\rho\theta(u_m)\int_{\Omega}g(x,u_{m})u_m dx-\theta(u_{m})\int_{\Omega}G(x,u_{m})dx.
\end{aligned}
\end{equation}
From Lemma \ref{lemma3-3} we know  there is a positive constant $M_1>M_{0}$, such that $|T_{1}(u)|, |T_{2}(u)|\leq \frac{1}{2}$ and $ \frac{1+T_{2}(u)}{1+T_{1}(u)}>\frac{2}{\mu}$ on $\widehat{A}_{M_{1}}$. Now, taking
 $\rho>0$ and $\varepsilon>0$ such that
\begin{equation}\label{3-31}
\frac{1}{2(1+T_1(u_m))}> \rho +\varepsilon > \rho-\varepsilon > \frac{1}{\mu(1+T_2(u_m))},
\end{equation}
we can deduce from  \eqref{3-4}, \eqref{3-29}, \eqref{3-30}, \eqref{3-31} and assumption ($g$2) that
\begin{equation*}
  \begin{aligned}
  &M^*+\rho\|u_m\|_{H_{X,0}^1(\Omega)}\\
&\geq
\varepsilon(1+T_1(u_m))\|u_m\|^2_{H_{X,0}^1(\Omega)}+\left(\frac{1}{\mu}+\varepsilon(1+T_2(u_m))\right)\int_{|u_m|\geq R_0}f(x,u_m)u_m dx\nonumber\\
&+\rho(1+T_2(u_m))\int_{|u_m|\leq R_0}f(x,u_m)u_m dx-\int_{\Omega}F(x,u_m)dx\\
&+\rho\theta(u_m)\int_{\Omega}g(x,u_{m})u_m dx-\theta(u_{m})\int_{\Omega}G(x,u_{m})dx\\
&\geq \frac{\varepsilon}{2}\|u_m\|^2_{H_{X,0}^1(\Omega)}+\left(\frac{1}{\mu}+\frac{\varepsilon}{2}\right)\int_{|u_m|\geq R_0}f(x,u_m)u_m dx\\
&-\int_{|u_m|\geq R_0}F(x,u_m)dx-\int_{\Omega}|g(x,u_{m})u_{m}|dx-\int_{\Omega}|G(x,u_{m})|dx-C\\
&\geq \frac {\varepsilon\|u_m\|^2_{H_{X,0}^1(\Omega)}}{2}+\frac{\varepsilon \mu}{2}\int_{\Omega}(F(x,u_m)+c_{2})dx-\int_{\Omega}|g(x,u_{m})u_{m}|dx-\int_{\Omega}|G(x,u_{m})|dx-C\\
&\geq
\frac{\varepsilon\|u_m\|^2_{H_{X,0}^1(\Omega)}}{2}+\frac{\varepsilon \mu}{2}\int_{\Omega}(F(x,u_m)+c_{2})dx-C\int_{\Omega}(|u_m|+|u_m|^{1+\sigma})dx-C\\
&\geq \frac{\varepsilon}{2}\|u_m\|^2_{H_{X,0}^1(\Omega)}+\frac{\varepsilon \mu}{4}\int_{\Omega}(F(x,u_m)+c_{2})dx-C\geq \frac{\varepsilon}{2}\|u_m\|^2_{H_{X,0}^1(\Omega)}-C,
\end{aligned}
\end{equation*}
which yields that $\{u_m\}_{m=1}^{\infty}$ is bounded in $H_{X,0}^1(\Omega)$.

We next show that $\{u_m\}_{m=1}^{\infty}$ has a convergent subsequence in  $H_{X,0}^1(\Omega)$.
For any $u\in H_{X,0}^{1}(\Omega)$, the Rellich-Kondrachov compact embedding theorem (Proposition \ref{prop2-3}) implies that the
linear functionals $K_{1}(u),K_{2}(u)$ given by
\begin{equation}\label{3-35}
  ( K_{1}(u),v)_{H_{X,0}^1(\Omega)}:=-\int_{\Omega}f(x,u)vdx\quad\mbox{and}\quad ( K_{2}(u),v)_{H_{X,0}^1(\Omega)}:=-\int_{\Omega}g(x,u)vdx~~~~\forall v\in H_{X,0}^{1}(\Omega)
\end{equation}
 map the bounded sets in $H_{X,0}^{1}(\Omega)$ to the relatively compact sets in $H_{X,0}^{1}(\Omega)$ (cf. \cite[Proposition 2.15]{Chen-Chen-Li-Liao2022}). It follows from \eqref{3-18} and \eqref{3-19} that the Fr\'{e}chet derivative of $J$ can be decomposed into
\begin{equation}\label{3-37}
J'(u_m)=(1+T_1(u_m))u_m+(1+T_2(u_m))K_{1}(u_m)+\theta(u_{m})K_{2}(u_{m}).
\end{equation}
Furthermore, since $\{T_1(u_m)\}_{m=1}^{\infty}$, $\{T_2(u_m)\}_{m=1}^{\infty}$ and $\{\theta(u_m)\}_{m=1}^{\infty}$ are bounded, there exists a subsequence $\{u_{m_k}\}_{k=1}^{\infty}\subset \{u_{m}\}_{m=1}^{\infty}$ such that
\begin{equation}\label{3-38}
\lim_{k\to\infty}T_1(u_{m_k})=\widetilde{a_1},~~\lim_{k\to\infty}T_2(u_{m_k})=\widetilde{a_2},~~\lim_{k\to\infty}\theta(u_{m_k})=\widetilde{a_3}.
\end{equation}
Observing  $\{u_{m_{k}}\}_{k=1}^{\infty}$ is also bounded in $H_{X,0}^1(\Omega)$, $K_{1}(u_{m_{k}})$ and $K_{2}(u_{m_{k}})$ converge along a subsequence $\{u_{m_{k_{j}}}\}_{j=1}^{\infty}\subset \{u_{m_k}\}_{k=1}^{\infty}$. As a result, we conclude that $\{u_{m_{k_j}}\}_{j=1}^{\infty}$ converges in $H_{X,0}^1(\Omega)$.
\end{proof}

Let $N_k$ be the subspace spanned by the first $k$ Dirichlet eigenfunctions of $-\triangle_{X}$, and let $\lambda_k$ denote the $k$-th eigenvalue. Then we have

\begin{lemma}\label{lemma3-6}
For each fixed $k\geq 1$,
\[\lim_{u\in N_k,~\|u\|_{H_{X,0}^{1}(\Omega)}\to+\infty}J(u)=-\infty.\]
\end{lemma}
\begin{proof}
The min-max theorem indicates that
$\|u\|_{H_{X,0}^{1}(\Omega)}^2\leq \lambda_k\|u\|_{L^2(\Omega)}^2$ for all $u\in N_k$.
Since $\mu>2$,  we obtain from \eqref{3-4} that
\begin{align*}
  \lim_{u\in N_k,~\|u\|_{H_{X,0}^{1}(\Omega)}\to+\infty}\frac{\int_{\Omega}F(x,u)dx}{\|u\|_{H_{X,0}^{1}(\Omega)}^2}
  &\geq \lim_{u\in N_k,~\|u\|_{L^2(\Omega)}\to+\infty}\frac{
c_1\|u\|_{L^{\mu}(\Omega)}^\mu-c_{2}|\Omega|}{\lambda_k\|u\|_{L^2(\Omega)}^2}\\
&\geq\lim_{u\in N_k,~\|u\|_{L^2(\Omega)}\to+\infty}\frac{
c_1|\Omega|^{1-\frac{\mu}{2}}\|u\|_{L^{2}(\Omega)}^\mu-c_{2}|\Omega|}{\lambda_k\|u\|_{L^2(\Omega)}^2}=+\infty.
\end{align*}
Observing that $G(x,u)\geq0$, we conclude from \eqref{3-10} that  $\lim_{u\in N_k,~\|u\|_{H_{X,0}^{1}(\Omega)}\to+\infty}J(u)=-\infty$.
\end{proof}

\begin{lemma}\label{lemma3-7}
For each fixed $k\geq 2$, there exist positive constants $C, C_0$  independent of $k$ such that
  \begin{equation}\label{3-39}
    J(u) \geq C \lambda_{k}^{\frac{2(2_{\tilde{\nu}}^{*}-p)}{(p-2)(2_{\tilde{\nu}}^{*}-2)}}\qquad\forall u\in Q_k,
  \end{equation}
  where
 \begin{equation}\label{3-40}
Q_k:=\left\{u\in N_{k-1}^\perp:\|u\|_{L^{p}(\Omega)}=\rho_{k}\right\}~~~\mbox{with}~~~\rho_{k}:=C_0
  \lambda_{k}^{\frac{2(2_{\tilde{\nu}}^{*}-p)}{p(p-2)(2_{\tilde{\nu}}^{*}-2)}}.
 \end{equation}

\end{lemma}
\begin{proof}
Using \eqref{3-2} and \eqref{3-6}, we have for small $\varepsilon>0$,
\begin{equation}\label{3-41}
\begin{aligned}
J(u)&=\frac{1}{2}\int_{\Omega}|Xu|^2dx-\int_{\Omega}F(x,u)dx-\theta(u)\int_{\Omega}G(x,u)dx\\
&\geq\frac{1}{2}\int_{\Omega}|Xu|^2dx-\int_{\Omega}(\varepsilon|u|^2+C(\varepsilon)|u|^p)dx\\
&\geq \frac{1}{4}\|u\|_{H_{X,0}^{1}(\Omega)}^2-C\|u\|_{L^{p}(\Omega)}^p.
\end{aligned}
\end{equation}
Since for all $ u \in N_{k-1}^\perp$, the Rayleigh-Ritz formula and Proposition \ref{prop2-1} give that
\begin{equation}\label{3-42}
  \|u\|_{L^{2}(\Omega)}\leq \lambda_{k}^{-\frac{1}{2}}\|u\|_{H_{X,0}^{1}(\Omega)}
\end{equation}
and
\begin{equation}\label{3-43}
\|u\|_{L^{2_{\tilde{\nu}}^{*}}(\Omega)}\leq C\|u\|_{H_{X,0}^{1}(\Omega)}.
\end{equation}
Substituting $r=\frac{2(2_{\tilde{\nu}}^{*}-p)}{p(2_{\tilde{\nu}}^{*}-2)}$, then $\frac{1}{p}=\frac{r}{2}+\frac{1-r}{2_{\tilde{\nu}}^{*}}$. \eqref{3-42} and \eqref{3-43} imply
\begin{equation}\label{3-44}
\|u\|_{L^p(\Omega)}\leq \|u\|_{L^{2}(\Omega)}^r\|u\|_{L^{2_{\tilde{\nu}}^{*}}(\Omega)}^{1-r}\leq
C \lambda_{k}^{-\frac{2_{\tilde{\nu}}^{*}-p}{p(2_{\tilde{\nu}}^{*}-2)}}\|u\|_{H_{X,0}^{1}(\Omega)}\quad \forall u \in N_{k-1}^\perp.
\end{equation}
Hence, by \eqref{3-41} and \eqref{3-44}
\[ J(u)\geq C\lambda_{k}^{\frac{2(2_{\tilde{\nu}}^{*}-p)}{p(2_{\tilde{\nu}}^{*}-2)}}\|u\|_{L^{p}(\Omega)}^{2}-C\|u\|_{L^{p}(\Omega)}^{p}. \]
 If we take $\|u\|_{L^{p}}=C_0\lambda_{k}^{\frac{2(2_{\tilde{\nu}}^{*}-p)}{p(p-2)(2_{\tilde{\nu}}^{*}-2)}}$ for suitable small $C_0>0$, then $J(u)  \geq C \lambda_{k}^{\frac{2(2_{\tilde{\nu}}^{*}-p)}{(p-2)(2_{\tilde{\nu}}^{*}-2)}}$.
\end{proof}

\subsection{Sign-changing settings}

We next introduce some notation related to the sign-changing settings. Throughout the paper, we adopt the notations $u^{+}:=\max\{u,0\}$ and $u^{-}:=\min\{u,0\}$.

\begin{lemma}[{\cite[Proposition 2.9]{Chen-Chen-Li2024}}]\label{lemma3-8}
  If $u\in H_{X,0}^{1}(\Omega)$, then $u^{+}, u^{-} \in H_{X,0}^{1}(\Omega)$, and
  \[ Xu^{+}=\left\{
              \begin{array}{ll}
                Xu, & \hbox{$u>0$;} \\
                0, & \hbox{$u\leq 0$.}
              \end{array}
            \right.\qquad Xu^{-}=\left\{
              \begin{array}{ll}
                0, & \hbox{$u>0$;} \\
                Xu, & \hbox{$u\leq 0$.}
              \end{array}
            \right. \]

\end{lemma}
 For any $a>0$, we let
\[ \begin{aligned}
  P^{\pm}&:=\{u\in H_{X,0}^{1}(\Omega):\pm u\geq 0 \},\\
  P^\pm _a&:=\{u\in H_{X,0}^{1}(\Omega):{\rm dist}(u, P^{\pm})<a\},
\end{aligned}
\]
where ${\rm dist}(u,P^{\pm})=\inf_{w\in P^\pm }\|u-w\|_{H_{X,0}^{1}(\Omega)}$. Then, we denote by
\[ P_a:=P^+_a\cup P^-_a,\qquad S_a:=H_{X,0}^{1}(\Omega)\setminus P_a.\]
We see that every $u\in S_a$ is  sign-changing. For any $R>0$, we define the ball
\[  B_R:=\{u\in H_{X,0}^{1}(\Omega): \|u\|_{H_{X,0}^{1}(\Omega)}<R\}.\]

\begin{definition}\label{def4-1}
Let $A$ be a subset of $H_{X,0}^{1}(\Omega)$. If $u\in A$ implies $-u\in A$, then we say $A$ is symmetric.
\end{definition}

To construct suitable critical values for sign-changing critical points within the context of H\"{o}rmander vector fields, a
more intricate intersection lemma is required. This intersection lemma is  based on the following result.

\begin{lemma}\label{lemma3-9}
For any $\rho>0$ and $a>0$, let
\[Z_{a}(\rho):=\overline{P_{a}}\cap \{u\in H_{X,0}^{1}(\Omega): \|u\|_{L^p(\Omega)}=\rho\},\]
where $\overline{P_{a}}$ is the closure of $P_{a}$ in $H_{X,0}^{1}(\Omega)$.
 Then there exists a constant $a(\rho)=C\rho>0$ such that $\gamma(Z_{a}(\rho))\leq 1$ for all $0<a<a(\rho)$, where $\gamma(A)$ is the Krasnoselskii genus (see \cite{Struwe2000,Rabinowitz1986}) of a closed symmetry set $A$, and $C>0$ is a positive constant independent of $\rho$.
\end{lemma}
\begin{proof}
 Let $\psi: Z_{a}(\rho)\to\mathbb{R}$ be the odd map given by
 \[\psi(u):={\rm dist}(u,P^+)-{\rm dist}(u,P^-).\]
   Observing that for any $w\in  P^{\mp}$,
\[ \int_{\Omega}|u^{\pm}|^pdx = \int_{\{x\in \Omega:\pm u\geq 0\}} |u|^pdx\notag\leq \int_{\{x\in \Omega:\pm u\geq 0\}} |u-w|^pdx\leq\int_{\Omega} |u-w|^pdx.\]
Since $u=u^{+}+u^-$, Proposition \ref{prop2-1} derives that
   \begin{equation}\label{3-45}
    \|u^{\pm}\|_{L^{p}(\Omega)}\leq\inf_{w\in  P^{\mp}}\|u-w\|_{L^{p}(\Omega)}\leq \widehat{C}_{p} \inf_{w\in P^{\mp}}\|u-w\|_{H_{X,0}^{1}(\Omega)}=\widehat{C}_{p} {\rm dist}(u, P^{\mp}).
  \end{equation}
Set  $a(\rho)=\frac{\rho}{2\widehat{C}_p}$. We claim that $\psi\neq 0$ on $Z_{a}(\rho)$  for any  $0<a<a(\rho)$. Note that
\[ \overline{P_{a}}=\overline{P^+_a}\cup \overline{P^-_a}=\{u\in H_{X,0}^{1}(\Omega): {\rm dist}(u, P^{+})\leq a~~\mbox{or}~~{\rm dist}(u, P^{-})\leq a\}. \]
If $\psi(u_0)=0$ for some $u_0\in Z_{a_{0}}(\rho)$ with $0<a_0<a(\rho)$, then ${\rm dist}(u_0,P^+)={\rm dist}(u_0,P^-)\leq a_{0}$.
By the definition of $Z_{a}(\rho)$ and \eqref{3-45},
  \[\rho=\|u_{0}\|_{L^p(\Omega)}\leq \|u_{0}^+\|_{L^p(\Omega)}+\|u_{0}^-\|_{L^p(\Omega)}\leq \widehat{C}_p{\rm dist}(u_{0},P^+)+\widehat{C}_p{\rm dist}(u_{0},P^-)\leq 2\widehat{C}_pa_{0},\]
  which contradicts the fact  $0<a_0<a(\rho)=\frac{\rho}{2\widehat{C}_p}$. Thus, by the definition of Krasnoselskii genus, we have $\gamma(Z_{a}(\rho))\leq1$ for all $0<a<a(\rho)$.
\end{proof}

\begin{lemma}[Intersection lemma]
\label{lemma3-10}
   For any fixed integer $k\geq 3$, there exist  $R_k^0>0$ and $a^0>0$, such that for any $R>R_k^0$ and any $0<a<a^0$ we have
  \[\phi(N_k\cap B_{R})\cap S_{a}\cap Q_{k-1}\neq \varnothing\]
  holds for any odd map $\phi \in C(N_k, H_{X,0}^{1}(\Omega))$ satisfying $\phi|_{N_k\cap B_{R}^c}=\mathbf{id }$. Here $a^0>0$ is a positive constant
 independent of $k$, $Q_{k-1}$ is the set defined in \eqref{3-40}, and $\mathbf{id}$ denotes the  identity mapping.
\end{lemma}
\begin{proof}
Using the min-max theorem and H\"{o}lder's inequality we have
\[ \|u\|_{H_{X,0}^{1}(\Omega)}\leq \sqrt{\lambda_{k}}\|u\|_{L^{2}(\Omega)}\leq  \sqrt{\lambda_{k}}|\Omega|^{\frac{1}{2}-\frac{1}{p}}\|u\|_{L^p(\Omega)}\qquad \forall u\in N_k. \]
 Define $R_k^0=2C_{k}\rho_{k-1}$, where $C_{k}:=\sqrt{\lambda_{k}}|\Omega|^{\frac{1}{2}-\frac{1}{p}}$ and $\rho_{k-1}=C_0
  \lambda_{k-1}^{\frac{2(2_{\tilde{\nu}}^{*}-p)}{p(p-2)(2_{\tilde{\nu}}^{*}-2)}}$ given by \eqref{3-40}. For any $R>R_k^0$  and  any odd map $\phi \in C(N_k, H_{X,0}^{1}(\Omega))$ satisfying $\phi|_{N_k\cap B_{R}^c}=\mathbf{id }$, we define:
  \[T:=\{u\in N_k\cap B_{R}:\|\phi(u)\|_{L^p(\Omega)}< \rho_{k-1}\}.\]
Consequently,  $T$ is a bounded symmetric open set containing the origin in $N_k$. According to the property of Krasnoselskii genus (see \cite[Proposition 5.2]{Struwe2000}), it follows that $\gamma(\partial T)=k$.

We then assert that for any $u\in \partial T$,  $\|\phi(u)\|_{L^p(\Omega)}=\rho_{k-1}$ and $\|u\|_{H_{X,0}^{1}(\Omega)}<R$. This implies that $\phi(\partial T)\cap N_{k-2}^\perp\subset Q_{k-1}$ and $\partial T\subset B_{R}\cap N_k$. Indeed, for $u\in \partial T$, there are only two possibilities:
\begin{itemize}
  \item $\|u\|_{H_{X,0}^{1}(\Omega)}=R$ and $\|\phi(u)\|_{L^p(\Omega)}\leq \rho_{k-1}$;
  \item $\|\phi(u)\|_{L^p(\Omega)}=\rho_{k-1}$ with $\|u\|_{H_{X,0}^{1}(\Omega)}< R$.
\end{itemize}
 If $\|u\|_{H_{X,0}^{1}(\Omega)}= R$ for some $u\in \partial T$, then by the definition of $\phi$ we have $\phi(u)=u$ and
  \[R_k^0=2C_{k}\rho_{k-1}<R\leq\|u\|_{H_{X,0}^{1}(\Omega)}\leq C_k\|u\|_{L^p(\Omega)}=C_k\|\phi(u)\|_{L^p(\Omega)}\leq C_k\rho_{k-1}\]
  which leads to a contradiction. Therefore, for any  $u\in \partial T$, we have $\|\phi(u)\|_{L^p(\Omega)}=\rho_{k-1}$ and $\|u\|_{H_{X,0}^{1}(\Omega)}<R$.

 \par
  Suppose that $\phi(\partial T)\cap S_a\cap Q_{k-1}=\varnothing$ for some fixed $a>0$. We consider the projection
\[ \mathcal{P}:H_{X,0}^{1}(\Omega)\to N_{k-2}.\]
   Since $\phi(\partial T)\cap N_{k-2}^\perp\subset Q_{k-1}$, it follows that $0\notin \mathcal{P}(\phi(\partial T)\cap S_a)$. Indeed, if $\mathcal{P}(y)=0$ for some $y\in \phi(\partial T)\cap S_a$, then $y\in N_{k-2}^\perp$ and $\phi(\partial T)\cap S_a\cap N_{k-2}^\perp\neq \varnothing$, which contradicts  $\phi(\partial T)\cap S_a\cap Q_{k-1}=\varnothing$. Observing that the restriction $\mathcal{P}|_{\phi(\partial T)\cap S_a}$ is odd and continuous on $\phi(\partial T)\cap S_a$, we have $\gamma(\phi(\partial T)\cap S_a)\leq k-2$. \par

   We next show that
   \begin{equation}\label{3-46}
     \phi(\partial T)\subset (\phi(\partial T)\cap S_a)\cup Z_{a}(\rho_{k-1}).
   \end{equation}
Indeed, for any $v\in \phi(\partial T)$, we have $\|v\|_{L^p(\Omega)}=\rho_{k-1}$, which implies that $v\in \{u\in H_{X,0}^{1}(\Omega): \|u\|_{L^p(\Omega)}=\rho_{k-1}\}$. Clearly, $v\in \phi(\partial T)\cap S_a$ provided $v\notin P_a$, and  $v\in Z_{a}(\rho_{k-1})$ if
 $v\in P_a$.

According to the property of Krasnoselskii genus and Lemma \ref{lemma3-9}, there exists $a(\rho_{k-1})>0$ such that for all $0<a<a(\rho_{k-1})$,
  \[k=\gamma(\partial T)\leq\gamma(\overline{\phi(\partial T)})=\gamma(\phi(\partial T))\leq\gamma(\phi(\partial T)\cap S_a)+\gamma(Z_{a}(\rho_{k-1}))\leq k-2+1=k-1,\]
  which gives a contradiction. Hence,  $\phi(\partial T)\cap Q_{k-1}\cap S_{a}\neq \varnothing$ for all $0<a<a(\rho_{k-1})$. Because of $\partial T\subset B_{R_k}\cap N_k$, we obtain
   \[\phi(N_k\cap B_{R})\cap S_{a}\cap Q_{k-1}\neq \varnothing,\qquad\forall R>R_k^0,~0< a<a(\rho_{k-1}).\]
Note that $\rho_{k}=C_0
  \lambda_{k}^{\frac{2(2_{\tilde{\nu}}^{*}-p)}{p(p-2)(2_{\tilde{\nu}}^{*}-2)}}$ given by \eqref{3-40} is increasing with respect to $k$. We can thus choose an $a^0:=a(\rho_2)=C\rho_2>0$ such that
  \[\phi(N_k\cap B_{R})\cap S_{a}\cap Q_{k-1}\neq \varnothing,\qquad\forall R>R_k^0,~0\leq a<a^{0}.\]
\end{proof}

We next construct the invariant properties for the operator $K_J$ determined by \eqref{3-19}.

\begin{lemma}\label{lemma3-11}
Under assumptions  $(f1)$-$(f5)$ and $(g1)$-$(g3)$, there exist some positive constants $M_2>M_1>0$ and $
0<\overline{a}<a^0$ such that for any
$0<a<\overline{a}$,
\[ K_J( u) \in P^\pm_\frac{a}{4} \]
holds for all $u\in P^\pm_a\cap\{u\in H_{X,0}^1(\Omega):J(u)>M_2\}$, where $K_J: H_{X,0}^{1}(\Omega)\to H_{X,0}^{1}(\Omega)$ is the operator defined in \eqref{3-19} above.

\end{lemma}
\begin{proof}
By Lemma \ref{lemma3-3}, we can choose a $M_2>0$ such that $0\leq -T_1(u)<\frac{1}{10}$ and $0<1+T_2(u)<2$ hold for
$J(u)>M_2$. For each $u\in P^\pm_a\cap\{u\in H_{X,0}^1(\Omega):J(u)>M_2\}$, we set $z=K_J(u)$.
 Using \eqref{3-19} we can write
  \[z=K_J(u)=w+(-T_1(u)u),\]
  where $w\in H_{X,0}^{1}(\Omega)$ is determined by
  \[(w,v)_{H_{X,0}^{1}(\Omega)}=(1+T_2(u))\int_{\Omega}f(x,u)v dx+\theta(u)\int_{\Omega}g(x,u)vdx,~~~\forall v\in H_{X,0}^{1}(\Omega).\]
Then we have
\begin{equation}\label{3-47}
   {\rm dist}(z, P^\pm)={\rm dist}(w+(-T_1(u)u),P^\pm)\leq {\rm dist}(w,P^\pm)+{\rm dist}(-T_1(u)u,P^\pm).
 \end{equation}
Since  $0\leq -T_1(u)<\frac{1}{10}$, we obtain $-T_1(u)y\in P^\pm$ for any $y\in P^\pm$, and therefore
  \[{\rm dist}(-T_1(u)u,P^\pm)\leq \|-T_1(u)u-(-T_1(u)y)\|_{H_{X,0}^{1}(\Omega)}=|T_1(u)|\|u-y\|_{H_{X,0}^{1}(\Omega)}.\]
   Taking the infimum of $y$ in $P^\pm$ we get
 \begin{equation}\label{3-48}
   {\rm dist}(-T_1(u)u, P^\pm)\leq\frac{1}{10}{\rm dist}(u,P^\pm)<\frac{a}{10}.
 \end{equation}
On the other hand,
 \begin{equation}\label{3-49}
  {\rm dist}(w,P^\mp)\leq \|w-w^\mp\|_{H_{X,0}^{1}(\Omega)}=\|w^{\pm}\|_{H_{X,0}^{1}(\Omega)}.
 \end{equation}
Thus by assumptions $(f1)$-$(f5)$ and $(g1)$-$(g3)$, \eqref{3-1}, \eqref{3-6}, \eqref{3-45} and \eqref{3-49} we have
\begin{align*}
  {\rm dist}(w,P^\mp)\|w^{\pm}\|_{H_{X,0}^{1}(\Omega)} &\leq \|w^{\pm}\|^2_{H_{X,0}^{1}(\Omega)}=(w,w^\pm)_{H_{X,0}^{1}(\Omega)}\\
  &=\int_{\Omega} ((T_2(u)+1)f(x,u)+\theta(u)g(x,u))w^{\pm}dx\\
  &\leq\int_{\Omega} ((T_2(u)+1)f^\pm(x,u)+\theta(u)g^\pm(x,u)) w^{\pm}dx\\
&=\int_{\Omega} ((T_2(u)+1)f(x,u^{\pm})+\theta(u)g(x,u^\pm))w^{\pm}dx\\
  &\leq 3\varepsilon\int_{\Omega}|u^{\pm}||w^{\pm}|dx+3C(\varepsilon)\int_{\Omega}|u^{\pm}|^{p-1}|w^{\pm}| dx\\
   &\leq3\varepsilon\|u^\pm\|_{L^2(\Omega)}\|w^\pm\|_{L^2(\Omega)}
    +3C(\varepsilon)\|u^\pm\|_{L^{p}(\Omega)}^{p-1}\|w^\pm\|_{L^{p}(\Omega)}\\
&\leq 3\varepsilon|\Omega|^{\frac{1}{2}-\frac{1}{p}}\|u^\pm\|_{L^p(\Omega)}\|w^\pm\|_{L^2(\Omega)}
    +3C(\varepsilon)\|u^\pm\|_{L^{p}(\Omega)}^{p-1}\|w^\pm\|_{L^{p}(\Omega)}\\
&\leq\left(\frac{1}{9}{\rm dist}(u,P^\mp)+C {\rm dist}(u,P^\mp )^{p-1} \right) \|w^{\pm}\|_{H_{X,0}^{1}(\Omega)}
  \end{align*}
holds for suitable $\varepsilon>0$. This means
\begin{equation}\label{3-50}
{\rm dist}(w,P^\pm)\leq \frac{1}{9}{\rm dist}(u,P^\pm)+C {\rm dist}(u,P^\pm )^{p-1} \leq \frac{1}{9}  a+Ca^{p-1}<\frac{a}{8}
\end{equation}
holds for all $0<a<\overline{a}$, where $0<\overline{a}<a^0$ is a  small constant. It follows from \eqref{3-47}-\eqref{3-50} that
\[{\rm dist}(z,P^\pm)<\frac{a}{10}+\frac{a}{8}<\frac{a}{4}\]
holds for all $0<a<\overline{a}$.
\end{proof}
\begin{remark}
\label{remark3-1}
In particular, if  $g(x,u)\equiv 0$, we have $T_{1}(u)=T_{2}(u)=0$, the  high energy level restriction $\{u\in H_{X,0}^{1}(\Omega): J(u)>M_{2}\}$ in Lemma \ref{lemma3-11} can be relaxed to the whole space $H_{X,0}^{1}(\Omega)$, and $K_{J}(P_{a}^{\pm})\subset P_{\frac{a}{4}}^{\pm}$ holds for all $0<a<\overline{a}$.
\end{remark}

\begin{lemma}\label{lemma3-12}
Let $J'(u)$ be the operator defined in \eqref{3-18}. For any $\varepsilon>0$, there exists a  locally Lipschitz continuous map $\tau: H_{X,0}^{1}(\Omega)\to H_{X,0}^{1}(\Omega)$ such that
\[ \|\tau(u)-J'(u)\|_{H_{X,0}^{1}(\Omega)}\leq\varepsilon\qquad \forall u\in H_{X,0}^{1}(\Omega).\]
\end{lemma}
\begin{proof}
 From \eqref{3-35} and \eqref{3-37}, we see that the operator $\kappa: H_{X,0}^{1}(\Omega)\to H_{X,0}^{1}(\Omega)$ given by
 \[ \kappa(u):=J'(u)-(1+T_1(u))u=(1+T_{2}(u))K_{1}(u)+\theta(u)K_{2}(u)\]
 maps the bounded sets to the relative compact sets. By \cite[Proposition A.23]{Rabinowitz1986},
   for any $\varepsilon>0$, there  exists a  locally Lipschitz continuous map $\kappa^0: H_{X,0}^{1}(\Omega)\to H_{X,0}^{1}(\Omega)$ satisfying
\[ \|\kappa^0(u)-\kappa(u)\|_{H_{X,0}^{1}(\Omega)}\leq\varepsilon\qquad \forall u\in H_{X,0}^{1}(\Omega). \]
 Let $\tau=\kappa^0+(1+T_1(\cdot))\mathbf{id}$. We can also deduce that $\tau$ is
locally Lipschitz continuous since $T_{1}$ is a $C^1$ functional with a locally bounded  Fr\'{e}chet derivative. Thus
\[ \|\tau(u)-J'(u)\|_{H_{X,0}^{1}(\Omega)}\leq\varepsilon\qquad \forall u\in H_{X,0}^{1}(\Omega).\]
\end{proof}

\begin{lemma}\label{lemma3-13}
For any given $a>0$, suppose there is  a locally Lipschitz continuous map $\tau$  on $H_{X,0}^1(\Omega)$ satisfying
\begin{equation}\label{3-51}
(\mathbf{id}-\tau)(\overline{P_{a}^{\pm}})\subset \overline{P_{\frac{a}{2}}^{\pm}}.
\end{equation}
Set $h(u):=\frac{1}{1+\|\tau(u)\|_{H_{X,0}^{1}(\Omega)}}$.
 For any given Lipschitz continuous function $\zeta$ in $H_{X,0}^{1}(\Omega)$ with $0\leq \zeta\leq 1$,  the following Cauchy problem
\begin{eqnarray}\label{3-52}
      \left\{
    \begin{aligned}
  \frac{d\eta(t,u)}{dt}&=-\zeta(\eta)h(\eta)\tau(\eta)
  , \\[1.5mm]
  \eta(0,u)&=u \in H_{X,0}^{1}(\Omega),
\end{aligned}
\right.
  \end{eqnarray}
 admits a global solution $\eta(t,u)\in C([0,+\infty)\times H_{X,0}^{1}(\Omega),H_{X,0}^{1}(\Omega))$. Moreover, if $u\in \overline{P_{a}^{\pm}}$, then $\eta(t,u)\in\overline{P_{a}^{\pm}}$ for all $t\geq 0$.
\end{lemma}

\begin{proof}
Observing that $\zeta(\cdot)h(\cdot)\tau(\cdot)$ is bounded and locally Lipschitz continuous on $H_{X,0}^{1}(\Omega)$, the basic existence-uniqueness theorem for ordinary differential equations (see \cite[Theorem A.4]{Rabinowitz1986}) implies that for each $u\in H_{X,0}^{1}(\Omega)$,  \eqref{3-52} admits a global unique solution $\eta(t,u)\in C([0,+\infty)\times H_{X,0}^{1}(\Omega), H_{X,0}^{1}(\Omega))$.

Let us suppose that the initial date $u\in \overline{P_{a}^{\pm}}$. If $\zeta(u)h(u)=0$, we have $\frac{d\eta(t,u)}{dt}|_{t=0}=0$. The uniqueness theorem of ordinary differential equation gives that
 $\eta(t,u)=u\in \overline{P_{a}^{\pm}}$ for all $t\geq 0$.
 In the case of $\zeta(u)h(u)>0$, by using Taylor expansion we
 have
 \[\eta(t,u)=\eta(0,u)+\frac{d\eta(t,u)}{dt}\bigg|_{t=0}t+\alpha_{u}(t) =u-\zeta(u)h(u)\tau(u)t+\alpha_u(t)\]
holds for small $t>0$,  where the reminder term $\alpha_{u}(t)$ satisfies $\lim_{t\to 0^+}\frac{\alpha_u(t)}{t}=0$ in $H_{X,0}^{1}(\Omega)$.
  Then for  small enough $t>0$ such that  $\xi:=\zeta(u)h(u)t\in (0,1)$, it follows that
\begin{equation}\label{3-53}
  \begin{aligned}
  &{\rm dist}(\eta(t,u), P^\pm)\\
  &={\rm dist}((1-\xi)u+\xi (u-\tau(u))+\alpha_u(t),P^\pm)\\
  &\leq{\rm dist}((1-\xi)u,P^\pm)+{\rm dist}(\xi(u-\tau(u)),P^\pm)+\|\alpha_u(t)\|_{H_{X,0}^{1}(\Omega)}
  \\
  &\leq (1-\xi){\rm dist}(u,P^\pm)+\xi{\rm dist}((u-\tau(u)),P^\pm)+\|\alpha_u(t)\|_{H_{X,0}^{1}(\Omega)}\\
  &\leq \left(1-\xi\right)a+\frac{a}{2}\xi+o(t)\leq \left(1-\frac{1}{2}\xi\right)a+o(t)<a,
  \end{aligned}
\end{equation}
which means $\eta(t,u)\in P_{a}^{\pm}$ for small $t>0$. Denote by $t(u):=\sup\{t_{0}|\eta(t,u) \in\overline{P_{a}^{\pm}},~~\forall t\in [0,t_0)\}$. If $t(u)<+\infty$,
the continuity of $\eta(\cdot,u)$ yields that $\eta(t(u),u)\in
\overline{P_{a}^{\pm}}$. Now,
  let $\widetilde{u}:=\eta(t(u),u)$ be the new initial data in \eqref{3-52}. Using the similar arguments above we can get a small  $\widetilde{t}>0$ such that
 \[\eta(t',\widetilde{u})=\eta(t',\eta(t(u),u))=\eta(t'+t(u),u)\in \overline{P_{a}^{\pm}}~~~\mbox{for all}~~0\leq t'<\widetilde{t},\]
 which contradicts the fact  $t(u)=\sup\{t_{0}|\eta(t,u) \in\overline{P_{a}^{\pm}},~~\forall t\in [0,t_0)\}$. Thus $t(u)=+\infty$, and $\eta(t,u)\in\overline{P_{a}^{\pm}}$ for all $t\geq 0$. The proof of Lemma \ref{lemma3-13} is now completed.
\end{proof}

\begin{lemma}\label{lemma3-14}
Let $a>\tilde{a}>0$ be some positive constants, and let $W$ be a subset in $H_{X,0}^{1}(\Omega)$. Suppose there is  a locally Lipschitz continuous map $\tau$  on $H_{X,0}^1(\Omega)$ satisfying
 \begin{equation}\label{3-54}
   (\mathbf{id}-\tau)(P_{a'}^{\pm}\cap W)\subset P_{\frac{a'}{2}}^{\pm}\qquad\forall a'\in(\tilde{a},a].
 \end{equation}
Then, for any Lipschitz continuous function $\zeta$ in $H_{X,0}^{1}(\Omega)$ with $0\leq \zeta\leq 1$ and $\zeta=0$ on $W^{c}$,  the  Cauchy problem \eqref{3-52} admits a global solution $\eta(t,u)\in C([0,+\infty)\times H_{X,0}^{1}(\Omega), H_{X,0}^{1}(\Omega))$. This solution ensures that for any given $u\in P_{a}^{\pm}\cap W$ and any $t_{0}\geq 0$, if $\eta(t_{0},u)\in W$, then
  $\eta(t_{0},u)\in P_{a}^{\pm}$.
\end{lemma}
\begin{proof}
From the proof of Lemma \ref{lemma3-13}, we can also deduce that \eqref{3-52} admits a global solution $\eta(t,u)\in C([0,+\infty)\times H_{X,0}^{1}(\Omega), H_{X,0}^{1}(\Omega))$. Observing that $\zeta=0$ on $W^{c}$, if $\eta(t_{0},u)\in W^{c}$ for some $t_{0}\geq 0$ and $u\in H_{X,0}^{1}(\Omega)$,
 the uniqueness theorem of ordinary differential equations implies
  \[ \eta(s,u)=\eta(t_{0},u)\in W^{c} \qquad \forall s>t_{0}.\]
Hence, if $u\in W$ and $\eta(t_{0},u)\in W$ for some $t_{0}\geq0$, then  $\eta(t,u)\in W$
 for all $t\in [0,t_{0}]$.

For any given $u\in P_{a}^{\pm}\cap W$, we may further restrict that
$u\in \overline{P_{a(u)}^{\pm}}\cap W$ for some $a(u)\in (\tilde{a},a)$. Note that \eqref{3-54} gives
$(\mathbf{id}-\tau)(P_{a(u)}^{\pm}\cap W)\subset P_{\frac{a(u)}{2}}^{\pm}$.
Using the same arguments in the proof of Lemma \ref{lemma3-13}, it follows that $\eta(t,u)\in\overline{P_{a(u)}^{\pm}}\subset P_a^{\pm}$ for all $t\geq 0$. \end{proof}

\begin{remark}
\label{remark3-2}
For any $M>M_2$ and any $a'\in (\tilde{a},\overline{a})$,
 if $W=\{u\in H_{X,0}^1(\Omega):J(u)>M\}$,  by Lemma \ref{lemma3-11} we have $K_J(P_{a'}^{\pm}\cap W)\subset P_{\frac{a'}{4}}^{\pm}$. In addition, if we choose $0<\varepsilon<\frac{\tilde{a}}{4}$ in Lemma \ref{lemma3-12},  there exists a  locally Lipschitz continuous map $\tau: H_{X,0}^{1}(\Omega)\to H_{X,0}^{1}(\Omega)$  such that
\begin{equation}\label{3-56}
 (\mathbf{id}-\tau)(P_{a'}^{\pm}\cap W)=(K_{J}+J'-\tau)(P_{a'}^{\pm}\cap W)\subset P_{\frac{a'}{2}}^{\pm}~~~~\forall a'\in (\tilde{a},\overline{a}).
\end{equation}

In particular, if $g(x,u)\equiv 0$ and $J'$ is locally Lipschitz continuous, by Remark \ref{remark3-1}
 we have $K_J(\overline{P_{a'}^{\pm}})\subset \overline{P_{\frac{a'}{4}}^{\pm}}$ for all $a'\in (0, \overline{a})$. Taking $\tau=J'$ in Lemma \ref{lemma3-13}, we obtain that $\eta(t, \overline{P_{a}^{\pm}})\subset \overline{P_{a}^{\pm}}$ holds for all $a\in (0, \overline{a})$ and $t\geq 0$.
\end{remark}

\subsection{The min-max values for perturbation functional $J$}
\label{subsection3-3}
In this subsection, we introduce two min-max values of the perturbation functional $J$. For this purpose, we initially define two increasing sequences $\{R_{k}\}_{k=3}^{\infty}$ and $\{a_{k}\}_{k=3}^{\infty}$ as follows.

Beginning from $k=3$, Lemma \ref{lemma3-6} enables us to select a positive constant $R_3>R_3^0$, ensuring that $J(u)<-1$ for any $u\in N_3\cap B_{R_3}^c$, where $R_3^0$ is determined by Lemma \ref{lemma3-10}. Then, for $k=4$, we may similarly choose $R_{4}>\max\{R_{4}^{0},R_{3}\}$ so that $J(u)<-1$ for $u\in N_{4}\cap B_{R_{4}}^c$. By iterating this procedure, we derive a strictly increasing sequence $\{R_{k}\}_{k=3}^{\infty}$ satisfying $R_{k}>R_{k}^{0}$ and $J(u)<-1$ for any $u\in N_{k}\cap B_{R_{k}}^{c}$.

Let $\overline{a}>0$ denote the positive constant specified in Lemma \ref{lemma3-11}. For any $0<a_{3}<\frac{\overline{a}}{2}$, we define a sequence $\{a_{k}\}_{k=3}^{\infty}$ by setting
\begin{equation}\label{3-57}
a_{k+1}=a_k+\frac{\overline{a}}{10^k}~~~~\mbox{for}~~k\geq 3.
\end{equation}
Clearly, $0<a_k<\overline{a}$ for any $k\geq 3$.

By above sequences $\{R_{k}\}_{k=3}^{\infty}$ and $\{a_{k}\}_{k=3}^{\infty}$,
 for each $k\geq 3$ we may set
\[\Gamma_k:=\{ \phi \in C(N_k, H_{X,0}^{1}(\Omega)):\phi~\text{is odd},~\phi|_{N_k\cap B_{R_k}^c}=\mathbf{id } \},\]
and
\begin{equation}\label{3-58}
  b_k(J):=\inf\limits_{\phi\in \Gamma_k}\sup_{\phi(N_k\cap B_{R_k})\cap S_{a_k}}J(u).
\end{equation}
Note that $\mathbf{id } \in \Gamma_k$ and Lemma \ref{lemma3-10} gives
that $\phi(N_k\cap B_{R_k})\cap S_{a}\neq \varnothing$ for any $\phi\in \Gamma_k$  and any $0<a<\overline{a}$. Thus, the min-max value $b_k(J)$ is well-defined.

Moreover, we define
\[\Lambda_k:=\Big\{\phi \in C(N^{+}_{k+1}, H_{X,0}^{1}(\Omega)): \phi ~\text{is odd in}~ N_k\cap B_{R_{k}},~\phi|_{U}=\mathbf{id},~\sup_{\phi(N_k\cap B_{R_k})\cap S_{a_k}}J(u)\leq b_k(J)+1\Big\},\]
where
\begin{equation}\label{3-59}
U:=((B_{R_{k+1}}\setminus B_{R_{k}})\cap N_k)\cup (N^{+}_{k+1}\cap B_{R_{k+1}}^c ),
\end{equation}
and $N^{+}_{k+1}:=
\{u+t\varphi_{k+1}:u \in N_k,~t \geq 0\}$. The
Dugundji extension theorem in \cite{Dugundji1951} yields that $\Lambda_{k}\neq \varnothing$. In addition, $N_{k}\cap B_{R_{k}}^{c}\subset U$ implies $\Lambda_k\subset \Gamma_k$. Since $a_{k+1}<\bar{a}<a^0$, we can obtain
$\phi(N^{+}_{k+1}\cap B_{R_{k+1}})\cap S_{a_{k+1}}\neq \varnothing$ due to
\[\phi(N_{k}\cap B_{R_{k}})\cap S_{a_{k+1}}\subset \phi(N^{+}_{k+1}\cap B_{R_{k+1}})\cap S_{a_{k+1}}\] and Lemma \ref{lemma3-10}. Therefore, the  min-max value
\[c_k(J):=\inf_{{\phi\in \Lambda_k}}\sup_{\phi(N^{+}_{k+1}\cap B_{R_{k+1}})\cap S_{a_{k+1}}}J(u)\]
is well-defined.

The min-max values $b_{k}(J)$ and $c_{k}(J)$ satisfy the following lower bound depending on the Dirichlet eigenvalue of $-\triangle_{X}$.
\begin{lemma}\label{lemma3-15}
For all $k\geq 3$, we have
\begin{equation}\label{3-60}
  b_k(J)\geq C \lambda_{k-1}^{\frac{2(2_{\tilde{\nu}}^{*}-p)}{(p-2)(2_{\tilde{\nu}}^{*}-2)}}
\end{equation}
and
\begin{equation}\label{3-61}
  c_k(J)\geq C \lambda_{k-1}^{\frac{2(2_{\tilde{\nu}}^{*}-p)}{(p-2)(2_{\tilde{\nu}}^{*}-2)}}.
\end{equation}
Thus $b_k(J),c_{k}(J)\to+\infty$ as $k\to \infty$.
\end{lemma}
\begin{proof}
Since $a_k<a_{k+1}<\overline{a}<a^0$, by Lemma \ref{lemma3-7} and Lemma \ref{lemma3-10} we have
\begin{align*}
  \sup_{\phi(N_k\cap B_{R_k})\cap S_{a_k}}J(u)\geq\sup_{\phi(N_{k}\cap B_{R_{k}})\cap S_{a_{k}}\cap Q_{k-1}}J(u)\geq\inf_{\phi(N_{k}\cap B_{R_{k}})\cap S_{a_{k}}\cap Q_{k-1}}J(u)\geq C \lambda_{k-1}^{\frac{2(2_{\tilde{\nu}}^{*}-p)}{(p-2)(2_{\tilde{\nu}}^{*}-2)}}
\end{align*}
and
\begin{align*}
  \sup_{\phi(N^{+}_{k+1}\cap B_{R_{k+1}})\cap S_{a_{k+1}}}J(u)\geq\sup_{\phi(N_{k}\cap B_{R_{k}})\cap S_{a_{k+1}}\cap Q_{k-1}}J(u)\geq C \lambda_{k-1}^{\frac{2(2_{\tilde{\nu}}^{*}-p)}{(p-2)(2_{\tilde{\nu}}^{*}-2)}},
\end{align*}
which yields \eqref{3-60} and \eqref{3-61}.
\end{proof}

Lemma \ref{lemma3-15} allows us to select an integer $k_0\geq 3$ such that
\[  b_{k}(J)\geq M_{2}+1~~~\mbox{and}~~~ c_{k}(J)\geq M_{2}+1~~\mbox{for all}~~k\geq k_0, \]
where $ M_{2}$ is the positive constant appeared in Lemma \ref{lemma3-11}.
Then we have
\begin{proposition}\label{prop3-1}
If $c_{k}(J)>b_{k}(J)+1$ for $k\geq k_0$, then
   $c_k(J)$ is a sign-changing critical value of $J$, i.e.,
    there exists a sign-changing critical point of $J$ with critical value $c_k(J)$.
\end{proposition}

\begin{proof}
Let us fix the integer $k\geq k_{0}$ such that $c_{k}(J)>b_{k}(J)+1>M_{2}+2$. If for any $0<\frac{1}{l}<\min\{\frac{c_k(J)-b_k(J)-1}{2}, a_3,\frac{1}{10}\}$ with $l\in \mathbb{N}^{+}$, there exists $u_{l}\in \{u\in S_{a_{k+1}}:|J(u)-c_k(J)| \leq \frac{2}{l}\}$ such that
\[ \|J'(u_{l})\|_{H_{X,0}^{1}(\Omega)} < \frac{1}{l}, \]
then $\{u_{l}\}_{l=1}^{\infty}$
is a Palais-Smale sequence of $J$. From Lemma \ref{lemma3-5} we see that $J$ satisfies Palais-Smale condition on
$\widehat{A}_{M_{1}}=\{u\in H_{X,0}^{1}(\Omega):J(u)\geq M_{1}\}$. Therefore, $\{u_{l}\}_{l=1}^{\infty}$
converges to a sign-changing critical point $u_{0} \in\{u\in S_{a_{k+1}}:J(u)=c_k(J)\}$ along a subsequence $\{u_{l_{j}}\}_{j=1}^{\infty}\subset \{u_{l}\}_{l=1}^{\infty}$, which yields Proposition \ref{prop3-1}. We shall establish the existence of such a Palais-Smale sequence $\{u_{l}\}_{l=1}^{\infty}$ by contradiction.

 Suppose there exists $0<\varepsilon<\min\{\frac{c_k(J)-b_k(J)-1}{2}, a_3,\frac{1}{10}\}$ such that $ \|J'(u)\|_{H_{X,0}^{1}(\Omega)} \geq \varepsilon$ holds for
 any
$u\in Y$, where
\begin{equation}\label{3-62}
  Y:=\{u\in S_{a_{k+1}}:|J(u)-c_k(J)| \leq 2\varepsilon\}.
\end{equation}
We will derive a contradiction through several steps as follows.\par
\noindent\textbf{Step 1}. Let
\begin{equation}\label{3-63}
A:=\left\{u\in H_{X,0}^{1}(\Omega): \|J'(u)\|_{H_{X,0}^1(\Omega)}\leq \frac{\varepsilon}{2}\right\}\cup \left\{u\in H_{X,0}^{1}(\Omega):|J(u)-c_k(J)| \geq 2\varepsilon\right\}\cup P_{a_k},
\end{equation}
and
\begin{equation}\label{3-64}
B:=\{u\in S_{a_{k+1}}: |J(u)-c_k(J)|\leq\varepsilon\}.
\end{equation}
Clearly, $B\subset Y$ and  $\|J'(u)\|_{H_{X,0}^1(\Omega)}\geq \varepsilon$ for all $u\in B$. This implies that $\overline{A}\cap \overline{B}=\varnothing$. By Lemma \ref{lemma3-12}, we find a  locally Lipschitz continuous map $\tau: H_{X,0}^{1}(\Omega)\to H_{X,0}^{1}(\Omega)$ such that
\begin{equation}\label{3-65}
  \|J'(u)-\tau(u)\|_{H_{X,0}^{1}(\Omega)}
<\frac{\varepsilon}{8}<\frac{a_{3}}{8}\qquad \forall u\in H_{X,0}^{1}(\Omega).
\end{equation}
Let $W=\{u\in H_{X,0}^1(\Omega):J(u)>M_{2}+1\}$ and $\zeta(u)=\frac{\text{dist} (u,A)}{\text{dist} (u,A)+\text{dist} (u,B)}$. It follows that $0\leq \zeta\leq 1$ is
a Lipschitz continuous function in $H_{X,0}^{1}(\Omega)$ satisfying
\[ \zeta(u)=\left\{
              \begin{array}{ll}
                1, & \hbox{$u\in B$;} \\
                0, & \hbox{$u\in A$.}
              \end{array}
            \right.\]
In addition, $W^{c}\subset A$ and $\zeta(u)=0$ for all $u\in W^{c}$.

For any $a'\in (\frac{a_{3}}{2}, \overline{a})$, by Remark \ref{remark3-2} and \eqref{3-65} we have
\[ (\mathbf{id}-\tau)(P_{a'}^{\pm}\cap W)=(K_{J}+J'-\tau)(P_{a'}^{\pm}\cap W)\subset P_{\frac{a'}{2}}^{\pm}.\]
In particular,
\begin{equation}\label{3-66}
  (\mathbf{id}-\tau)(P_{a'}^{\pm}\cap W)=(K_{J}+J'-\tau)(P_{a'}^{\pm}\cap W)\subset P_{\frac{a'}{2}}^{\pm}\quad \forall a'\in \left(\frac{a_{3}}{2}, a_{k+1}\right].
\end{equation}
Now, applying Lemma \ref{lemma3-14} to $a_{k+1}>\frac{a_{3}}{2}>0$, $W$, $\tau$ and $\zeta$,
 we see that \eqref{3-52}  admits a global  solution $\eta(t,u)\in C([0,+\infty)\times H_{X,0}^{1}(\Omega), H_{X,0}^{1}(\Omega))$. Furthermore,  for any  $u\in P_{a_{k+1}}^{\pm}\cap W$ and any $t_{0}\geq 0$, if $\eta(t_{0},u)\in W$, then $\eta(t_{0},u)\in P_{a_{k+1}}^{\pm}$.

\noindent
\textbf{Step 2}. Analyze the derivative of $J(\eta(t,u))$ at $t=0$. For any $u\in H_{X,0}^{1}(\Omega)$, by \eqref{3-52} we obtain
\begin{equation}\label{3-67}
  \begin{aligned}
    \frac{dJ(\eta(t,u))}{dt}\Big|_{t=0}& =\left(J'(u),\frac{d\eta(t,u)}{dt}\Big|_{t=0}\right)_{H_{X,0}^{1}(\Omega)}\\
  & =-C_{u}\left(J'(u),\tau(u)\right)_{H_{X,0}^{1}(\Omega)},\\
\end{aligned}
\end{equation}
where $0\leq C_{u}=\zeta(u)h(u)\leq 1$. If $u\in A$, then $\zeta(u)=0$ gives $C_{u}=0$. For
$u\notin A$, \eqref{3-63} and \eqref{3-65} derive that $\|J'(u)\|_{H_{X,0}^{1}(\Omega)}> \frac{\varepsilon}{2}$, and
\begin{equation}\label{3-68}
\|\tau(u)\|_{H_{X,0}^{1}(\Omega)}\geq \|J'(u)\|_{H_{X,0}^{1}(\Omega)}-\|J'(u)-\tau(u)\|_{H_{X,0}^{1}(\Omega)}\geq \frac{3\varepsilon}{8}.
\end{equation}
This means
\begin{equation}\label{3-69}
\begin{aligned}
(J'(u),\tau(u))_{H_{X,0}^{1}(\Omega)}&=\|\tau(u)\|_{H_{X,0}^{1}(\Omega)}^2+(J'(u)-\tau(u),\tau(u))_{H_{X,0}^{1}(\Omega)}\\
&\geq \|\tau(u)\|_{H_{X,0}^{1}(\Omega)}(\|\tau(u)\|_{H_{X,0}^{1}(\Omega)}-\|J'(u)-\tau(u)\|_{H_{X,0}^{1}(\Omega)})\\
&\geq
\|\tau(u)\|_{H_{X,0}^{1}(\Omega)}\left(\frac{3\varepsilon}{8}-\frac{\varepsilon}{8}\right)= \frac{\varepsilon}{4}\|\tau(u)\|_{H_{X,0}^{1}(\Omega)}.
\end{aligned}
\end{equation}
Combining \eqref{3-67}-\eqref{3-69} we have
\begin{equation*}
  \frac{dJ(\eta(t,u))}{dt}\Big|_{t=0}\leq 0 \qquad \mbox{for all}~~ u\in H_{X,0}^{1}(\Omega).
\end{equation*}
Furthermore, for any $t_0\geq 0$, we can deduce from \eqref{3-52} that
\begin{equation}\label{3-70}
 \frac{dJ(\eta(t,u))}{dt}\Big|_{t=t_{0}}
 =\frac{dJ(\eta(t,\eta(t_{0},u)))}{dt}\Big|_{t=0}\leq 0\qquad \mbox{for all}~~ u\in H_{X,0}^{1}(\Omega).
\end{equation}
\noindent
\textbf{Step 3}.
By the definition of $c_k(J)$, we can choose a map $\phi_0\in \Lambda_k$ such that
\begin{equation}\label{3-71}
  \sup_{\phi_0(N^{+}_{k+1}\cap B_{R_{k+1}})\cap S_{a_{k+1}}}J(u)\leq c_k(J)+\varepsilon.
\end{equation}
We then prove that
  \begin{equation}\label{3-72}
  \sup_{\eta(T,\phi_0(N^{+}_{k+1}\cap B_{R_{k+1}}))\cap S_{a_{k+1}}}J(u)\leq c_k(J)-\varepsilon\qquad\mbox{for}~~T=\frac{64}{\varepsilon}.
\end{equation}

Suppose that $J(u)> c_k(J)-\varepsilon$ for some $u\in \eta(T,\phi_0(N^{+}_{k+1}\cap B_{R_{k+1}}))\cap S_{a_{k+1}}$, we have
$u=\eta(T,v)\in S_{a_{k+1}}$ for some $v\in \phi_{0}(N^{+}_{k+1}\cap B_{R_{k+1}})$. From \eqref{3-70}  we deduce that for all $t\in[0, T]$,
\begin{equation}\label{3-73}
  J(\eta(t,v))\geq J(u)=J(\eta(T,v))> c_k(J)-\varepsilon>M_{2}+1,
\end{equation}
i.e. $\eta(t,v)\in W=\{u\in H_{X,0}^1(\Omega):J(u)>M_{2}+1\}$ for all $t\in[0, T]$. Thus, by Step 1 we have
$\eta(t,v)\in S_{a_{k+1}}$ holds for all $t\in[0, T]$. Indeed, if $\eta(t_{0},v)\in P_{a_{k+1}}^{\pm}$ for some $t_{0} \in[0, T]$, it follows that $u=\eta(T,v)=\eta(T-t_{0},\eta(t_{0},v))\in P_{a_{k+1}}^{\pm}$, which contradicts $u=\eta(T,v)\in S_{a_{k+1}}$. In particular, $v\in S_{a_{k+1}}$. Owing to \eqref{3-71} and \eqref{3-73}, we have for $0\leq t\leq T$,
\begin{equation}\label{3-74}
c_{k}(J)-\varepsilon\leq J(\eta(t,v))\leq J(v)\leq  \sup_{\phi_0(N^{+}_{k+1}\cap B_{R_{k+1}})\cap S_{a_{k+1}}}J(u)\leq c_{k}(J)+\varepsilon,
\end{equation}
which implies that
 $\eta(t,v)\in B$ holds for all $0\leq t\leq  T$.

From \eqref{3-67}-\eqref{3-69}, we have
\begin{equation}\label{3-75}
\begin{aligned}
  J(u)-J(v)&=J(\eta(T,v))-J(\eta(0,v))\\
&=\int_{0}^{T}\frac{dJ(\eta(t,v))}{dt} dt\\
&\leq -\int_{0}^{T}\frac{\varepsilon C_{\eta(t,v)}}{4} \|\tau(\eta(t,v)) \|_{H_{X,0}^{1}(\Omega)} dt\\
  &=-\int_{0}^{T} \frac{\|\tau(\eta(t,v))\|_{H_{X,0}^{1}(\Omega)}}{1+\|\tau(\eta(t,v))\|_{H_{X,0}^{1}(\Omega)}}
   \frac{\varepsilon}{4}dt\\
   &\leq -T \left(\frac{\varepsilon}{4}\right)^2
   \leq -4\varepsilon
  .
  \end{aligned}
\end{equation}
The last inequality is due to  $\|\tau(u)\|_{H_{X,0}^{1}(\Omega)}\geq \frac{3\varepsilon}{8}$ for all $u\in B\subset A^{c}$. By \eqref{3-74} and \eqref{3-75}, we derive that $J(u)\leq J(v)-4\varepsilon\leq c_{k}(J)-3\varepsilon$, which contradicts  $J(u)\geq c_k(J)-\varepsilon$ in \eqref{3-73}. Thus, we achieved \eqref{3-72}.

\noindent
\textbf{Step 4}.  We finally prove that $\eta(t,\phi_0(\cdot))\in \Lambda_k$ for any $t\geq 0$, which contradicts  \eqref{3-72} by the definition of $c_k(J)$, and Proposition \ref{prop3-1} follows immediately.

Since $\phi_{0}\in \Lambda_k$, we have $\eta(t,\phi_0(\cdot))\in C(N^{+}_{k+1},H_{X,0}^{1}(\Omega))$ for any $t\geq 0$. Suppose $u\in U$, the construction of $R_{k}$ implies that $J(u)<0$. By \eqref{3-70}, for any $t\geq 0$ we have
\[J(\eta(t,\phi_0(u)))\leq J(\eta(0,\phi_0(u)))=J(\phi_0(u))=J(u)<0<c_k(J)-2\varepsilon,\]
which implies $\eta(t,\phi_0(u))\in A$ for any $t\geq 0$. Thus $\frac{d\eta(t,\phi_0(u))}{dt}=0$ and $\eta(t,\phi_0(u))=\phi_0(u)=u$ for any $t\geq 0$.

Next, we verify that
\begin{equation}\label{3-76}
  \sup_{\eta(t,\phi_0(N_k\cap B_{R_k}))\cap S_{a_k}}J(u)\leq b_k(J)+1~~~\mbox{for all}~~ t\geq 0.
\end{equation}
Assume that $u\in  \eta(t_{0},\phi_0(N_k\cap B_{R_k}))\cap S_{a_k}$ for some $t_{0}>0$ such that $J(u)>b_k(J)+1>M_{2}+2$. Then  $u=\eta(t_{0},v)\in S_{a_k}$ for some
 $v\in \phi_{0}(N_k\cap B_{R_k})$. Moreover, \eqref{3-70} implies
\begin{equation}\label{3-77}
M_{2}+2<b_{k}(J)+1<J(\eta(t_{0},v))=J(u)\leq J(v),
\end{equation}
which yields $v\in W$ and $\eta(t_{0},v)\in W$. By Step 1, we can also deduce that $v\in S_{a_{k}}$. Thus,
\[J(\eta(t_{0},v))\leq J(v)\leq\sup_{\phi_0(N_k\cap B_{R_k})\cap S_{a_k}}J(u)\leq b_k(J)+1,\]
which contradicts  \eqref{3-77}.

It remains to verify that $\eta(t,\phi_0(\cdot))$ is odd in $N_k\cap B_{R_{k}}$ for any $t\geq 0$. For $w\in N_k\cap B_{R_{k}}$, if $\phi_0(w)\in P_{a_k}\subset A$, then $\eta(t,\phi_0(w))=\phi_0(w)$ for any $t\geq 0$.  In the case of  $\phi_0(w)\in S_{a_k}$, recalling that $2\varepsilon<c_k(J)-b_k(J)-1$, we have
      \[J(\phi_0(w))\leq\sup_{\phi_0(N_k\cap B_{R_k})\cap S_{a_k}}J(u)\leq b_k(J)+1<c_k(J)-2\varepsilon,\]
      which implies $\phi_0(w)\in A$. Hence, $\eta(t,\phi_0(w))=\phi_0(w)$ holds for any $t\geq 0$ and any $w\in N_k\cap B_{R_{k}}$.  Consequently, for any $t\geq 0$ we have $\eta(t,\phi_0(\cdot))$ is odd in $N_k\cap B_{R_{k}}$  and  $\eta(t,\phi_0(\cdot))\in \Lambda_k$.
\end{proof}

\section{Estimates on augmented Morse index}
\label{Section4}

By \eqref{3-2}, \eqref{3-6} and Proposition \ref{prop2-1}, we can
choose  suitable positive constants $\varepsilon>0$ and $C_{p}>0$ such that
 \begin{equation}\label{4-4}
 \begin{aligned}
   J(u) &=\frac{1}{2}\int_{\Omega}|Xu|^2dx-\int_{\Omega}F(x,u)dx-\theta(u)\int_{\Omega}G(x,u)dx\\
   &\geq \frac{1}{2}\int_{\Omega}|Xu|^2dx-\int_{\Omega}\varepsilon|u|^2dx-\int_{\Omega}C(\varepsilon)|u|^p\\
    & \geq\frac{1}{2}\left(\frac{1}{2}\int_{\Omega}|Xu|^2dx-C_p\int_{\Omega}|u|^pdx\right).\\
 \end{aligned}
 \end{equation}

In this section, we are concerned with the augmented Morse index of the sign-changing critical points for following auxiliary functional
\begin{equation}\label{4-5}
I_p(u):= \frac{1}{2}\int_{\Omega}|Xu|^2dx-C_p\int_{\Omega}|u|^{p}dx.
\end{equation}
Clearly, if we set $f(x,u)=pC_p|u|^{p-2}u$ and $g(x,u)\equiv 0$, the functional $J$ reduces to $I_p$.

To proceed our arguments, we present some abstract results related to the even functional.
\subsection{Deformation lemma and relative genus}
Given an even functional $I\in C^{2}(H_{X,0}^{1}(\Omega),\mathbb{R})$. For any $c\in \mathbb{R}$, we adopt the notations
\[I^c:=\{u\in H_{X,0}^{1}(\Omega): I(u)\leq c\}\]
and
\[K_{I,c}:=\{u\in H_{X,0}^1(\Omega): I'(u)=0,~I(u)=c\}.\]
Clearly, both $I^c$ and $K_{I,c}$ are symmetric and closed. The second order
Fr\'{e}chet derivative $I''(u)$ belongs to $\mathcal{L}(H_{X,0}^{1}(\Omega),H_{X,0}^{1}(\Omega))$, where $\mathcal{L}(H_{X,0}^{1}(\Omega),H_{X,0}^{1}(\Omega))$ denotes the collection of all bounded linear operators from $H_{X,0}^{1}(\Omega)$ to $H_{X,0}^{1}(\Omega)$. Additionally, $I''(u):h\mapsto I''(u)(h)$ is self-adjoint.
The Morse index and augmented Morse index on the critical points of functional  $I$ are defined as follows:
\begin{definition}\label{def4-3}
 Let $u\in H_{X,0}^{1}(\Omega)$ be the critical point of $I$. Denote by
 $V$ the subspace of $ H_{X,0}^{1}(\Omega)$.
The  Morse index $m(u,I)$ at $u$ is given by
 \[m(u,I):=\max\{{\rm dim}~V: (I''(u)(v),v)_{H_{X,0}^1(\Omega)}< 0,~\forall v\in V\setminus\{0\}\}.\]
In particular, the augmented Morse index $m^*(u, I)$ at $u$ is defined by
 \[m^*(u,I):=\max\{{\rm dim}~V: (I''(u)(v),v)_{H_{X,0}^1(\Omega)}\leq0,~\forall v\in V\}.\]
 Additionally, we say the critical point $u$ of $I$ is  non-degenerate if $I''(u)$ is invertible.
\end{definition}
\begin{remark}
\label{remark4-1-2}
Clearly, we have $m(u,I)\leq m^{*}(u,I)$. In particular, if the critical point $u$ of $I$ is  non-degenerate, then $m(u,I)=m^{*}(u,I)$.
\end{remark}

Associated with the sign-changing settings, we can obtain the following results.
\begin{proposition}\label{prop4-2}
Let $I\in C^{2}(H_{X,0}^{1}(\Omega),\mathbb{R})$ be an even functional satisfying the Palais-Smale condition. For any given $a>0$, assume that $I':H_{X,0}^{1}(\Omega)\to H_{X,0}^{1}(\Omega)$ is a locally Lipschitz continuous map satisfying
\begin{equation}\label{4-3}
(\mathbf{id}-I')(\overline{P_{a}^{\pm}})\subset\overline{P_{\frac{a}{2}}^{\pm}}.
\end{equation}
Then for any $c\in\mathbb{R}$, any symmetric neighborhood $N$ of $K_{I,c}\cap S_a$ and any $\varepsilon_1>0$,  there exist $\delta\in(0,\varepsilon_1)$ and an odd  homeomorphism $\Theta:H_{X,0}^{1}(\Omega)\to H_{X,0}^{1}(\Omega)$ such that
  \begin{enumerate}[(i)]
  \item $\Theta((I^{c+\delta}\cup \overline{P_a})\setminus N)\subset I^{c-\delta}\cup \overline{P_a}$;
     \item $\Theta(\overline{P_a^{\pm}})\subset \overline{P_a^{\pm}}$;
     \item $\Theta(u)=u,~ \forall u\in I^{c-\varepsilon_1}$;
     \item $I(\Theta(u))\leq I(u),~ \forall u\in H_{X,0}^{1}(\Omega)$.
   \end{enumerate}
\end{proposition}
\begin{proof}
Taking $N':=N\cup P_a$, it follows that $N'$ is a symmetric neighborhood of $K_{I,c}$. By the classic deformation lemma in \cite[Theorem A.4]{Rabinowitz1986}, there exist $\delta\in(0,\varepsilon_1)$  and an odd homeomorphism $\Theta\in C(H_{X,0}^{1}(\Omega),H_{X,0}^{1}(\Omega))$ satisfying (iii), (iv), and $\Theta(I^{c+\delta}\setminus N')\subset I^{c-\delta}$. More precisely, $\Theta$ can be formulated by $\Theta(u)=\eta(1,u)$, where $\eta(t,u)$ is the solution of \eqref{3-52} associated with suitable function $\zeta$.

Taking $\tau=I'$  in
     Lemma \ref{lemma3-13}, we derive from \eqref{4-3} that $\Theta(\overline{P_a^{\pm}})\subset \overline{P_a^{\pm}}$, and (ii) follows.
      Thus,
     \[\Theta((I^{c+\delta}\setminus N')\cup\overline{P_a})\subset I^{c-\delta}\cup\overline{P_a}.\]
    On the other hand, by $N'=N\cup P_a$ we have
\[(I^{c+\delta}\cup\overline{P_a})\setminus N\subset (I^{c+\delta}\setminus N)\cup\overline{P_a}\subset(I^{c+\delta}\setminus N')\cup\overline{P_a},\]
    which gives (i).
\end{proof}

\begin{lemma}\label{lemma4-1}
Under the assumptions of Proposition \ref{prop4-2}, we have $\partial P_a \cap K_{I,c}=\varnothing$ holds for any $c\in \mathbb{R}$.
\end{lemma}
\begin{proof}
 If $u\in \partial P_a\cap K_{I,c}$, we have $u-I'(u)=u$. According to \eqref{4-3}, it follows that  $u\in \overline{P_{\frac{a}{2}}}$, which yields a contradiction.
\end{proof}

We next introduce the concepts of relative genus, which will be used to define the min-max value of $I_p$.

\begin{definition}[cf. \cite{Bartsch2005}]
\label{def4-2}
  For symmetric and closed subsets $ A \subset B \subset C $ in $H_{X,0}^{1}(\Omega)$, we define the genus of $C$ relative to the pair $(B, A)$, denoted by $\gamma(C; B, A)$, to be the smallest non-negative integer
$k\geq 0 $ such that there exist closed and symmetric subsets $U , V\subset H_{X,0}^{1}(\Omega)$ with
\begin{enumerate}[(1)]
\item
$C\subset U \cup V, B \subset U,$ and $\gamma(V)\leq k$;
\item
There is an odd and continuous map $h: U\to B$ with $h(A)\subset A$.
\end{enumerate}
If no such $k$ exists we set $\gamma(C; B, A)=+\infty$.
\end{definition}
\begin{remark}\label{remark4-1}
Note that $\gamma(B; B, A)=0$ for every closed and symmetric subsets $A\subset B$, and the usual genus is contained in the above definition via the relation  $\gamma(C)=\gamma(C; \varnothing, \varnothing)$.
\end{remark}

The relative genus admits the following properties.

\begin{lemma}[cf. {\cite[Proposition 4.2]{Bartsch2005}}]
\label{prop4-1}
  Assume that $ A \subset B \subset C $ are closed and symmetric subsets in $H_{X,0}^{1}(\Omega)$. We have
  \begin{enumerate}[(1)]
    \item If there exist closed and symmetric subsets  $C_0$ and $C_1$ in $H_{X,0}^{1}(\Omega)$ such that $C\subset C_0\cup C_1$ and $C_1\cap B=\varnothing$, then
\begin{equation}\label{4-1}
  \gamma(C; B, A)\leq \gamma(C_0; B, A)+\gamma(C_1).
\end{equation}
    \item If $C'$ is a closed and symmetric subset in $H_{X,0}^{1}(\Omega)$ satisfying $B\subset C'$ and there exists an odd continuous map $\varphi: C'\to C$ such that $\varphi(C')\subset C$, $\varphi(B)\subset B$ and $\varphi(A)\subset A$, then

        $$\gamma(C'; B, A)\leq\gamma(C; B, A).$$
        Especially, if $C'\subset C$, then $\gamma(C'; B, A)\leq\gamma(C; B, A)$ holds by substituting $\varphi=\mathbf{id}$.
  \end{enumerate}
\end{lemma}

For the even functional  $I\in C^{2}(H_{X,0}^{1}(\Omega),\mathbb{R})$, we have

\begin{lemma}[{\cite[Corollary 4.4]{Bartsch2005}}]
\label{prop4-3}
Under the assumptions of Proposition \ref{prop4-2}, for any $c>0$ there exists $0<\delta<c$ such that
  \[\gamma(I^{c+\delta}\cup \overline{P_a}; I^{0}\cup \overline{P_a} , I^{-1})\leq \gamma(I^{c-\delta}\cup \overline{P_a}; I^{0}\cup \overline{P_a} , I^{-1})+\gamma(K_{I,c}\cap S_a).\]
\end{lemma}

\subsection{The sign-changing critical values of $I_{p}$}

In this subsection, we are concerned with the sign-changing critical values for the functional $I_{p}$. First, we have

\begin{proposition}
\label{prop4-4}
The functional $I_p$ satisfies the following properties:
 \begin{enumerate}[(P1)]
\item $I_p\in C^2(H_{X,0}^{1}(\Omega), \mathbb{R})$ and $I_p$ is even in $H_{X,0}^{1}(\Omega)$;
  \item $I_p$ satisfies the Palais-Smale condition;
  \item  $I_p(0)=0$;
\item There exists $0<a^*<a^{0}$ such that
   $(\mathbf{id}-I_p')(\overline{P_{a}^{\pm}})\subset \overline{P_{\frac{a}{4}}^{\pm}}$ for all $0<a<a^*$;
\item The map $I_{p}':H_{X,0}^{1}(\Omega)\to H_{X,0}^{1}(\Omega)$ defined by
\begin{equation}\label{4-6}
  (I_{p}'(u),v)_{H_{X,0}^{1}(\Omega)}=\int_{\Omega}Xu\cdot Xv dx-pC_{p}\int_{\Omega}|u|^{p-2}uv dx~~~~\forall u,v\in H_{X,0}^{1}(\Omega)
\end{equation}
is locally Lipschitz continuous.
\end{enumerate}
\end{proposition}
\begin{proof}

The properties (P1)-(P3) can be obtained by standard arguments,
while the property (P4) is due to Remark \ref{remark3-1}, \eqref{3-18} and \eqref{3-19}.

To prove property (P5), we suppose that $u_{0} \in H_{X,0}^{1}(\Omega)$ and $R>0$. For any $u_{1},u_{2}\in B(u_{0},R)$, we can deduce from Proposition \ref{prop2-1} and the mean value theorem that
\[ \begin{aligned}
&|(I_{p}'(u_{1})-I_{p}'(u_{2}),v)_{H_{X,0}^{1}(\Omega)}|\\
&=\left|\int_{\Omega}X(u_{1}-u_{2})\cdot Xv dx-pC_{p}\int_{\Omega}(|u_{1}|^{p-2}u_{1}-|u_{2}|^{p-2}u_{2})v dx\right|\\
&\leq \|u_{1}-u_{2}\|_{H_{X,0}^{1}(\Omega)}\|v\|_{H_{X,0}^{1}(\Omega)}+C\int_{\Omega}(|u_{1}|^{p-2}+|u_{2}|^{p-2})|u_{1}-u_{2}||v|dx\\
&\leq \|u_{1}-u_{2}\|_{H_{X,0}^{1}(\Omega)}\|v\|_{H_{X,0}^{1}(\Omega)}+C(\|u_{1}\|_{L^{p}(\Omega)}^{p-2}+\|u_{2}\|_{L^{p}(\Omega)}^{p-2})\|u_{1}-u_{2}\|_{L^{p}(\Omega)}\|v\|_{L^{p}(\Omega)}\\
&\leq C(u_{0}) \|u_{1}-u_{2}\|_{H_{X,0}^{1}(\Omega)}\|v\|_{H_{X,0}^{1}(\Omega)},
\end{aligned}
\]
where $C(u_{0})>0$ is a constant depending on $u_{0}$. This means
\[ \|I_{p}'(u_{1})-I_{p}'(u_{2})\|_{H_{X,0}^{1}(\Omega)}\leq C(u_{0}) \|u_{1}-u_{2}\|_{H_{X,0}^{1}(\Omega)},\]
and therefore $I_{p}':H_{X,0}^{1}(\Omega)\to H_{X,0}^{1}(\Omega)$ is locally Lipschitz continuous.
\end{proof}

\begin{lemma}
\label{prop4-5}
  For any $u\in H_{X,0}^1(\Omega)$, $I_p''(u)$ is a Fredholm operator. Moreover, $I_p''(u)$ admits an increasing sequence of discrete eigenvalues $\{\iota_j(u)\}_{j=1}^\infty$ such that $\iota_j(u)\to 1$ as $j\to \infty$.
\end{lemma}
\begin{proof}
By \eqref{4-5}, we have
\begin{equation}\label{4-7}
\begin{aligned}
(I_{p}''(u)(h),v)_{H_{X,0}^1(\Omega)}=\int_{\Omega}Xh\cdot Xvdx-p(p-1)C_p\int_{\Omega}|u|^{p-2}hv dx~~~\forall h,v\in H_{X,0}^1(\Omega).
\end{aligned}
\end{equation}
This means
 \begin{equation}\label{4-8}
 I_p''(u)(\cdot)=\mathbf{id}(\cdot)+W(u)(\cdot),
 \end{equation}
 where $W(u): H_{X,0}^1(\Omega)\to H_{X,0}^1(\Omega)$ is a bounded linear operator given by
  \begin{equation}\label{4-9}
(W(u)(h),v)_{H_{X,0}^1(\Omega)}=-p(p-1)C_p\int_{\Omega}|u|^{p-2}hv dx~~~\forall h,v\in H_{X,0}^1(\Omega).
 \end{equation}
It follows from \eqref{4-9} that $W(u)$ is a self-adjoint operator.

According to \eqref{4-8} and \eqref{4-9}, the Fredholm property of  $I_p''(u)$ amounts to showing that  $W(u)$ is a compact operator. Let
$\{h_n\}_{n=1}^{\infty}$ be a sequence in $H_{X,0}^{1}(\Omega)$ satisfying $\|h_n\|_{H_{X,0}^{1}(\Omega)}\leq 1$.  By Proposition \ref{prop2-3}, we can choose a subsequence $\{h_{k_{j}}\}_{j=1}^{\infty}\subset \{h_{k}\}_{k=1}^{\infty}$ such that $h_{k_{j}} \to h_{0}$ in $L^{p}(\Omega)$ as $j\to \infty$. Since
 \begin{equation*}
   \begin{aligned}
   \|W(u)(h_{k_{j}})-W(u)(h_{k_{m}})\|_{H_{X,0}^1(\Omega)}
   &\leq C\sup\limits_{\|v\|_{H_{X,0}^1(\Omega)}=1}
  \left|\int_{\Omega}|u|^{p-2}(h_{k_{j}}-h_{k_{m}})v dx\right|\\
  &\leq C\sup\limits_{\|v\|_{H_{X,0}^1(\Omega)}=1} \|u\|_{L^{p}(\Omega)}^{p-2}\|h_{k_{j}}-h_{k_{m}}\|_{L^{p}(\Omega)}\|v\|_{L^{p}(\Omega)}\\
  &\leq C\|u\|_{L^{p}(\Omega)}^{p-2}\|h_{k_{j}}-h_{k_{m}}\|_{L^{p}(\Omega)},
   \end{aligned}
 \end{equation*}
we know $\{W(u)(h_{k})\}_{k=1}^{\infty}$ converges along a subsequence
 in $H_{X,0}^1(\Omega)$. Therefore, $W(u)$ is a self-adjoint compact operator in
 $H_{X,0}^1(\Omega)$, and $I_p''(u)$ is a Fredholm operator.  Moreover, by the spectral theory of compact operator and \eqref{4-9},
 $I_{p}''(u)$ admits an increasing sequence of discrete eigenvalues $\{\iota_j(u)\}_{j=1}^\infty$ such that $\iota_j(u)\to 1$ as $j\to \infty$.
\end{proof}

We next denote by
\[B^{p}_{\rho}:=\{u\in H_{X,0}^{1}(\Omega): \|u\|_{L^p(\Omega)}\leq \rho\}\qquad \mbox{and}\qquad S^{p}_{\rho}:=\{u\in H_{X,0}^{1}(\Omega): \|u\|_{L^p(\Omega)}=\rho\}.\]
It follows that
\begin{lemma}\label{lemma4-2}
  There exists $\rho_0>0$  such that $I_p^{0}\cap
  B^{p}_{\rho_{0}}
  =\{0\}.$
\end{lemma}
\begin{proof}
Using Proposition \ref{prop2-1} and \eqref{4-5}, there exists $C>0$ such that
 \[ I_p(u)=\frac{1}{2}\int_{\Omega}|Xu|^2dx-C_p\int_{\Omega}|u|^{p}dx  \geq C\|u\|_{L^p(\Omega)}^2-C_p\|u\|_{L^p(\Omega)}^p.\]
  Since $p>2$, we can find a positive constant $0<\rho_0<\left(\frac{C}{C_{p}}\right)^{\frac{1}{p-2}}$ such that  $I_p(u)\geq I_p(0)=0$ for any $u\in B^{p}_{\rho_{0}}$, and $I_p|_{B^{p}_{\rho_{0}}}(u)=0$ if and only if $u=0$.
\end{proof}

Let $a^*>0$ be the positive constant in Proposition \ref{prop4-4} above. For any fixed $0<a<a^*$, we denote by
\begin{equation}\label{4-10}
d_{k,a}(I_{p}):=\inf\{c\geq 0: \gamma(I_{p}^{c}\cup \overline{P_a}; I_{p}^{0}\cup \overline{P_a}, I_{p}^{-1})\geq k\}.
\end{equation}
The rest of this subsection aims to the well-definedness of $d_{k,a}(I_{p})$ for small $a>0$. Furthermore, we will show that $\{d_{k,a}(I_{p})\}_{k=1}^{\infty}$ is a sequence of sign-changing critical values of $I_{p}$.

To verify the well-definedness of $d_{k,a}(I_{p})$, we introduce another
min-max value $b_{k,a}(I_p)$, which is a generalization of $b_k(I_p)$. 
For any $k\geq 1$, the construction of
  $\{R_{k}\}_{k=3}^{\infty}$ in subsection \ref{subsection3-3} and \eqref{4-4} derive that $I_{p}(u)<-1$  for all $u\in N_k\cap B_{R_k}^c$. As a generalization of \eqref{3-58}, we set the min-max value
\begin{equation}\label{4-11}
b_{k,a}(I_{p}):=\inf\limits_{\phi\in \Gamma_k}\sup_{\phi(N_k\cap B_{R_k})\cap S_{a}}I_{p}(u)~~~\mbox{for}~~k\geq 3,~~ 0<a< a^*.
\end{equation}

Lemma \ref{lemma3-10} shows that for any $\phi\in \Gamma_k$ and any $0<a< a^*$, the set
$\phi(N_k\cap B_{R_k})\cap S_{a}\cap Q_{k-1}$ is non-empty. Therefore, the min-max value
 $b_{k,a}(I_{p})$ is well-defined for $k\geq 3$ and $0<a< a^*$. Moreover,  \eqref{3-39} is also applicable to $I_p$, since the functional $J$ reduces to $I_p$ when $f(x,u)=pC_p|u|^{p-2}u$ and $g(x,u)\equiv 0$. The same arguments in the proof of Lemma \ref{lemma3-15} implies that
  $b_{k+2,a}(I_p)>0$ for all $k\geq 1$ and $0<a< a^*$. This yields the well-definedness of  $\gamma(I_p^{b_{k+2,a}(I_p)}\cup \overline{P_a}; I_p^{0}\cup \overline{P_a} , I_p^{-1})$.
  Furthermore, we can deduce the following proposition.

\begin{proposition}\label{prop4-6}
  For $k\geq 1$, there exists $0<a^1<a^*$ such that for any $0<a<a^1$ and $\varepsilon>0$ we have
 \begin{equation}\label{4-12}
    \gamma(I_p^{b_{k+2,a}(I_p)+\varepsilon}\cup \overline{P_a}; I_p^{0}\cup \overline{P_a} , I_p^{-1})\geq k.
     \end{equation}
     As a result, \[ d_{k,a}(I_p)=\inf\{c\geq 0: \gamma(I_{p}^{c}\cup \overline{P_a}; I_{p}^{0}\cup \overline{P_a}, I_{p}^{-1})\geq k\}\]
      is well-defined, and $b_{k+2,a}(I_p)\geq d_{k,a}(I_p)$ for any $k\geq 1$ and any $0<a<a^1$.

\end{proposition}
\begin{proof}
For any $k\geq 1$ and $a\in (0,a^{*})$, the relative genus $\gamma(I_p^{b_{k+2,a}(I_p)+\varepsilon}\cup \overline{P_a}; I_p^{0}\cup \overline{P_a} , I_p^{-1})$ is well-defined for any $\varepsilon>0$ since $b_{k+2,a}(I_p)>0$. For any $a\in (0,a^{*})$,    by \eqref{4-11} we can choose a $\phi\in \Gamma_{k+2}$ such that
   \[\sup_{\phi(N_{k+2}\cap B_{R_{k+2}})\cap S_{a}}I_p(u)\leq b_{k+2,a}(I_p)+\varepsilon,\]
 which means
\begin{equation}\label{4-14}
\phi(N_{k+2}\cap B_{R_{k+2}})\cap S_{a}\subset I_p^{b_{k+2,a}(I_p)+\varepsilon}.
\end{equation}
According to Definition \ref{def4-2},  there exist closed and symmetric subsets $U, V \subset H_{X,0}^1(\Omega)$ such that
  \begin{equation}\label{4-15}
 I_p^{b_{k+2,a}(I_p)+\varepsilon}\cup \overline{P_a} \subset U\cup V,~ ~~ I_p^{0}\cup \overline{P_a} \subset U,~~~ \gamma(V)\leq \gamma(I_p^{b_{k+2,a}(I_p)+\varepsilon}\cup \overline{P_a}; I_p^{0}\cup \overline{P_a} , I_p^{-1}),
\end{equation}
  and an odd continuous map $h$ satisfying
  \begin{equation}\label{4-16}
  h: U\to I_p^{0}\cup \overline{P_a}, \qquad h(I_p^{-1})\subset I_p^{-1}.
   \end{equation}
Hence, we can conclude from \eqref{4-14} and \eqref{4-15} that
\begin{equation}\label{4-17}
   \phi(N_{k+2}\cap B_{R_{k+2}})\subset I_p^{b_{k+2}(I_p)+\varepsilon}\cup \overline{P_a}\subset U\cup V.
\end{equation}

Let $\rho_0>0$ be the constant given in Lemma \ref{lemma4-2}, we define
  \begin{equation}\label{4-18}
   O:=\{u\in N_{k+2}\cap B_{R_{k+2}}: \|h(\phi(u))\|_{L^p(\Omega)}<\rho_0\}.
   \end{equation}
Since $h$ and $\phi$ are odd maps, $O$ is a symmetric bounded open set in $N_{k+2}$ containing the origin. From \cite[Proposition 5.2]{Struwe2000}, we have
   \begin{equation}\label{4-19}
  \gamma(\partial O)=k+2.
  \end{equation}
  If $u\in \partial O$, there are only two cases occur:
  \begin{itemize}
    \item $\|u\|_{H_{X,0}^{1}(\Omega)}=R_{k+2}$ and $\|h(\phi(u))\|_{L^p(\Omega)}\leq \rho_0$;
    \item $\|h(\phi(u))\|_{L^p(\Omega)}=\rho_0$ with $\|u\|_{H_{X,0}^{1}(\Omega)}< R_{k+2}$.
  \end{itemize}
In the first case, we have $I_p(u)<-1$ (i.e. $u\in I_{p}^{-1}$) due to the construction of $\{R_k\}_{k=3}^{\infty}$. Recalling that $\phi|_{N_{k+2}\cap B_{R_{k+2}}^{c}}=\mathbf{id}$, by \eqref{4-16} we obtain
  \[ h(\phi(u))=h(u)\in I_p^{-1}\cap B^{p}_{\rho_{0}}. \]
But the fact $0\notin I_{p}^{-1}\subset I_{p}^{0}$ and Lemma \ref{lemma4-2} tell us that $I_p^{-1}\cap B^{p}_{\rho_{0}}=\varnothing$, which leads to a contradiction. Therefore
  \begin{equation}\label{4-20}
    h(\phi( \partial O))\subset S^p_{\rho_{0}},
  \end{equation}
  and
  \begin{equation}\label{4-21}
    \phi(\partial O)\subset \phi(N_{k+2}\cap B_{R_{k+2}})\subset U\cup V.
  \end{equation}
  Furthermore, since  $h: U\to I_p^{0}\cup \overline{P_a}$ and $I_p^{0}\cap S^p_{\rho_{0}}=\varnothing$, \eqref{4-20} gives
   \begin{equation}\label{4-22}
   h(U\cap \phi(\partial O))\subset (I_{p}^{0}\cup \overline{P_a})\cap S^p_{\rho_{0}} \subset  \overline{P_a}\cap S^p_{\rho_{0}}=Z_{a}(\rho_{0}).
    \end{equation}

    By Lemma \ref{lemma3-9} and the properties of Krasnoselskii genus, there exists $0<a^1< a^*$ depending on $\rho_0$ such that for any $a\in(0,a^{1})$,
    \begin{equation}\label{4-23}
    \gamma(U\cap \phi(\partial O))\leq \gamma(\overline{h(U\cap \phi(\partial O))})
    \leq\gamma(\overline{P_a}\cap S^p_{\rho_{0}})=\gamma(Z_{a}(\rho_{0}))\leq 1.
     \end{equation}
    Note that \eqref{4-15} yields
     \begin{equation}\label{4-24}
     \gamma(V\cap \phi(\partial O))\leq \gamma(V)\leq \gamma(I_p^{b_{k+2,a}(I_p)+\varepsilon}\cup \overline{P_a}; I_p^{0}\cup \overline{P_a} , I_p^{-1}).
     \end{equation}
  Combining \eqref{4-19}, \eqref{4-21}, \eqref{4-23} and \eqref{4-24}, we deduce that
      \begin{equation}\label{4-25}
        k+2=\gamma(\partial O)\leq\gamma( \overline{\phi(\partial O)})\leq
       \gamma(U\cap \phi(\partial O))+ \gamma(V\cap \phi(\partial O))\leq
      \gamma(I_p^{b_{k+2,a}(I_p)+\varepsilon}\cup \overline{P_a}; I_p^{0}\cup \overline{P_a} , I_p^{-1})+1,
      \end{equation}
   which derives \eqref{4-12} for any $a\in (0,a^{1})$ and $\varepsilon>0$. We mention that in \eqref{4-23} and \eqref{4-25},  the bounded closed subset $\partial O$ in the finite dimensional space $N_{k+2}$ is a compact set, therefore $\overline{\phi(\partial O)}=\phi(\partial O)$ and $\overline{h(U\cap \phi(\partial O))}=h(U\cap \phi(\partial O))$.
\end{proof}

\begin{lemma}
  \label{prop4-7}
Let $a^{1}>0$ be the positive constant in Proposition \ref{prop4-6}.  For any $0<a<a^1$, we have $d_{1,a}(I_p)>0$.
\end{lemma}

\begin{proof}
  If $u\in K_{I_p,0}$, then $I_p(u)=0$ and $(I_p'(u), u)_{H_{X,0}^{1}(\Omega)}=0$. It follows that
  \[\frac{1}{2}\|u\|_{H_{X,0}^{1}(\Omega)}^2=\frac{1}{p}\|u\|_{H_{X,0}^{1}(\Omega)}^2,\]
  which forces $u=0$. Thus $K_{I_p,0}\cap S_a=\varnothing$ due to $0\notin S_{a}$. From Proposition \ref{prop4-4}, we see that $I_p$ satisfies all conditions outlined in Proposition \ref{prop4-2}. By applying Proposition \ref{prop4-2}  with
  $N=\varnothing$ and $c=0$, for any $\varepsilon_1>0$, there exist $\delta\in(0,\varepsilon_1)$ and an odd continuous map $\Theta:H_{X,0}^{1}(\Omega)\to H_{X,0}^{1}(\Omega)$ such that
  \[\Theta(I_p^\delta\cup \overline{P_a})\subset I_p^{-\delta}\cup \overline{P_a}\subset I_p^{0}\cup \overline{P_a}\]
  and
  \[ I_{p}(\Theta(u))\leq I_{p}(u)\qquad \forall u\in H_{X,0}^{1}(\Omega). \]
This means $\Theta(I_p^{0}\cup \overline{P_a})\subset I_p^{0}\cup \overline{P_a}$ and $\Theta(I_{p}^{-1})\subset I_{p}^{-1}$. By the conclusion (2) in Lemma \ref{prop4-1}, we have
  \[0\leq\gamma(I_p^{\delta}\cup \overline{P_a}; I_p^{0}\cup \overline{P_a}, I_p^{-1})
  \leq\gamma(I_p^{0}\cup \overline{P_a}; I_p^{0}\cup \overline{P_a}, I_p^{-1})=0.\]
  Hence, we obtain from \eqref{4-10} that $d_{1,a}(I_p)> \delta>0$.
\end{proof}

 \begin{proposition}
 \label{prop4-8}
 Let $a^{1}>0$ be the positive constant in Proposition \ref{prop4-6}. For any $0<a<a^1$ and any $k\geq 1$, $d_{k,a}(I_p)$ is a sign-changing critical value of $I_p$, i.e., $K_{I_{p},d_{k,a}(I_{p})}\cap S_{a}\neq \varnothing$.
\end{proposition}
\begin{proof}
It follows from the definition of $d_k(I_p)$ and Lemma \ref{prop4-7} that, for any $k\geq1$ and any $0<a<a^1$, we have
\[d_{k+1,a}(I_p)\geq d_{k,a}(I_p)\geq d_{1,a}(I_p)>0.\]
Then, \eqref{4-10} and Lemma \ref{prop4-3} indicate that for some $0<\delta<d_{k,a}(I_p)$,
  \[ \begin{aligned}
 &\gamma(K_{I_{p},d_{k,a}(I_p)}\cap S_a)\geq \gamma(I_{p}^{{d_{k,a}(I_p)}+\delta}\cup \overline{P_a}; I_{p}^{0}\cup \overline{P_a} , I_p^{-1})-\gamma(I_{p}^{{d_{k,a}(I_p)}-\delta}\cup \overline{P_a}; I_{p}^{0}\cup \overline{P_a} , I_{p}^{-1})\\
 &\geq k-(k-1)=1.
  \end{aligned}\]
This means $K_{I_{p}, d_{k,a}(I_p)}\cap S_a\neq\varnothing$ and  $d_{k,a}(I_p)$ is a sign-changing critical value of $I_{p}$.
\end{proof}

\subsection{The Marino-Prodi perturbation of $I_p$}
In this subsection, we employ the Marino-Prodi perturbation technique to seek the perturbation functional of $I_{p}$.

\begin{proposition}\label{prop4-10}
 For any given $\epsilon_1,\epsilon_2,\epsilon_3 >0$ and any $c\in \mathbb{R}$, we denote by
 \[K_{I_{p},c,2\epsilon_3}:=\{u\in H_{X,0}^{1}(\Omega):~ I_{p}'(u)=0,~ |I_{p}(u)-c|\leq 2\epsilon_3\}.\]
 Then
  there exists an even functional $\widehat{I}\in C^{2}(H_{X,0}^{1}(\Omega),\mathbb{R})$ satisfying the Palais-Smale condition. Moreover, we have
\begin{enumerate}[(a)]
  \item $\|I_{p}-\widehat{I}\|_{C^2}<\epsilon_1$;
  \item $I_{p}(u)=\widehat{I}(u)$ if $u\notin K_{I_{p},c,2\epsilon_3}^{\epsilon_2}$, where $K_{I_{p},c,2\epsilon_3}^{\epsilon_2}:=\{u\in H_{X,0}^1(\Omega):{\rm dist}(u,K_{I_{p},c,2\epsilon_3})<\epsilon_2\}$;
  \item $K_{\widehat{I},c,\epsilon_3}\subset K_{I_{p},c,2\epsilon_3}^{\epsilon_2}$ and $K_{\widehat{I},c,\epsilon_3}$ consists of finitely many pairs of non-degenerate critical points of $\widehat{I}$.
\end{enumerate}
Here, for $C^2$-bounded functional $\varphi$ on $H_{X,0}^{1}(\Omega)$, the $C^2$-norm is defined by
\[ \|\varphi\|_{C^2}:=\sup\limits_{x\in H_{X,0}^{1}(\Omega)} \left\{|\varphi(x)|+\|\varphi'(x)\|_{H_{X,0}^{1}(\Omega)}+\|\varphi''(x)\|_{\mathcal{L}(H_{X,0}^{1}(\Omega),H_{X,0}^{1}(\Omega))}\right\}.\]
\end{proposition}
\begin{proof}
According to Proposition \ref{prop4-4}, $I_p\in C^{2}(H_{X,0}^{1}(\Omega),\mathbb{R})$ is an even functional satisfying the Palais-Smale condition. Moreover, Lemma \ref{prop4-5} indicates that $I_{p}^{''}(u)$ is a Fredholm operator for any $u\in H_{X,0}^{1}(\Omega)$. Thus, Proposition \ref{prop4-10} can be derived from \cite[Proposition $B_{3}$]{Wang2016}.
\end{proof}

For any given parameters $\epsilon_1,\epsilon_2,\epsilon_3>0$  and $c=d_{k,a}(I_{p})$,
 Proposition \ref{prop4-10} provides a new functional $\widehat{I}$, which serves as a perturbation of $I_{p}$. Specifically, we have:

\begin{proposition}
\label{prop4-11}
For any $0<a<a^1 $ and any $k\geq 1$, we can find some positive constants $\epsilon_1,\epsilon_2,\epsilon_3>0$ and a corresponding even functional $\widehat{I}\in C^{2}(H_{X,0}^{1}(\Omega),\mathbb{R})$ (depending on $k,\epsilon_1,\epsilon_2$ and $\epsilon_3$) satisfying the Palais-Smale condition and the properties (a)-(c) in Proposition \ref{prop4-10}. Furthermore, we have
\begin{enumerate}[(1)]
\item For all $j\geq 1$, $d_{j,a}(\widehat{I})$ is well-defined, and $d_{j,a}(\widehat{I})>0$;
  \item $ K_{\widehat{I},d_{k,a}(\widehat{I})}\subset K_{\widehat{I},d_{k,a}(I_{p}),\epsilon_3}\subset
 K_{I_p,d_{k,a}(I_p),2\epsilon_3}^{\epsilon_2}$;
  \item $(\mathbf{id}-\widehat{I}')(\overline{P_{a}^{\pm}})\subset \overline{P_{\frac{a}{2}}^{\pm}}$;
  \item $\widehat{I}':H_{X,0}^{1}(\Omega)\to H_{X,0}^{1}(\Omega)$ is locally Lipschitz continuous;
 \item $K_{\widehat{I},d_{k,a}(\widehat{I})}\cap S_{a}\neq \varnothing$, where $K_{\widehat{I},d_{k,a}(\widehat{I})}$ consists of finitely many pairs of non-degenerate critical points of $\widehat{I}$.
\end{enumerate}
\end{proposition}
\begin{proof}
 For any fixed
 $0<a<a^1 $ and $k\geq 1$,  Proposition \ref{prop4-4} and Lemma \ref{prop4-7} indicate that $K_{I_p,d_{k,a}(I_p)}$ is compact and
 $d_{k,a}(I_p)\geq d_{1,a}(I_p)>0$. Thus, we can choose $\epsilon_2,\epsilon_3>0$ such that
 \begin{equation}\label{4-28}
   I_p(u)>\frac{d_{k,a}(I_p)}{2}>0\qquad \forall u\in  K_{I_p,d_{k,a}(I_p),2\epsilon_3}^{\epsilon_2}.
 \end{equation}
Additionally,  we set $0<\epsilon_1<\min\left\{\frac{a}{4},\epsilon_3,\frac{d_{1,a}(I_{p})}{3}\right\}$ and $c=d_{k,a}(I_{p})$. According to Proposition \ref{prop4-10}, there exists an even functional $\widehat{I}\in C^{2}(H_{X,0}^{1}(\Omega),\mathbb{R})$ that satisfies the Palais-Smale condition and conditions (a)-(c) in Proposition \ref{prop4-10} above. We mention that the perturbation functional $\widehat{I}$ depends on the parameters $k,\epsilon_1,\epsilon_2$ and $\epsilon_3$.

By \eqref{4-28} we can see that $I_{p}^{0}\cap K_{I_p,d_{k,a}(I_p),2\epsilon_3}^{\epsilon_2}=\varnothing$, which implies that
\[ I_p^0=\widehat{I}^0\qquad \text{and}\qquad I_p^{-1}=\widehat{I}^{-1}.\]
Thus, according to \eqref{4-10} and Proposition \ref{prop4-6},  for any $\varepsilon>0$ and any $j\geq 1$,
\begin{equation}\label{4-29}
\gamma(I_{p}^{d_{j,a}(I_p)+\varepsilon}\cup \overline{P_a}; \widehat{I}^{0}\cup \overline{P_a} , \widehat{I}^{-1})=\gamma(I_{p}^{d_{j,a}(I_p)+\varepsilon}\cup \overline{P_a}; I_p^{0}\cup \overline{P_a} , I_p^{-1})
\geq j.
\end{equation}
Using $\|I_{p}-\widehat{I}\|_{C^2}<\epsilon_1$ we have
\begin{equation}\label{4-30}
\sup_{u\in H_{X,0}^1(\Omega)}|I_p(u)-\widehat{I}(u)|\leq \epsilon_1,
\end{equation}
which yields that
\begin{equation}\label{4-31}
 I_p^{d_{j,a}(I_p)+\varepsilon}\subset \widehat{I}^{d_{j,a}(I_p)+\varepsilon+\epsilon_1}.
\end{equation}
Combining Lemma \ref{prop4-1}, \eqref{4-29} and \eqref{4-31}, we get
\[\gamma(\widehat{I}^{d_{j,a}(I_p)+\varepsilon+\epsilon_1}\cup \overline{P_a};\widehat{I}^{0}\cup \overline{P_a} , \widehat{I}^{-1})
\geq\gamma(I_p^{d_{j,a}(I_p)+\varepsilon}\cup \overline{P_a}; \widehat{I}^{0}\cup \overline{P_a} , \widehat{I}^{-1})\geq j,\]
which implies $d_{j,a}(\widehat{I})$ is well-defined, and $d_{j,a}(\widehat{I})\leq d_{j,a}(I_{p})+\epsilon_1$ for all $j\geq 1$. Similarly, we can deduce that
 $d_{j,a}(I_{p})\leq d_{j,a}(\widehat{I})+\epsilon_1$. Recalling that $0<\epsilon_1<\min\left\{\epsilon_3,\frac{d_{1,a}(I_{p})}{3}\right\}$, we obtain from Lemma \ref{prop4-7} that
\[ d_{j,a}(\widehat{I})\geq d_{j,a}(I_{p})-\epsilon_1\geq d_{1,a}(I_{p})-\epsilon_1>\frac{2}{3}d_{1,a}(I_{p})>0~~~\mbox{for all}~~j\geq 1.\]
Thus, conclusion (1) is proved.

For any $u\in K_{\widehat{I},d_{k,a}(\widehat{I})}$, we have $\widehat{I}'(u)=0$, and \eqref{4-30} derives that
\[ \widehat{I}(u)=d_{k,a}(\widehat{I})\in [d_{k,a}(I_{p})-\epsilon_{1},d_{k,a}(I_{p})+\epsilon_{1}]\subset [d_{k,a}(I_{p})-\epsilon_{3},d_{k,a}(I_{p})+\epsilon_{3}], \]
which gives that $ K_{\widehat{I},d_{k,a}(\widehat{I})}\subset K_{\widehat{I},d_{k,a}(I_{p}),\epsilon_3}$. Using Proposition \ref{prop4-10} (c), we obtain
\[  K_{\widehat{I},d_{k,a}(\widehat{I})}\subset K_{\widehat{I},d_{k,a}(I_{p}),\epsilon_3}\subset K_{I_{p},d_{k,a}(I_{p}),2\epsilon_3}^{\epsilon_{2}},\]
which yields the conclusion (2).

It follows from  Proposition \ref{prop4-10} (a) that
 $\sup_{u\in H_{X,0}^1(\Omega)}\|I_p'(u)-\widehat{I}'(u)\|_{H_{X,0}^1(\Omega)}<\epsilon_1$.
   Then  for any $u\in \overline{P_{a}^{\pm}}$, by Proposition \ref{prop4-4} we have
   \begin{align*}
    {\rm dist}(u-\widehat{I}'(u), P^\pm) & = {\rm dist}(u-I_p'(u)+I_p'(u)-\widehat{I}'(u), P^\pm)\\
    & \leq {\rm dist}(u-I_p'(u),P^\pm)+\|I_p'(u)-\widehat{I}'(u)\|_{ H_{X,0}^1(\Omega)}\\
    &\leq \frac{a}{4}+\epsilon_1\leq \frac{a}{2},
  \end{align*}
which gives the conclusion (3).

We next show that $\widehat{I}'$ is locally Lipschitz continuous. Using  Proposition \ref{prop4-10} (a) again, we have
\begin{equation*}
\sup\limits_{u\in H_{X,0}^{1}(\Omega)}\|\widehat{I}''(u)-I_{p}''(u)\|_{\mathcal{L}( H_{X,0}^{1}(\Omega),H_{X,0}^{1}(\Omega))}<\epsilon_1.
\end{equation*}
Then for any $u_1, u_2\in H_{X,0}^{1}(\Omega)$, by mean value theorem we have for some $t_{0}\in [0,1]$,
\[ \begin{aligned}
&\|(\widehat{I}'-I_{p}')(u_{2})-(\widehat{I}'-I_{p}')(u_{1})\|_{H_{X,0}^{1}(\Omega)}\\
&=\|(\widehat{I}''-I_{p}'')(u_1+t_{0}(u_2-u_1))(u_2-u_1)\|_{H_{X,0}^{1}(\Omega)}\\
&\leq \sup\limits_{u\in H_{X,0}^{1}(\Omega)}\|\widehat{I}''(u)-I_{p}''(u)\|_{\mathcal{L}( H_{X,0}^{1}(\Omega),H_{X,0}^{1}(\Omega))}\|u_1-u_2\|_{H_{X,0}^{1}(\Omega)}\leq \epsilon_1\|u_1-u_2\|_{H_{X,0}^{1}(\Omega)}.
\end{aligned}\]
This means $\widehat{I}'-I_{p}'$ is Lipschitz continuous. Hence, the conclusion (4) follows from
(P5) in Proposition \ref{prop4-4}.

From Proposition \ref{prop4-10} (c),  $K_{\widehat{I},d_{k,a}(\widehat{I})}$ consists of finitely many pairs of non-degenerate  critical points of $\widehat{I}$. Therefore, the conclusion (5) amounts to proving that $K_{\widehat{I},d_{k,a}(\widehat{I})}\cap S_{a}\neq \varnothing$. Given the previous results, $\widehat{I}$ meets the  conditions of Proposition \ref{prop4-2}. By Lemma \ref{prop4-3} and the fact that  $d_{k,a}(\widehat{I})>0$, we can select a $\delta\in (0, d_{k,a}(\widehat{I}))$ such that
\[ \begin{aligned}
 &\gamma(K_{\widehat{I},d_{k,a}(\widehat{I})}\cap S_{a})\geq \gamma(\widehat{I}^{{d_{k,a}(\widehat{I})}+\delta}\cup \overline{P_a}; \widehat{I}^{0}\cup \overline{P_a} , \widehat{I}^{-1})-\gamma(\widehat{I}^{{d_{k,a}(\widehat{I})}-\delta}\cup \overline{P_a}; \widehat{I}^{0}\cup \overline{P_a} , \widehat{I}^{-1})\\
 &\geq k-(k-1)=1.
  \end{aligned}\]
Consequently, $K_{\widehat{I},d_{k,a}(\widehat{I})}\cap S_{a}\neq\varnothing$, confirming that $d_{k,a}(\widehat{I})$ is a sign-changing critical value of $\widehat{I}$.
\end{proof}
By refining the choices of $\epsilon_1, \epsilon_2
,\epsilon_3$, we have the following result.
\begin{proposition}\label{prop4-12}
For any $0<a<a^1 $ and any $k\geq 1$, there exist some positive constants $\epsilon_1, \epsilon_2
,\epsilon_3>0$ and a corresponding perturbation functional $\widehat{I}$ that satisfies all the conditions outlined in Proposition \ref{prop4-11}. Furthermore,
\begin{equation}\label{4-32}
\sup \limits_{u^*\in K_{\widehat{I},d_{k,a}(\widehat{I})}\cap S_{a}}m(u^*,\widehat{I})\leq \sup \limits_{u\in K_{I_p,d_{k,a}(I_p)}\cap S_{a}}m^*(u,I_p).
\end{equation}
\end{proposition}
\begin{proof}
Proposition \ref{prop4-8} indicates that $K_{I_{p},d_{k,a}(I_{p})}\cap S_{a}\neq \varnothing$ for $0<a<a^1$. We need only consider the case where
\[   m^*:=\sup \limits_{u\in K_{I_p,d_{k,a}(I_p)}\cap S_{a}}m^*(u,I_p)<\infty.\]
 By Lemma \ref{prop4-5}, for each  $u\in K_{I_p,d_{k,a}(I_p)}\cap S_{a}$, $I_{p}''(u)$ admits a sequence of discrete eigenvalues $\{\iota_j(u)\}_{j=1}^\infty$ such that $\iota_{1}(u)\leq  \iota_{2}(u)\leq \cdots\leq \iota_{j}(u)\leq \cdots$, and $\iota_j(u)\to 1$ as $j\to \infty$. Additionally, Definition \ref{def4-3} gives that
 \[ m^*(u,I_{p})=\#\{j|\iota_j(u)\leq 0\}. \]
Thus, $\iota_j(u)>0$ holds for all $j> m^*$ and all $u\in K_{I_p,d_{k,a}(I_p)}\cap S_{a}$.

 Let $p_u:=\iota_{m^*+1}(u)>0$. By the Rayleigh-Ritz formula, there is a subspace $H_u\subset H_{X,0}^1(\Omega)$ with co-dimension $m^*$ such that
  \[(I_{p}''(u)(v),v)_{H_{X,0}^1(\Omega)}\geq p_u\|v\|_{H_{X,0}^1(\Omega)}^2\qquad \forall v\in H_u.\]
Since  $I_{p}''\in C(H_{X,0}^{1}(\Omega),\mathcal{L}(H_{X,0}^{1}(\Omega),H_{X,0}^{1}(\Omega)))$, for
 any $u\in K_{I_p,d_{k,a}(I_p)}\cap S_{a}$, there exists $r_u>0$ such that for all $w\in B_{r_u}(u)$, we have
\begin{equation}\label{4-33}
(I_{p}''(w)(v),v)_{H_{X,0}^1(\Omega)}\geq \frac{p_u}{2}\|v\|_{H_{X,0}^{1}(\Omega)}^2\qquad \forall v\in H_u.
\end{equation}
From Lemma \ref{lemma4-1} and the Palais-Smale condition, it follows that $K_{I_p,d_{k,a}(I_p)}\cap S_{a}$ is a compact set contained in the interior of $S_{a}$. Thus, we can select a finite collection of triples $\{(u_i, r_{u_i}, p_{u_{i}})\}_{i=1}^l$ such that
\begin{equation}\label{4-34}
K_{I_p,d_{k,a}(I_p)}\cap S_{a}\subset \bigcup_{i=1}^l B_{r_{u_i}}(u_i)\subset S_{a},
\end{equation}
which means
\[ K_{I_p,d_{k,a}(I_p)}=(K_{I_p,d_{k,a}(I_p)}\cap S_{a})\cup (K_{I_p,d_{k,a}(I_p)}\cap P_{a})\subset \bigcup_{i=1}^l B_{r_{u_i}}(u_i)\cup P_{a},\]
and $\cup_{i=1}^l B_{r_{u_i}}(u_i) \cup P_{a}$ is an open neighborhood of $K_{I_p,d_{k,a}(I_p)}$. According to \cite[Chapter II, Lemma 2.3]{Struwe2000}, there exist some small $\epsilon_2,\epsilon_3>0$ such that
 \begin{equation}\label{4-35}
K_{I_p,d_{k,a}(I_p),2\epsilon_3}^{\epsilon_2}\subset\bigcup_{i=1}^l B_{r_{u_i}}(u_i) \cup P_{a},
\qquad\text{and }\qquad I_{p}(u)>\frac{d_{k,a}(I_p)}{2}~~~\forall u\in K_{I_p,d_{k,a}(I_p),2\epsilon_3}^{\epsilon_2}.
\end{equation}
 Therefore,  \eqref{4-34} and \eqref{4-35} imply that
\begin{equation}\label{ssss}
K_{I_p,d_{k,a}(I_p),2\epsilon_3}^{\epsilon_2} \cap S_{a}\subset \bigcup_{i=1}^l B_{r_{u_i}}(u_i).
\end{equation}
Now, by restricting $0<\epsilon_1
<\min\left\{\frac{a}{4},\frac{d_{1,a}(I_p)}{3},\epsilon_3,\frac{p_{u_1}}{6},\cdots,\frac{p_{u_l}}{6}\right\}$ in Proposition \ref{prop4-11}, we can obtain a perturbation functional $\widehat{I}$ that satisfies all the conditions outlined in Proposition \ref{prop4-11}.

Note that $ K_{\widehat{I},d_{k,a}(\widehat{I})}\subset
 K_{I_p,d_{k,a}(I_p),2\epsilon_3}^{\epsilon_2}$. For any $u^*\in K_{\widehat{I},d_{k,a}(\widehat{I})}\cap S_{a}$, by \eqref{ssss} we have
$ u^*\in B_{r_{u_j}}(u_j)$  for some $1\leq j\leq l$. Besides, Proposition \ref{prop4-10} (a) implies that
\[ \|\widehat{I}''(u^*)-I_{p}''(u^*)\|_{\mathcal{L}( H_{X,0}^{1}(\Omega),H_{X,0}^{1}(\Omega))}<\epsilon_1<\min_{1\leq i\leq l}\left\{\frac{p_{u_i}}{6}\right\},\]
which derives that
\begin{equation}\label{4-36}
|(\widehat{I}''(u^*)(v),v)_{H_{X,0}^{1}(\Omega)}-(I_{p}''(u^*)(v),v)_{H_{X,0}^{1}(\Omega)}|\leq \epsilon_1\|v\|_{H_{X,0}^{1}(\Omega)}^{2}\qquad\forall v\in H_{X,0}^{1}(\Omega).
\end{equation}
As a result of \eqref{4-33} and \eqref{4-36}, we have
\[(\widehat{I}''(u^*)(v),v)_{H_{X,0}^{1}(\Omega)}\geq \frac{p_{u_j}}{3}\|v\|_{ H_{X,0}^{1}(\Omega)}^2\qquad\forall v\in H_{u_j},\]
where $H_{u_j}$ is a subspace of $H_{X,0}^{1}(\Omega)$ with  co-dimension $m^*$. According to Proposition \ref{prop4-11}, $K_{\widehat{I},d_{k,a}(\widehat{I})}$ consists of finitely many pairs of non-degenerate critical points of $\widehat{I}$. Thus,
 Definition \ref{def4-3} and Remark \ref{remark4-1-2} imply that $m(u^*,\widehat{I})=m^{*}(u^*,\widehat{I})\leq m^*$ for all $u^*\in K_{\widehat{I},d_{k,a}(\widehat{I})}\cap S_{a}$.
   \end{proof}

\subsection{Lower bound estimates of Morse index}
  Let us fix $0<a<a^1$ and $k\geq 1$, and  consider the perturbation functional $\widehat{I}$, as obtained from Proposition \ref{prop4-12} and satisfying \eqref{4-32}. In this subsection, we are dedicated to determining the lower bound of  $\sup_{u\in K_{\widehat{I},d_{k,a}(\widehat{I})}\cap_{S_{a}}}m(u,\widehat{I})$.

For any $ u\in K_{\widehat{I},d_{k,a}(\widehat{I})}\cap S_a$, the symmetry of $K_{\widehat{I},d_{k,a}(\widehat{I})}\cap S_a$ ensures that $-u\in  K_{\widehat{I},d_{k,a}(\widehat{I})}\cap S_a$. According to Lemma \ref{lemma4-1}, both
$u$ and $-u$ are located in the interior of $S_{a}$.  Proposition \ref{prop4-11} further indicates that $u$ is non-degenerate.  Utilizing the Morse lemma in \cite[Lemma 9.3]{Ghoussoub1993}, we derive an open symmetric neighborhood
 $V(0)$ of the origin in $H_{X,0}^{1}(\Omega)$ and an open neighborhood $U(u)$ of $u$, such that
 \begin{equation}\label{4-37}
   -V(0)=V(0)\subset H_{X,0}^{1}(\Omega)~~\mbox{and}~~U(u)\subset S_a.
    \end{equation}
Moreover, a homeomorphism $\mathcal{H}:V(0)\to U(u)$ exists satisfying
\begin{equation}\label{4-38}
\mathcal{H}(0)=u.
 \end{equation}
Meanwhile, there are two closed subspaces $H^+$ and $H^-$ of $H_{X,0}^{1}(\Omega)$ such that
\[ H_{X,0}^{1}(\Omega)=H^+\oplus H^-,\qquad \mbox{and}\qquad  \text{dim}(H^-)=m(u,\widehat{I}).\]
For each $v\in V(0)$, we can decompose $v=v_{+}+v_{-}$ (distinct from the notation $v^{\pm}$) with $v_{\pm}\in H^{\pm}$, and the Morse lemma states that
\begin{equation}\label{4-39}
\widehat{I}(\mathcal{H}(v))=d_{k,a}(\widehat{I})+\|v_{+}\|^2_{H_{X,0}^1(\Omega)}-\|v_-\|^2_{H_{X,0}^1(\Omega)}.
\end{equation}
Additionally, Proposition \ref{prop4-11} implies that  $\widehat{I}\in C^{2}(H_{X,0}^{1}(\Omega),\mathbb{R})$ and $\widehat{I}(u)=d_{k,a}(\widehat{I})\geq d_{1,a}(\widehat{I})>0$. Thus, we can further require $U(u)$ satisfies
\begin{equation}\label{4-40}
\widehat{I}(v)>0~~~~\forall v\in  \overline{U(u)}, ~~\text{and}~ ~U(u)\cap -U(u)=\varnothing.
\end{equation}

Denote by $B^-(r)\subset H^-$ and $B^+(r)\subset H^+$ the open balls centered at origin with radius $r>0$. Let  $0<2r^-<r^+$ small enough such that $B^-(2r^-)\oplus B^+(r^+)\subset\subset V(0)$. Then we define the following two open sets
\begin{equation}\label{4-41}
N(u):=\mathcal{H}(B^-(r^-)\oplus B^+(r^+)) ~~~\mbox{and}~~~N'(u):=\mathcal{H}(B^-(2r^-)\oplus B^+(r^+)).
\end{equation}
It follows that
\begin{equation}\label{4-42}
\pm N( u) \subset \pm N'( u) \subset  \pm U( u)\subset S_a.
\end{equation}

Drawing inspiration from Lazer-Solimini \cite{Lazer1988}, we can derive the following results.
\begin{lemma}\label{lemma4-3}
Let $\psi:\mathbb{R} \to[0,1]$ be a smooth function satisfying $\psi=0$
on $(-\infty,0]$ and $\psi=1$ on $[1,+\infty)$. For any $v\in V(0)$
  we define
  \begin{equation}\label{4-43}
    \phi(v)=\phi(v_{+}+v_{-})=v_{-} +  \psi\left(\frac{\|v_{-}\|_{H_{X,0}^{1}(\Omega)}}{r^-}-1\right)v_{+}
  \end{equation}
and
\begin{equation}\label{4-44}
\Phi(x):=
\begin{cases}
x, & x\notin N'(u)\cup -N'(u);\\
\mathcal{H}(\phi (\mathcal{H}^{-1}(x))),  & x\in N'(u);
\\
-\Phi(-x),  & x\in -N'(u).
\end{cases}
\end{equation}
Then we have
\begin{enumerate}[(a)]
\item $\Phi$ is odd and $\Phi|_{\widehat{I}^{0}\cup\overline{P_a}}= \mathbf{id}$;
  \item $\widehat{I}(\Phi(x))\leq \widehat{I}(x),~ \forall x\in H_{X,0}^1(\Omega)$;
  \item $(\mathcal{H}^{-1}(x))_{-}=(\mathcal{H}^{-1}(\Phi(x)))_{-} ,~\forall x\in U(u)$;
  \item $\Phi(N'(u))\subset N'(u)$;
  \item  The discontinuous points of $\Phi$ lie in
  \[ \mathcal{H}(B^-(2r^-)\oplus \partial B^+(r^+))\bigcup -\mathcal{H}(B^-(2r^-)\oplus \partial B^+(r^+)).\]
\end{enumerate}
\end{lemma}
\begin{proof}
\eqref{4-44} indicates that $\Phi$ is odd. Since $\widehat{I}(u)=\widehat{I}(-u)$ for all $u\in H_{X,0}^{1}(\Omega)$, by \eqref{4-40} we have $\widehat{I}>0$ on $U(u)\cup -U(u)$. Thus, \eqref{4-42} gives that $(\widehat{I}^{0}\cup\overline{P_a})\cap (N'(u)\cup - N'(u))=\varnothing$, which implies $\Phi|_{\widehat{I}^{0}\cup\overline{P_a}}= \mathbf{id}$. Hence, the conclusion (a) is proved.

For any $x\in  N'(u)$, \eqref{4-43} indicates that
\[ \phi(\mathcal{H}^{-1}(x))=(\mathcal{H}^{-1}(x))_{-}+\psi\left(\frac{\|(\mathcal{H}^{-1}(x))_{-}\|_{H_{X,0}^{1}(\Omega)}}{r^-}-1\right)(\mathcal{H}^{-1}(x))_{+}.
\]
Then, by \eqref{4-39}, \eqref{4-44}  and the fact that $|\psi|\leq 1$ we obtain
\[ \begin{aligned}
\widehat{I}(\Phi(x))&=\widehat{I}(\mathcal{H}(\phi(\mathcal{H}^{-1}(x))))\\
&=d_{k,a}(\widehat{I})+\left(\psi\left(\frac{\|(\mathcal{H}^{-1}(x))_{-}\|_{H_{X,0}^{1}(\Omega)}}{r^-}-1\right)\right)^{2}\|(\mathcal{H}^{-1}(x))_{+}\|^2_{H_{X,0}^1(\Omega)}-\|(\mathcal{H}^{-1}(x))_{-}\|^2_{H_{X,0}^1(\Omega)}\\
&\leq d_{k,a}(\widehat{I})+\|(\mathcal{H}^{-1}(x))_{+}\|^2_{H_{X,0}^1(\Omega)}-\|(\mathcal{H}^{-1}(x))_{-}\|^2_{H_{X,0}^1(\Omega)}=\widehat{I}(x).
\end{aligned} \]
Thus, conclusion (b) follows from the evenness of $\widehat{I}$.

 The conclusion (c) can be derived from \eqref{4-43} and \eqref{4-44} immediately. In addition,
for any $y\in \Phi(N'(u))$, we have $y=\Phi(\mathcal{H}(x))$ for some $x\in B^-(2r^-)\oplus B^+(r^+)\subset V(0)$ such that $x=x_{+}+x_{-}$ with $x_{+}\in B^{+}(r^{+})$ and $x_{-}\in B^{-}(2r^{-})$. According to \eqref{4-43}, $\phi(x)\in B^-(2r^-)\oplus B^+(r^+)$ and $y=\Phi(\mathcal{H}(x))=\mathcal{H}(\phi(x))\in N'(u)$. As a result, $\Phi(N'(u))\subset N'(u)$.

On the other hand, \eqref{4-43} and \eqref{4-44} indicate that  $\mathcal{H}^{-1}(\Phi(x))$ changes abruptly only in $\pm \mathcal{H}(B^-(2r^-)\oplus \partial B^+(r^+))$. Thus, the discontinuous points of $\Phi$ lie in $\mathcal{H}(B^-(2r^-)\oplus \partial B^+(r^+))\cup -\mathcal{H}(B^-(2r^-)\oplus \partial B^+(r^+))$.
\end{proof}

\begin{lemma}\label{lemma4-4}

Let $0<\xi<(r^+)^2-(2r^-)^2$ and
\begin{equation}\label{4-45}
  D(u):=\Phi(\widehat{I}^{d_{k,a}(\widehat{I})+\xi})\cap N(u).
\end{equation}
Then, $\Phi$ is continuous on $\widehat{I}^{d_{k,a}(\widehat{I})+\xi}$, and $D(u)=\mathcal{H}(B^-(r^-))$.
\end{lemma}
\begin{proof}

Using $0<2r^-<r^+$ and \eqref{4-39} we have
\begin{equation}
\widehat{I}(x)\geq d_{k,a}(\widehat{I})+(r^+)^2-(2r^-)^2> d_{k,a}(\widehat{I})+\xi
\end{equation}
holds for any $x\in \mathcal{H}(B^-(2r^-)\oplus \partial B^+(r^+))\cup -\mathcal{H}(B^-(2r^-)\oplus \partial B^+(r^+))$,
which implies that  $\Phi$ is continuous on $\widehat{I}^{d_{k,a}(\widehat{I})+\xi}$ due to Lemma \ref{lemma4-3} (e).

For any $y\in D(u)$, we have $y=\Phi(x)\in \mathcal{H}(B^-(r^-)\oplus B^+(r^+))$ for some $x\in \widehat{I}^{d_{k,a}(\widehat{I})+\xi}$. Since $y\in N(u)\subset N'(u)$, by \eqref{4-44} and  Lemma \ref{lemma4-3} (d) we can deduce that $x\in N'(u)\subset U(u)$. Using  Lemma \ref{lemma4-3} (c),
we have
\begin{equation}\label{4-47}
\|{(\mathcal{H}^{-1}(x))}_{-}\|_{H_{X,0}^{1}(\Omega)}\leq r^{-},\qquad\text{and}~~~
\psi\left(\frac{\|{\mathcal{H}^{-1}(x)}_{-}\|_{H_{X,0}^{1}(\Omega)}}{r^-}-1\right)=0.
\end{equation}
This means $y=\Phi(x)\in \mathcal{H}(B^-(r^-))$ and thus
  $D(u)\subset\mathcal{H}(B^-(r^-)).$

On the other hand, for any $x\in B^{-}(r^{-})$, we can verify that $\Phi(\mathcal{H}(x))=\mathcal{H}(\phi(x))=\mathcal{H}(x)$. Moreover, it follows from \eqref{4-39}  that $\widehat{I}(\mathcal{H}(x))\leq d_{k,a}(\widehat{I}) \leq  d_{k,a}(\widehat{I})+\xi$, i.e., $\mathcal{H}(x)\in \widehat{I}^{d_{k,a}(\widehat{I})+\xi}$. Hence, $\mathcal{H}(x)=\Phi(\mathcal{H}(x))\in \Phi(\widehat{I}^{d_{k,a}(\widehat{I})+\xi})$ gives $\mathcal{H}(B^-(r^-))\subset D(u)$.
\end{proof}

We now applying above results for the critical points to $K_{\widehat{I},d_{k,a}(\widehat{I})}\cap S_a$ under the restriction $0<a<a^1$ and $k\geq 1$. From Proposition \ref{prop4-11}, we may assume that
\begin{equation}\label{4-48}
K_{\widehat{I},d_{k,a}(\widehat{I})}\cap S_a=\{\pm u_1, \pm u_2, \dots , \pm u_{n_0}\}.
\end{equation}
Furthermore, for each $i\in \{1, 2, \ldots , n_{0}\}$,
there exist corresponding $\mathcal{H}_i, H_i^+, H_i^-, r_i^+, r_i^-,\Phi_i$ and
 \[ \pm N( u_i) \subset \pm N'( u_i) \subset  \pm U( u_i)\subset S_a . \]
In addition, we can require $\overline{U( u_i)}\cap \overline{U( u_j)}=\varnothing$ for $i\neq j$. Then, we denote by
\begin{equation}\label{4-49}
\Phi=\Phi_1\circ \Phi_2\circ\cdots \circ \Phi_{n_0}.
\end{equation}
We mention that $\Phi(x)=\Phi_{j}(x)$ for any $x\in N'( u_i)\cup -N'( u_i)$, and $\Phi: H_{X,0}^{1}(\Omega)\to H_{X,0}^{1}(\Omega)$ is an odd map.

By the definition of $d_{k,a}(\widehat{I})$, for each $\xi>0$, we can choose a symmetric closed set $X_\xi\subset H_{X,0}^{1}(\Omega)$ such that
 \begin{equation}\label{4-50}
 \gamma(X_\xi\cup \overline{P_a}; \widehat{I}^{0}\cup \overline{P_a} , \widehat{I}^{-1})\geq k,~~~
\sup_{u\in X_\xi}\widehat{I}(u)<d_{k,a}(\widehat{I})+\xi.
\end{equation}
Next, we have the following lemmas.
\begin{lemma}\label{lemma4-5}
Suppose $0<a<a^1$ and $k\geq 1$. For $0<\xi<\min_{1\leq i\leq n_{0}}\{(r_i^+)^2-(2r_i^-)^2\}$, we have
\begin{equation}\label{4-51}
\gamma(\overline{\Phi(X_\xi)}\cup \overline{P_a}; \widehat{I}^{0}\cup  \overline{P_a} , \widehat{I}^{-1})\geq k.
\end{equation}
\end{lemma}
\begin{proof}
The definition of relative genus and \eqref{4-50} imply that $ \widehat{I}^0\cup \overline{P_a}\subset X_\xi\cup \overline{P_a}$ and $X_{\xi}\subset \widehat{I}^{d_{k,a}(\widehat{I})+\xi}$. Using
 Lemma \ref{lemma4-3} (a), it follows that
\[\widehat{I}^0\cup \overline{P_a}=\Phi(\widehat{I}^0\cup \overline{P_a}) \qquad\text{and}\qquad \widehat{I}^{-1}=\Phi(\widehat{I}^{-1}),\]
which means
\[\widehat{I}^0\cup \overline{P_a}
\subset\Phi(X_\xi\cup\overline{P_a} )\subset  \Phi(X_\xi)\cup \Phi(\overline{P_{a}}) \subset \overline{\Phi(X_\xi)}\cup\overline{P_a}.\]
Note that $\overline{\Phi(X_\xi)}$ is symmetric due to the oddness of $\Phi$. Thus, $\gamma(\overline{\Phi(X_\xi)}\cup \overline{P_a}; \widehat{I}^{0}\cup  \overline{P_a} , \widehat{I}^{-1})$ is well-defined.

Since $0<\xi<\min_{1\leq i\leq n_{0}}\{(r_i^+)^2-(2r_i^-)^2\}$, we can deduce from Lemma \ref{lemma4-4} and \eqref{4-49} that $\Phi$ is continuous on $\widehat{I}^{d_{k,a}(\widehat{I})+\xi}$. Thus by  Lemma \ref{prop4-1} (2) and \eqref{4-50}, we have
\begin{equation}\label{4-52}
\gamma(\overline{\Phi(X_\xi)}\cup \overline{P_a}; \widehat{I}^{0}\cup  \overline{P_a} , \widehat{I}^{-1})\geq\gamma(X_\xi\cup \overline{P_a}; \widehat{I}^{0}\cup \overline{P_a} , \widehat{I}^{-1})\geq k.
\end{equation}
\end{proof}

\begin{lemma}\label{lemma4-7}
Assume that $0<a<a^1$, $k\geq 3$, and
\[
\max\limits_{u\in K_{\widehat{I}, d_{k,a}(\widehat{I})}\cap S_{a}}m(u,\widehat{I})\leq k-2.\] Then, for any $0<\xi<\min_{1\leq i\leq n_0}\{(r_i^+)^2-(2r_i^-)^2\}$, the relative genus $ \gamma(X^{*}_\xi\cup  \overline{P_a}; \widehat{I}^{0}\cup  \overline{P_a} , \widehat{I}^{-1})$ is well-defined and satisfying
\begin{equation}\label{4-55}
 \gamma(X^{*}_\xi\cup  \overline{P_a}; \widehat{I}^{0}\cup  \overline{P_a} , \widehat{I}^{-1})\geq k,
\end{equation}
where
\begin{equation}\label{4-53}
X^{*}_\xi:=\overline{\Phi(X_\xi)}\setminus N,
\end{equation}
and
 \begin{equation}\label{4-54}
      N:= \bigcup_{i=1}^{n_0}(N(u_i)\cup-N(u_i)).
      \end{equation}

\end{lemma}
\begin{proof}
Clearly,  $X^{*}_\xi$ is closed and symmetric. From \eqref{4-40} and \eqref{4-42}, we have
  \[(N(u_i)\cup -N(u_i))\cap (\widehat{I}^0\cup \overline{P_a})=\varnothing~~~~~\mbox{for}~~~1\leq i\leq n_0,\]
  which implies $\widehat{I}^0\cup \overline{P_a}\subset N^{c}$. It follows from Lemma \ref{lemma4-5} that  $\widehat{I}^0\cup \overline{P_a}\subset X^{*}_\xi\cup \overline{P_a}$, and
 $\gamma(X^{*}_\xi\cup \overline{P_a}; \widehat{I}^{0}\cup \overline{P_a} , \widehat{I}^{-1})$ is thus well-defined.

Suppose that $ \gamma(X^{*}_\xi\cup  \overline{P_a}; \widehat{I}^{0}\cup  \overline{P_a} , \widehat{I}^{-1})\leq k-1$. By Definition \ref{def4-2}, there are closed and symmetric subsets $U, V \subset H_{X,0}^{1}(\Omega)$ such that
  \begin{equation}\label{4-56}
  X^{*}_\xi\cup  \overline{P_a} \subset U\cup V,~~~  \widehat{I}^{0}\cup  \overline{P_a} \subset U,~~~\mbox{and}~~~ \gamma(V)\leq k-1.
  \end{equation}
Besides, there exists an odd continuous map $h:U\to \widehat{I}^{0}\cup  \overline{P_a}$ satisfying
  \begin{equation}\label{4-57}
 h(\widehat{I}^{-1})\subset \widehat{I}^{-1}.
  \end{equation}
 The definition of genus and \eqref{4-56} implies there exists an
  odd continuous map
\[ f: V \to \mathbb{S}^{k-2}. \]

On the other hand, for $0<\xi<\min_{1\leq i\leq n_0}\{(r_i^+)^2-(2r_i^-)^2\}$, Lemma \ref{lemma4-4} shows that the
$D(u_i)$ given by
\[ D(u_i)=\Phi(\widehat{I}^{d_{k,a}(\widehat{I})+\xi})\cap N(u_{i})\]
satisfying $D(u_i)= \mathcal{H}_i(B^-(r_i^-))$. This means
$\pm D(u_i)=\pm \mathcal{H}_i(B^-(r_i^-))$
is  homeomorphic to the ball in $\mathbb{R}^{m(u_i,\widehat{I})}$ with $m(u_i,\widehat{I})\leq k-2$. Using the extension lemma in \cite[Theorem D.9]{Ghoussoub1993} (also see \cite[Lemma 4.1]{Lazer1988}), for each $1\leq i\leq n_0$, the restriction $f|_{V \cap \overline{D(u_i)\cup -D(u_i)}}$ of map $f$ induces a continuous map $f_i:\overline{D(u_i)\cup -D(u_i)}\to \mathbb{S}^{k-2}$ such that
   \begin{equation}\label{4-58}
    f_i(u)=f(u)\qquad \forall u\in V \cap \overline{D(u_i)\cup -D(u_i)}.
   \end{equation}
 We can further require that $f_i$ is odd, since $\frac{f_i(x)-f_i(-x)}{2}$ is an odd map. Let
   \begin{equation}\label{4-59}
 \hat{f}:=\left\{\begin{array}{ll}{f(u),} & { u\in V,} \\[2mm]
 {f_i(u),} & {u\in \overline{D(u_i)\cup -D(u_i)},~~1\leq i\leq n_0.}\end{array}\right.
\end{equation}
It follows that $\hat{f}:V\cup  \overline{D}\to \mathbb{S}^{k-2}$ is an odd and continuous map,
where
\[D:= \bigcup_{i=1}^{n_0}(D(u_i)\cup -D(u_i)).\]
  Thus $\gamma(V\cup \overline{D}) \leq k-1$.

Since $\overline{\Phi(X_\xi)}= X^{*}_\xi\cup (\overline{\Phi(X_\xi)}\cap N)$
   and $X_\xi\subset \widehat{I}^{d_{k,a}(\widehat{I})+\xi}$, we have
   \[\overline{\Phi(X_\xi)}\subset X^{*}_\xi\cup ( \overline{\Phi(\widehat{I}^{d_{k,a}(\widehat{I})+\xi})\cap N}) \subset X^{*}_\xi\cup \overline{D}.\]
Additionally, \eqref{4-56} gives that  $\overline{\Phi(X_\xi)}\cup \overline{P_a}\subset U\cup V'$, where $V'=V\cup \overline{D}$ is a symmetric closed set satisfying $\gamma(V')\leq k-1$. Hence, we obtain from Definition \ref{def4-2} that
 \[ \gamma(\overline{\Phi(X_\xi)}\cup \overline{P_a}; \widehat{I}^{0}\cup  \overline{P_a} , \widehat{I}^{-1})\leq k-1,\]
   which contradicts  Lemma \ref{lemma4-5}.
\end{proof}

Owing to Lemma \ref{lemma4-5} and Lemma \ref{lemma4-7}, we can obtain the following proposition.
\begin{proposition}\label{prop4-13}
 For any integer $k\geq 3$, $0<a<a^1$, and $\widehat{I}$ be the functional given by Proposition \ref{prop4-12},
 there exists  $u\in K_{\widehat{I}, d_{k,a}(\widehat{I})}\cap S_{a}$ such that
\begin{equation}\label{4-60}
m(u, \widehat{I})\geq k-1.
\end{equation}

\end{proposition}
\begin{proof}
We shall prove \eqref{4-60} by contradiction. Suppose
\[  m(u, \widehat{I})\leq k-2\qquad \forall u\in  K_{\widehat{I}, d_{k,a}(\widehat{I})}\cap S_{a}.\]
Note that $N= \bigcup_{i=1}^{n_0}(N(u_i)\cup-N(u_i))$ is a symmetry neighborhood of $K_{\widehat{I}, d_{k,a}(\widehat{I})}\cap S_{a}$ and the functional $\widehat{I}$ given by Proposition \ref{prop4-12} admits the assumptions in Proposition \ref{prop4-2}. Applying Proposition \ref{prop4-2} for $N$, $\widehat{I}$, $c=d_{k,a}(\widehat{I})$ and
\[ \varepsilon_1=\min\left\{\frac{d_{1,a}(\widehat{I})}{2},(2r_1^-)^2-(r_1^+)^2,\ldots,(2r_{n_{0}}^-)^2-(r_{n_{0}}^+)^2\right\}>0,\]
 there exist $\delta\in(0,\varepsilon_1)$ and an odd  continuous map $\Theta:H_{X,0}^{1}(\Omega)\to H_{X,0}^{1}(\Omega)$ such that
 \begin{equation}\label{4-61}
\Theta((\widehat{I}^{d_{k,a}(\widehat{I})+\delta}\cup  \overline{P_a})\setminus N)\subset \widehat{I}^{d_{k,a}(\widehat{I})-\delta}\cup \overline{P_a}.
 \end{equation}

 Now, substituting  $\xi=\delta$ in Lemma \ref{lemma4-7}, we obtain
\[\gamma(X^{*}_{\delta}\cup  \overline{P_a}; \widehat{I}^{0}\cup  \overline{P_a} , \widehat{I}^{-1})\geq k,\]
where $X^{*}_\delta:=\overline{\Phi(X_\delta)}\setminus N$, and $X_{\delta}$ is given by \eqref{4-50}. Using  Lemma \ref{lemma4-3} (b) and \eqref{4-50}, we have
\[ \sup_{u\in\overline{\Phi(X_\delta)}}\widehat{I}(u)\leq d_{k,a}(\widehat{I})+\delta,\]
i.e. $\overline{\Phi(X_\delta)}\subset \widehat{I}^{d_{k,a}(\widehat{I})+\delta}$. By \eqref{4-42} and \eqref{4-54}, $N\cap \overline{P_a}=\varnothing$. Thus, we can obtain
\begin{equation}\label{4-62}
X_\delta^{*}\cup\overline{P_a}=(\overline{\Phi(X_\delta)}\setminus N)\cup\overline{P_a}
 \subset
 (\widehat{I}^{d_{k,a}(\widehat{I})+\delta}\cup  \overline{P_a})\setminus N.
\end{equation}
Combining \eqref{4-61} and \eqref{4-62}, we have
  \[\Theta(X_\delta^{*}\cup\overline{P_a})\subset \widehat{I}^{d_{k,a}(\widehat{I})-\delta}\cup \overline{P_a}.\]
 Meanwhile,  Proposition \ref{prop4-2} (ii) (iii) ensure that $\Theta(\widehat{I}^0\cup\overline{P_a})\subset \widehat{I}^0\cup\overline{P_a}$ and $\Theta(\widehat{I}^{-1})\subset \widehat{I}^{-1}$. Therefore, by Lemma \ref{prop4-1} (2) we have
 \[\gamma( \widehat{I}^{d_{k,a}(\widehat{I})-\delta}\cup \overline{P_a}; \widehat{I}^{0}\cup \overline{P_a} , \widehat{I}^{-1})\geq
 \gamma( X_\delta^{*}\cup\overline{P_a}; \widehat{I}^{0}\cup \overline{P_a} , \widehat{I}^{-1})\geq k,\]
  which contradicts the definition of $d_{k,a}(\widehat{I})$.
    Consequently, there exists a sign-changing critical point $u\in K_{\widehat{I},d_{k,a}(\widehat{I})}\cap S_a$ such that $m(u,\widehat{I})\geq k-1$.
 \end{proof}

\section{Proofs of Theorem 1.1 and Theorem 1.2}\label{Section5}

\subsection{Proof of Theorem \ref{thm1}}

According to the results in Section \ref{Section3}, we can present the proof of Theorem \ref{thm1}.

\begin{proof}[Proof of Theorem \ref{thm1}]
Let $b_k(J)$ and $c_k(J)$ be the min-max values defined in subsection \ref{subsection3-3}.
 Note that $b_k(J)>M_{2}+1$ and $c_k(J)>M_{2}+1$ for all $k\geq k_0$. We first prove that
 there exists an increasing sequence
\[ k_{0}\leq k_{j_{1}}<k_{j_{2}}<\cdots<k_{j_{l}}\to+\infty\quad \mbox{as}~~l\to+\infty \]
satisfying $c_{k_j}(J)>b_{k_j}(J)+1$ for all $j\geq 1$.

Suppose  there exists a $m\geq k_0$ such that $c_k(J)\leq b_k(J)+1$ holds for all $k\geq m$. By the definitions of $c_k(J)$ and $\Lambda_k$, we can choose a map $\phi\in \Lambda_k$ such that
\begin{equation}\label{3-78}
  0<c_k(J)\leq \sup_{\phi(N^{+}_{k+1}\cap B_{R_{k+1}})\cap S_{a_{k+1}}}J(u)\leq b_k(J)+2.
\end{equation}
Then, we continuously extend $\phi$ as an odd map on $N_{k+1}\cap B_{R_{k+1}}$. It follows that $\phi\in \Gamma_{k+1}$. Denoting by
\[G=\{w\in N_{k+1}\cap B_{R_{k+1}}:\phi(w)\in S_{a_{k+1}}\}=N_{k+1}\cap B_{R_{k+1}}\cap \phi^{-1}(S_{a_{k+1}}),\]
 we can rewrite
\[ 0<c_k(J)\leq\sup_{\phi(N_{k+1}\cap B_{R_{k+1}})\cap S_{a_{k+1}}}J(u)=\sup_{w\in G}J(\phi(w)).\]
Observe that $\overline{G}\subset N_{k+1}\cap \overline{B_{R_{k+1}}}\cap \phi^{-1}(S_{a_{k+1}})$ is a compact subset in $N_{k+1}$, and $J(\phi(\cdot))$ attains its positive maximum on $\overline{G}$. Therefore,  $J(\phi(v))=\sup_{w\in G}J(\phi(w))$ for some $v\in G$ since $J(\phi(u))=J(u)<-1$ for all $u\in \overline{G}\setminus G\subset N_{k+1}\cap \partial B_{R_{k+1}}$. Then
\begin{itemize}
  \item If $v\in N_{k+1}^+$, by \eqref{3-78} and the definition of $b_{k+1}(J)$, we have
  \begin{equation}\label{3-79}
    b_{k+1}(J)\leq J(\phi(v))\leq b_k(J)+2;
  \end{equation}
  \item If $v\in -N_{k+1}^+$ (i.e. $-v\in N_{k+1}^+$), by Lemma \ref{lemma3-2} and \eqref{3-79} we have
  \begin{equation}\label{3-80}
  \begin{aligned}
 b_{k+1}(J)\leq J(\phi(v))=J(-\phi(-v))&\leq J(\phi(-v))+C(1+|J(\phi(-v))|^{\frac{\sigma+1}{\mu}})\\
  &\leq b_k(J)+2+C(1+|b_k(J)+2|^{\frac{\sigma+1}{\mu}}).
  \end{aligned}
  \end{equation}
\end{itemize}
Combining \eqref{3-79} and \eqref{3-80}, we have
\[b_{k+1}(J)\leq b_{k}(J)+Cb_{k}(J)^{\frac{\sigma +1}{\mu}}~~~~\forall k\geq m,\]
where $m\geq 3$ is some positive integer. By iteration, it follows that
\begin{equation}\label{3-81}
\begin{aligned}
 b_{m+l}(J)&\leq   b_{m}(J)\prod_{k=m}^{m+l-1}\left(1+C b_{k}(J)^{\frac{\sigma +1-\mu}{\mu}}\right)\\
 &= b_{m}(J)\exp\left(\sum_{k=m}^{m+l-1}\ln\left(1+C b_{k}(J)^{\frac{\sigma +1-\mu}{\mu}} \right) \right)\leq b_{m}(J)\exp\left(C\sum_{k=m}^{m+l-1} b_{k}(J)^{\frac{\sigma +1-\mu}{\mu}}\right),
\end{aligned}
\end{equation}
where we use the inequality $\ln(1+x)\leq x$ for $x\geq 0$ in the last step.

On the other hand, Lemma \ref{lemma3-15} and assumption $(L)$ in Theorem \ref{thm1} assert that
\[ b_{k}(J)\geq C \lambda_{k-1}^{\frac{2(2_{\tilde{\nu}}^{*}-p)}{(p-2)(2_{\tilde{\nu}}^{*}-2)}}\geq C\cdot k^{\frac{2}{\vartheta}\left(\frac{p}{p-2}-\frac{\tilde{\nu}}{2}\right)}(\ln k)^{-\kappa\left(\frac{p}{p-2}-\frac{\tilde{\nu}}{2}\right)}~~~~~\forall k\geq k_{1}. \]
Here $k_{1}\geq 3$ is some positive integer.
The condition $(A1)$ gives $\frac{2p}{\vartheta(p-2)}-\frac{\tilde{\nu}}{\vartheta}>\frac{\mu}{\mu-\sigma-1}$, i.e.
\[ \frac{2}{\vartheta}\left(\frac{p}{p-2}-\frac{\tilde{\nu}}{2}\right)\frac{\sigma+1-\mu}{\mu}<-1. \]
Hence, there exists an integer $k_{2}\geq k_{1}$ such that
\begin{equation}\label{3-82}
  \sum_{k=k_{2}}^{\infty} b_{k}(J)^{\frac{\sigma+1-\mu}{\mu}}\leq C\sum_{k=k_{2}}^{\infty} k^{\frac{2}{\vartheta}\left(\frac{p}{p-2}-\frac{\tilde{\nu}}{2}\right)\frac{\sigma+1-\mu}{\mu}}(\ln k)^{\kappa\left( \frac{p}{p-2}-\frac{\tilde{\nu}}{2}\right)\frac{\mu-\sigma-1}{\mu}} \leq C<+\infty.
\end{equation}
Therefore,  \eqref{3-81} and \eqref{3-82} derive that $b_{k}(J)\leq C$ for all $k\geq 3$ and some positive constant $C>0$, which contradicts  Lemma \ref{lemma3-15}. Hence for any $m\geq k_0$, there exists $k_m\geq m$ such that $c_{k_m}(J)> b_{k_m}(J)+1$. Using Proposition \ref{prop3-1}, we can obtain a sequence of sign-changing critical points $\{u_j\}_{j=1}^\infty$ of $J$ with the critical values $\{c_{k_j}(J)\}_{j=1}^\infty$ satisfying $c_{k_j}(J)>M_{2}+2>M_{0}$. From Lemma \ref{lemma3-4},  $\{u_j\}_{j=1}^\infty$ is also the sign-changing critical points of functional $E$ associated with the critical values $\{c_{k_j}(J)\}_{j=1}^\infty$.

\par
  Finally, we show that  $\{u_j\}_{j=1}^\infty$ is unbounded in $H_{X,0}^{1}(\Omega)$. According to \eqref{3-2}, \eqref{3-6}, \eqref{3-10} and $|\theta(u)|\leq 1$, it follows that
\begin{align*}
  c_{k_j}(J)=|J(u_j)| &\leq \frac{1}{2}\|u_j\|_{H_{X,0}^1(\Omega)}^2+\varepsilon\|u_j\|_{L^2(\Omega)}^2+ C(\varepsilon)\|u_j\|_{L^p(\Omega)}^p\\
   &\leq C\|u_j\|_{H_{X,0}^1(\Omega)}^2+C\|u_j\|_{H_{X,0}^1(\Omega)}^p,
\end{align*}
which implies $\lim_{j\to+\infty}\|u_j\|_{H_{X,0}^1(\Omega)}=+\infty$.
  \end{proof}

\subsection{Proof of Theorem \ref{thm2}}

According to the proof of Theorem \ref{thm1}, we find that the lower bound of min-max value $b_k(J)$ in Lemma \ref{lemma3-15} determines the condition $(A1)$. In this subsection, we present an additional lower bound for the min-max value $b_k(J)$, which determines the condition $(A2)$ in Theorem \ref{thm2}. Without loss of generality, we may assume that $\bar{a}<a^{1}$.

\begin{proposition}\label{prop4-14}
Let $b_k(J)$ be the min-max value defined in \eqref{3-58} above. For $k\geq 5$, we have
\begin{equation}
 b_k(J)\geq C(k-3)^{\frac{2p}{\tilde{\nu}(p-2)}}.
\end{equation}

 \end{proposition}
 \begin{proof}
The iteration \eqref{3-57} and assumption $\bar{a}<a^{1}$ ensure that $0<a_{k}<a^{1}$ for $k\geq 3$. For any fixed  $a_{k+2}>0$, we define the corresponding $d_{k,a_{k+2}}(I_{p})$ by \eqref{4-10} and Proposition \ref{prop4-6}. Applying Proposition \ref{prop4-12} and  Proposition \ref{prop4-13} to $a=a_{k+2}$, we have
\begin{equation}\label{4-64}
  m^{*}(u_{k},I_{p})\geq k-1~~~\mbox{for some}~~u_{k}\in K_{I_{p},d_{k,a_{k+2}}(I_{p})}\cap S_{a_{k+2}}.
\end{equation}
Consider the functional $V_{\varepsilon}(x):=-p(p-1)C_p|u_k(x)|^{p-2}-\varepsilon$ with $0\leq \varepsilon<1$.
Since $u_{k}\in H_{X,0}^{1}(\Omega)$, we deduce from Proposition \ref{prop2-3} that $V_{\varepsilon}\in L^{\frac{p_{1}}{2}}(\Omega)$ holds for some $p_{1}>\tilde{\nu}$ and all $0\leq \varepsilon<1$. Thus, Proposition \ref{prop2-5} implies that the Dirichlet eigenvalue problems of subelliptic Schr\"{o}dinger operators $-\triangle_{X}+V_{\varepsilon}$ are well-defined for all $0\leq \varepsilon<1$. The quadratic form of $-\triangle_{X}+V_{\varepsilon}$ is given by
\begin{equation}
\begin{aligned}\label{4-65}
  \mathcal{Q}_{\varepsilon}(h,v):&=\int_{\Omega}Xh\cdot Xv dx+\int_{\Omega}V_{\varepsilon}hv dx\\
  &=  \int_{\Omega}Xh\cdot Xv dx-\int_{\Omega}(p(p-1)C_p|u_k|^{p-2}+\varepsilon)hv dx \qquad\forall h,v\in H_{X,0}^{1}(\Omega).
\end{aligned}
\end{equation}
In particular, \eqref{4-7} and \eqref{4-65} yield that $ \mathcal{Q}_{0}(v,v)=(I''_{p}(u_{k})v,v)_{H_{X,0}^{1}(\Omega)}$ for all $v\in H_{X,0}^{1}(\Omega)$. Thus by Rayleigh-Ritz formula and Definition \ref{def4-3}, for any $0\leq\varepsilon<1$ we have
\begin{equation}
\begin{aligned}\label{4-66}
N(0,-\triangle_{X}+V_{\varepsilon})=\max\{\text{dim} V:\mathcal{Q}_{\varepsilon}(u,u)<0~~\forall u\in V\setminus\{0\}\}
\end{aligned}
\end{equation}
and
\begin{equation}
\label{4-67}
m^*(u_k, I_p)=\max\{\text{dim} V: \mathcal{Q}_{0}(u,u)\leq 0~~\forall u\in V\},
\end{equation}
where $V$ is a subspace of $H_{X,0}^{1}(\Omega)$.

Combining Proposition \ref{CLR} and \eqref{4-64}-\eqref{4-67}, we have for any $k\geq 3$,
\begin{equation}\label{4-68}
 C\int_{\Omega}\left(p(p-1)C_p|u_{k}|^{p-2}+\varepsilon\right)^{\frac{\tilde{\nu}}{2}}dx \geq N(0,-\triangle_{X}+V_{\varepsilon})
  \geq m^*(u_k, I_p) \geq k-1,
\end{equation}
where $C>0$ is a positive constant. Taking $\varepsilon \to 0$ in \eqref{4-68}, we get
\begin{equation}\label{4-69}
\int_{\Omega}|u_k|^{\frac{(p-2)\tilde{\nu}}{2}} dx \geq C\cdot (k-1).
\end{equation}
In addition, $I_p'(u_k)=0$ and \eqref{4-6} yields $pC_p\int_{\Omega} |u_k|^pdx=\int_{\Omega} |Xu_k|^2dx$. Hence \eqref{4-69} yields
\begin{equation}\label{4-70}
\begin{aligned}
d_{k,a_{k+2}}(I_p)= I_p(u_k)&=\frac{1}{2}\int_{\Omega}|Xu_k|^2dx-C_p\int_{\Omega}|u_k|^p dx\\
&=C_p\left(\frac{p}{2}-1\right)\int_{\Omega}|u_k|^p dx\\
&\geq C\left(\int_{\Omega}|u_k|^{\frac{(p-2)\tilde{\nu}}{2}} dx \right)^{\frac{2p}{\tilde{\nu}(p-2)}}\\
&\geq C (k-1)^{\frac{2p}{\tilde{\nu}(p-2)}}~~~~\mbox{for}~~k\geq 3.
\end{aligned}
\end{equation}

 Using \eqref{4-4} we have $J(u)\geq \frac{1}{2}I_{p}(u)$ for all $u\in H_{X,0}^{1}(\Omega)$, which implies
 \begin{equation}
  b_k(J)\geq \frac{1}{2}b_k(I_p)~~~~~\mbox{for}~~k\geq 3.
 \end{equation}
It follows from \eqref{4-11} and Proposition \ref{prop4-6} that
\[ b_{k+2}(I_{p})= b_{k+2,a_{k+2}}(I_{p})\geq d_{k,a_{k+2}}(I_{p})~~~~~\mbox{for}~~k\geq 1. \]
Consequently,
 \[  b_k(J)\geq C(k-3)^{\frac{2p}{\tilde{\nu}(p-2)}}\qquad \mbox{for}~~k\geq 5. \]
 \end{proof}

\begin{proof}[Proof of Theorem \ref{thm2}]
By Proposition \ref{prop4-14}, the proof of Theorem \ref{thm2} immediately follows from the proof of Theorem \ref{thm1}, with the replacement of $\frac{2p}{\vartheta(p-2)}-\frac{\tilde{\nu}}{\vartheta}$ by $\frac{2p}{\tilde{\nu}(p-2)}$.
\end{proof}

\section{An example}
\label{Section6}
In this section, we give a simple example to illustrate that condition $(A1)$ includes some cases which are not included in condition $(A2)$. Actually, there are many such examples in degenerate cases.

\begin{example}
\label{ex5-1}
Let $X=(X_{1},X_{2})=(\partial_{x_{1}}, \partial_{x_{2}}+x_{1}^{2}\partial_{x_{3}})$
be the Martinet type vector fields defined on  $\mathbb{R}^3$.
The Martinet operator generated by $X$ is given by
\[ \triangle_{M}:=\partial_{x_{1}}^{2}+\left(\partial_{x_{2}}+x_{1}^{2}\partial_{x_{3}}\right)^{2}. \]
Assume $\Omega\subset\mathbb{R}^3$  is a smooth bounded open domain  containing the origin in $\mathbb{R}^{3}$. Clearly, $X$ satisfy the H\"{o}rmander's condition in $\mathbb{R}^3$ but fail to satisfy the M\'etivier condition on $\overline{\Omega}$. Besides, we have $\tilde{\nu}=5>3=n$.
Denoting by $\lambda_{k}$  the $k$-th Dirichlet eigenvalue of $-\triangle_{X}$ on $\Omega$, it follows from \cite{Chen-Chen-Li2022} that
\begin{equation}\label{6-1}
  \lambda_{k}\geq  C\left(\frac{k}{\ln k} \right)^{\frac{2}{\tilde{\nu}-1}}=C\left(\frac{k}{\ln k} \right)^{\frac{1}{2}}
\end{equation}
holds for sufficient large $k$. This means $\vartheta=4$ and $\kappa=\frac{1}{2}$ in the condition (L) of Theorem \ref{thm1}.

Now, we consider the problem \eqref{problem1-1} with $f(x,u)=u|u|^{p-2}$ such that $f(x,u)$ satisfies assumptions (f1)-(f5) with $p=\mu>2$. Additionally, we set
\begin{equation}
g(x,u)=\left\{
      \begin{array}{cc}
      u|u|,\hfill & |u|<1, \\[2mm]
      u|u|^{\sigma-1},\hfill           & |u|\geq 1,
      \end{array}
 \right.
\end{equation}
where $\sigma=\frac{9}{10}p-1$ such that $0\leq \sigma<\mu-1=p-1$. It derives that $g(x,u)$ satisfies assumptions (g1)-(g3). In this case,  $(A1)$ is equivalent to
\begin{equation}\label{6-2}
  \frac{p}{2(p-2)}-\frac{5}{4}>\frac{\mu}{\mu-\sigma-1}=10,
\end{equation}
while $(A2)$ is equivalent to
\begin{equation}\label{6-3}
\frac{2p}{5(p-2)}>\frac{\mu}{\mu-\sigma-1}=10.
\end{equation}
If we take $p=\frac{1}{2}(\frac{25}{12}+\frac{90}{43})=\frac{2155}{1032}\in (2,\frac{2\tilde{\nu}}{\tilde{\nu}-2}) $, then
\[ \frac{p}{2(p-2)}-\frac{5}{4}=\frac{3855}{364}>\frac{\mu}{\mu-\sigma-1}=10,\]
but
\[ \frac{2p}{5(p-2)}=\frac{862}{91}<\frac{\mu}{\mu-\sigma-1}=10.\]
This means the triple parameters $(p,\mu,\sigma)=\left(p,p,\frac{9}{10}p-1\right)$ satisfies $(A1)$ but fails to $(A2)$.
\end{example}

\section*{Acknowledgements}
Hua Chen is supported by National Natural Science Foundation of China (Grant Nos. 12131017, 12221001) and National Key R\&D Program of China (no. 2022YFA1005602). Hong-Ge Chen is supported by National Natural Science Foundation of China (Grant No. 12201607) and Knowledge Innovation Program of Wuhan-Shuguang Project (Grant No. 2023010201020286). Jin-Ning Li is supported by China National Postdoctoral Program for Innovative Talents (Grant No.  BX20230270).

\end{document}